\documentclass[12pt]{article}
\usepackage{aurical}
\usepackage[T1]{fontenc}
\usepackage{amssymb,amsmath,bm}
\usepackage{amsthm}
\usepackage{setspace,mathrsfs,nicefrac}
\usepackage{bbm,url,color,wasysym}
\usepackage[margin=1in]{geometry}
\usepackage{paralist}\setdefaultitem{$-$}{\tiny $\bullet$}{\tiny$\circ$}{$\star}
\usepackage{hyperref}
\usepackage[title]{appendix}
\allowdisplaybreaks

\renewcommand{\geq}{\geqslant}
\renewcommand{\leq}{\leqslant}

\author{{\bf Robert C. Dalang$^*$}  and {\bf Fei Pu\footnote{Research
partially supported by the Swiss National Foundation for Scientific
Research.}}\\
\\
\it\small \'Ecole Polytechnique F\'ed\'erale de Lausanne \\
}

\title{\bf\Large{Optimal lower bounds on hitting probabilities for non-linear systems of stochastic fractional heat equations}}
\date{}

\newtheorem{stat}{Statement}[section]
\newtheorem{prop}[stat]{Proposition}
\newtheorem{corollary}[stat]{Corollary}
\newtheorem{theorem}[stat]{Theorem}

\newtheorem{lemma}[stat]{Lemma}
\theoremstyle{definition}

\newtheorem{remark}[stat]{Remark}

\numberwithin{equation}{section}

\begin{document}
\maketitle
\begin{abstract}
 We consider a system of $d$ non-linear stochastic fractional heat equations in spatial dimension $1$ driven by multiplicative $d$-dimensional space-time white noise. We establish a sharp Gaussian-type upper bound on the two-point probability density function of $(u(s, y), u (t, x))$. From this result, we deduce optimal lower bounds on hitting probabilities of the process $\{u(t, x): (t, x) \in [0, \infty[ \times \mathbb{R}\}$ in the non-Gaussian case, in terms of Newtonian capacity, which is as sharp as that in the Gaussian case. This also improves the result in Dalang, Khoshnevisan and Nualart [\textit{Probab. Theory Related Fields} \textbf{144} (2009) 371--424] for systems of classical stochastic heat equations. We also establish upper bounds on hitting probabilities of the solution in terms of Hausdorff measure.
\end{abstract}
\vskip.5in 



\noindent{\it Abbreviated title: Bounds on hitting probabilities}

\section{Introduction}\label{section1}

We consider a system of non-linear stochastic fractional heat equations with vanishing initial conditions on the whole space $\mathbb{R}$, that is,
\begin{align} \label{eq2017-11-17-1}
\frac{\partial u_i}{\partial t}(t, x) &= {}_{x}D^{\alpha}u_i(t, x) +\sum_{j = 1}^d\sigma_{ij}(u(t, x))\dot{W}^j(t, x) + b_i(u(t, x)),
\end{align}
for $1 \leq i \leq d, t \in [0, T], x \in \mathbb{R}$, where $u := (u_1, \ldots, u_d)$, with initial conditions $u(0, x) = 0$ for all $x \in \mathbb{R}$.  Here, $\dot{W} := (\dot{W}^1, \ldots, \dot{W}^d)$ is a vector of $d$ independent space-time white noises on $[0, T] \times \mathbb{R}$
 defined on a probability space $(\Omega, \mathscr{F}, \mbox{P})$. The functions $b_i$, $\sigma_{ij} : \mathbb{R}^d \rightarrow \mathbb{R}$ are globally Lipschitz continuous for all $1 \leq i, j \leq d$. We set $b = (b_i)$, $\sigma = (\sigma_{ij})$. The fractional differential operator $D^{\alpha}$ ($1 < \alpha \leq 2$) is given by
 \begin{align*}
 D^{\alpha}\varphi(x) = \mathscr{F}^{-1}\{-|\lambda|^{\alpha}\mathscr{F}\{\varphi(x); \lambda\}; x\},
 \end{align*}
 where $\mathscr{F}$ denotes the Fourier transform. The operator $D^{\alpha}$ coincides with the fractional power $\alpha/2$ of the Laplacian. When $\alpha = 2$, it is Laplacian itself. For $1 < \alpha < 2$, it can also be represented by
 \begin{align*}
 D^{\alpha}\varphi(x) = c_{\alpha}\int_{\mathbb{R}}\frac{\varphi(x + y) - \varphi(x) - y\varphi'(x)}{|y|^{1 + \alpha}}dy
 \end{align*}
 with certain positive constant $c_{\alpha}$ depending only on $\alpha$; see \cite{Deb07}, \cite{DeD05}, \cite{Kom84} and \cite{ChD15}. We refer to \cite{Kwa17} for additional equivalent definitions of $D^{\alpha}$.

Eq. \eqref{eq2017-11-17-1} is formal: a rigorous formulation, following Walsh \cite{Wal86}, is as follows. For $t \geq 0$, let $\mathscr{F}_t = \sigma\{W(s, x), s \in [0, t], x \in \mathbb{R}\} \vee \mathcal{N}$, where $\mathcal{N}$ is the $\sigma$-field generated by $\mbox{P}$-null sets. A {\em mild solution}  of (\ref{eq2017-11-17-1}) is a jointly measurable $\mathbb{R}^d$-valued process $u = \{u(t, x), t \geq 0, x \in \mathbb{R}\}$, adapted to the filtration $(\mathscr{F}_t)_{t \geq 0}$, such that for $i \in \{1, \ldots, d\}$,
  \begin{align}\label{eq2017-09-04-5}
  u_i(t, x) &= \int_0^t\int_{\mathbb{R}} G_{\alpha}(t - r, x - v) \sum_{j = 1}^d\sigma_{ij}(u(r, v))W^j(dr, dv) \nonumber \\
  & \quad  + \int_0^t\int_{\mathbb{R}} G_{\alpha}(t - r, x - v)b_i(u(r, v))drdv ,
 \end{align}
  where the stochastic integral is interpreted as in \cite{Wal86} and $G_{\alpha}(t, x)$ denotes the Green kernel for the (fractional) heat equation. If $\alpha = 2$, the Green kernel $G_2(t, x)$  (denoted by $G(t, x)$) for the heat equation without boundary is given by $G(t, x) = (4\pi t)^{-1/2}\exp(-x^2/(4t))$.
 The Green kernel for the fractional heat equation ($1 < \alpha < 2$) is given via Fourier transform:
 \begin{align*}
 G_{\alpha}(t, x) = \frac{1}{2\pi}\int_{\mathbb{R}}\exp(- i\lambda x - t|\lambda|^{\alpha})d\lambda.
 \end{align*}
We refer to \cite{BMS95,ChD15,DeD05,Wu11} for the properties of the Green kernel.
In fact, to make sense of the stochastic integral in (\ref{eq2017-09-04-5}),  the function $(r, v) \mapsto 1_{\{r < t\}}G_{\alpha}(t - r, x - v)$ must belongs to $L^2([0, T] \times \mathbb{R})$. This explains the requirement $1 < \alpha \leq 2$; see also \cite{ChD15, DeD05}.

The problems of existence, uniqueness and H\"{o}lder continuity of the solution to non-linear stochastic fractional heat equations have been studied by many authors; see, e.g., \cite{AzM07, BEM10, ChD15, DeD05} and the references therein. Adapting these results to the case $d \geq 1$, one can show that there exists a unique process $u = \{u(t, x), t \geq 0, x \in \mathbb{R}\}$ that is a mild solution of (\ref{eq2017-11-17-1}), such that for any $T > 0$ and $p \geq 1$,
\begin{align}\label{eq2017-09-06-1}
\sup_{(t, x) \in [0, T] \times \mathbb{R}}\mbox{E}\left[|u_i(t, x)|^p\right] < \infty, \quad i \in \{1, \ldots, d\}.
\end{align}
Moreover, the following estimate holds for the moments of increments of the solution (see \cite[Theorem 3.1]{AzM07}): for all $s, t \in [0, T], x, y \in \mathbb{R}$ and $p > 1$,
\begin{align}\label{eq2017-08-18-1}
\mbox{E}[\|u(t, x) - u(s, y)\|^p] \leq C_{T, p} (\Delta_{\alpha}((t, x); (s, y)))^{p},
\end{align}
where $\Delta_{\alpha}$ is the fractional parabolic metric defined by
 \begin{align*}
    \Delta_{\alpha}((t, x); (s, y)) := |t - s|^{\frac{\alpha - 1}{2\alpha}} + |x - y|^{\frac{\alpha - 1}{2}}, \quad \mbox{for}\, \,\, t, \, s \in [0, T] \,\,\, \mbox{and} \,\,\, x, \, y \in \mathbb{R}.
\end{align*}

 We will also establish an analogous  estimate on the H\"{o}lder continuity of the Malliavin derivative of the solution; see Proposition \ref{prop2017-09-13-1}. We denote by $K_m = [0, m] \times [-m, m]$ and $\beta_p = 1 - \frac{2(\alpha + 1)}{p(\alpha - 1)}$ with $p > \frac{2(\alpha + 1)}{\alpha - 1}$. By Kolmogorov's continuity theorem (see \cite[Theorem 1.4.1, p.31]{Kun90} and \cite[Proposition 4.2]{ChD14}), the solution $u$ has a continuous modification which we continue to denote by $u$ that satisfies, for all integers $m$ and $0 \leq \beta < \beta_p$,
\begin{align}\label{eq2017-08-21-1}
\mbox{E}\left[\left(\sup_{\mbox{\tiny$\begin{array}{c}
 (t, x), (s, y) \in K_m \\
 (t, x) \neq (s, y)
 \end{array}$}}\frac{\|u(t, x) - u(s, y)\|}{[\Delta_{\alpha}((t, x); (s, y))]^{\beta}}\right)^p\right] < \infty.
\end{align}

Let $I \subset \, ]0, T]$ and $J \subset \mathbb{R}$ be two fixed compact intervals with positive length. We choose $m$ sufficiently large so that $I \times J \subset K_m$. We are interested in the hitting probability $\mbox{P}\{u(I \times J) \cap A \neq \emptyset \}$, where $u(I \times J)$ denotes the range of $I \times J$ under the random map $(t, x) \mapsto u(t, x)$. For systems of stochastic heat equations on the spatial interval $[0, 1]$, in the case where the noise is additive, i.e., $\sigma \equiv \mbox{Id}$, $b \equiv 0$, Dalang, Khoshnevisan and Nualart \cite{DKN07} have established upper and lower bounds on hitting probabilities for the Gaussian solution. They show that there exists $c > 0$ depending on $M, I, J$ with $M > 0$, such that, for all Borel sets $A \subseteq [-M, M]^d$,
\begin{align}\label{eq2017-10-30-100}
c^{-1} \mbox{Cap}_{d - 6}(A) \leq \mbox{P}\{u(I \times J) \cap A \neq \emptyset \} \leq c \, \mathscr{H}_{d - 6}(A),
\end{align}
where $\mbox{Cap}_{\beta}$ denotes the capacity with respect to the Newtonian $\beta$-kernel and $\mathscr{H}_{\beta}$ denotes the $\beta$-dimensional Hausdorff measure (see \eqref{eq2018-04-30-4}, \eqref{eq2018-04-30-5} for definitions). If the noise is multiplicative, i.e., $\sigma$ and $b$ are not constants (but are sufficiently regular), then using techniques of Malliavin calculus, Dalang, Khoshnevisan and Nualart \cite{DKN09} have obtained upper and lower bounds on hitting probabilities for the non-Gaussian solution. Indeed, they prove that there exists $c > 0$ depending on $M, I, J, \eta$ with $M > 0, \eta > 0$, such that, for all Borel sets $A \subseteq [-M, M]^d$,
\begin{align}\label{eq2017-10-30-1000}
c^{-1} \mbox{Cap}_{d + \eta - 6 }(A) \leq \mbox{P}\{u(I \times J) \cap A \neq \emptyset \} \leq c \, \mathscr{H}_{d - \eta - 6 }(A).
\end{align}
Furthermore, these results have been extended to higher spatial dimensions driven by spatially homogeneous noise in \cite{DKN13}. This type of question has also been studied for systems of stochastic wave equations in \cite{DaN04}, and in higher spatial dimensions \cite{DaS10} and \cite{DaS15}, and for systems of stochastic Poisson equations \cite{SaV16}.

The objective of this paper is to remove the $\eta$ in the dimension of capacity in \eqref{eq2017-10-30-1000} so that the lower bound on hitting probabilities is consistent with the Gaussian case in \eqref{eq2017-10-30-100}, and to generalize these results to systems of stochastic fractional heat equations.

Consider the following three hypotheses on the coefficients of the system (\ref{eq2017-11-17-1}):
\begin{enumerate}
  \item [\textbf{P1}] The functions $\sigma_{ij}$ and $b_i$ are bounded and infinitely differentiable with bounded partial derivatives of all orders, for $1 \leq i, j \leq d$.
  \item [\textbf{P1'}] The functions $\sigma_{ij}$ and $b_i$ are infinitely differentiable with bounded partial derivatives of all positive orders, and the $\sigma_{ij}$ are bounded, for $1 \leq i, j \leq d$.
  \item [\textbf{P2}] The matrix $\sigma$ is uniformly elliptic, that is, $\|\sigma(x)\xi\|^2 \geq \rho^2 > 0$ for some $\rho > 0$, for all $x \in \mathbb{R}^d, \|\xi\| = 1$.
\end{enumerate}
Notice that hypothesis \textbf{P1'} is weaker than hypothesis \textbf{P1}, since in \textbf{P1'}, the functions $b_i$, $i = 1, \ldots, d$ are not assumed to be bounded.

Adapting the results from \cite{BEM10} to the case $d \geq 1$, the $\mathbb{R}^d$-valued random vector $u(t, x) = (u_1(t, x), \ldots, u_d(t, x))$ admits a smooth probability density function, denoted by $p_{t, x}(\cdot)$ for all $(t, x) \in [0, T] \times \mathbb{R}$: see our Proposition \ref{prop2017-09-07-1}. For $(s, y) \neq (t, x)$, let $p_{s, y; t, x}(\cdot, \cdot)$ denote the joint density function of the $\mathbb{R}^{2d}$-valued random vector
\begin{align*}
(u(s, y), u(t, x)) = (u_1(s, y), \ldots, u_d(s, y), u_1(t, x), \ldots, u_d(t, x))
\end{align*}
(the existence of $p_{s, y; t, x}(\cdot, \cdot)$ is a consequence of our Theorem \ref{th2017-08-24-1}, (\ref{eq2017-09-05-5}) and Proposition \ref{prop2017-09-19-3}).

\begin{theorem}\label{theorem2017-08-17-1}
Assume \textbf{P1} and \textbf{P2}. Fix $T > 0$ and let $I \subset \, ]0, T]$ and $J \subset \mathbb{R}$ be two fixed non-trivial compact intervals.
\begin{enumerate}
  \item [(a)] The density $p_{t, x}(z)$ is a $C^{\infty}$ function of $z$ and is uniformly bounded over $z \in \mathbb{R}^d$ and $(t, x) \in I \times J$.
    \item [(b)] There exists $c > 0$ such that for all $s, t \in I, x, y \in J$ with $(s, y) \neq (t, x)$  and $z_1, z_2 \in \mathbb{R}^d$,
\begin{align} \label{eq2017-11-20-1}
p_{s, y; t, x}(z_1, z_2) &\leq c(\Delta_{\alpha}((t, x); (s, y)) )^{-d}\exp\left(-\frac{\|z_1 - z_2\|^2}{c(\Delta_{\alpha}((t, x); (s, y)))^2}\right).
\end{align}
\end{enumerate}
\end{theorem}

\begin{remark}\label{remark2017-11-20-1}
\begin{enumerate}
 \item [(a)] Theorem \ref{theorem2017-08-17-1}(a) remains valid under a slightly weaker version of \textbf{P1}, in which the $b_i, \, \sigma_{ij}$ need not be bounded (but their derivatives of all positive orders are bounded).
 \item [(b)] With hypothesis \textbf{P1} replaced by the slightly weaker version \textbf{P1'} in Theorem \ref{theorem2017-08-17-1}, the estimate  \eqref{eq2017-11-20-1} in statement (b) is replaced by: There exists $c > 0$ such that for all $s, t \in I, x, y \in J$ with $(s, y) \neq (t, x)$, $z_1, z_2 \in \mathbb{R}^d$ and $p \geq 1$,
\begin{align}\label{eq2017-11-20-2}
p_{s, y; t, x}(z_1, z_2) &\leq c(\Delta_{\alpha}((t, x); (s, y)) )^{-d}\left[\frac{(\Delta_{\alpha}((t, x); (s, y)))^2}{\|z_1 - z_2\|^2} \wedge 1\right]^{p/(4d)}.
\end{align}
\end{enumerate}
\end{remark}

The right-hand side of \eqref{eq2017-11-20-2} is larger than the r.h.s. of \eqref{eq2017-11-20-1} (after adjusting the constant).
In fact, the boundedness of the functions $b_i, \, i = 1, \ldots, d$ in hypothesis \textbf{P1} is only used when we derive the exponential factor on the right-hand side of \eqref{eq2017-11-20-1} by applying Girsanov's theorem. However, under the hypothesis \textbf{P1'}, when $b_i$ is not bounded, Girsanov's theorem is no longer applicable. We establish \eqref{eq2017-11-20-2} in Section \ref{section2.5.2} and, following \cite{DKN13, DaS15}, show in Section \ref{section2.2.3} that this estimate is also sufficient for our purposes.

We prove the smoothness and uniform boundedness of the one-point density (Theorem \ref{theorem2017-08-17-1}(a)) in Section \ref{section2.3}. We present the Gaussian-type upper bound on the two-point density (Theorem \ref{theorem2017-08-17-1}(b)) in Section \ref{section2.5.2}.

We will also need the strict positivity of $p_{t, x}(\cdot)$.
\begin{theorem}\label{th2018-04-26-1}
Assume $\textbf{P1'}$ and $\textbf{P2}$. For all $(t, x) \in\ ]0, T] \times \mathbb{R}$ and $z \in \mathbb{R}^d$, the density $p_{t, x}(z)$ is strictly positive.
\end{theorem}

\begin{remark}\label{remark2017-11-21-1}
The results of Theorems \ref{theorem2017-08-17-1}, \ref{th2018-04-26-1} and Remark \ref{remark2017-11-20-1} (as well as Theorems \ref{theorem2017-08-18-1}, \ref{theorem2017-08-18-2} below) include the case $\alpha = 2$, that is, they apply to the solutions of the stochastic heat equations with Neumann or Dirichlet boundary conditions; see Remark \ref{remark2017-11-21-6}.
\end{remark}

The proof of strict positivity of the one-point density (Theorem \ref{th2018-04-26-1}) is parallel to that in \cite{Nua13}. We refer to \cite{Pu18} for details. We mention that Chen, Hu and Nualart \cite{CHN16} have recently studied the strict positivity of the density on the support of the law for the non-linear stochastic fractional heat equation without drift term and with measure-valued initial data and unbounded diffusion coefficient.

Our main contribution is to obtain the Gaussian-type upper bound in Theorem \ref{theorem2017-08-17-1}(b), which is an improvement over  \cite[Theorem 1.1(c)]{DKN09}. There, for the stochastic heat equation, the optimal Gaussian-type upper bound was shown to hold when $t = s$, while an extra term $\eta$ appeared in the exponent when $t \neq s$; see \cite[Theorem 1.1]{DKN09}. We manage to remove this $\eta$ in the Gaussian-type upper bound on the joint density in \cite[Theorem 1.1(c)]{DKN09}, so that this becomes the best possible upper bound, as in the Gaussian case. This requires a detailed analysis of the small eigenvalues of the Malliavin matrix $\gamma_Z$ of $Z := (u(s, y), u(t, x) - u(s, y))$; see Proposition \ref{prop2017-09-19-4}. We prove Proposition \ref{prop2017-09-19-4} by giving a better estimate on the Malliavin derivative of the solution; see Lemma \ref{lemma2017-09-14-1}, which, for a certain range of parameters,  is an improvement of Morien \cite[Lemma 4.2]{Mor99}; see also Lemma \ref{lemma2017-09-06-1}. This estimate is used in Lemma \ref{lemma2017-09-19-1} to obtain a bound on the integral terms in the Malliavin derivative of $u$ (compare with \cite[Lemma 6.11]{DKN09}), then in Proposition \ref{prop2017-09-19-4} to bound negative moments of the smallest eigenvalue of the Malliavin matrix (compare with \cite[Proposition 6.9]{DKN09}), and finally in Proposition \ref{prop2017-09-19-3} and Theorem \ref{th2017-09-22-1} to bound negative moments of the Malliavin matrix (compare with \cite[Proposition 6.6]{DKN09} and \cite[Theorem 6.3]{DKN09}). This improves the result of \cite[Theorem 1.1(c)]{DKN09}, and the method extends to systems of stochastic fractional heat equations \eqref{eq2017-11-17-1} for $1 < \alpha \leq 2$ with a unified proof.

Coming back to potential theory, let us introduce some notation, following \cite{Kho02}. For all Borel sets $F \subseteq \mathbb{R}^d$, we define $\mathscr{P}(F)$ to be the set of all probability measures with compact support contained in $F$. For all integers $k \geq 1$ and $\mu \in \mathscr{P}(\mathbb{R}^k)$, we let $I_{\beta}(\mu)$ denote the $\beta$-{\em dimensional energy} of $\mu$, that is,
\begin{align*}
    I_{\beta}(\mu) := \iint \mbox{K}_{\beta}(\|x - y\|)\mu(dx)\mu(dy),
\end{align*}
where $\|x\|$ denotes the Euclidian norm of $x \in \mathbb{R}^k$,
\begin{align} \label{eq2017-08-23-100}
    \mbox{K}_{\beta}(r) := \left\{\begin{array}{cc}
                             r^{-\beta} & \mbox{if} \, \,\,  \beta > 0, \\
                             \log_+(1/r)  &  \mbox{if} \, \, \, \beta = 0,\\
                             1  &   \mbox{if} \, \, \, \beta < 0,
                           \end{array} \right.
\end{align}
where $\log_+(x): = \log(x \vee e)$.

For all $\beta \in \mathbb{R}$, integers $k \geq 1$, and Borel sets $F \subseteq \mathbb{R}^k$, $\mbox{Cap}_{\beta}(F)$ denotes the $\beta$-{\em dimensional capacity} of $F$, that is,
\begin{align} \label{eq2018-04-30-4}
    \mbox{Cap}_{\beta}(F) := \left[\inf\limits_{\mu \in \mathscr{P}(F)}I_{\beta}(\mu)\right]^{-1},
\end{align}
where $1/\infty := 0$. Note that if $\beta < 0$, then $\mbox{Cap}_{\beta}(\cdot) \equiv 1$.

Given $\beta \geq 0$, the $\beta$-dimensional {\em Hausdorff measure} of $F$ is defined by
\begin{align}\label{eq2018-04-30-5}
\mathscr{H}_{\beta}(F) = \lim_{\epsilon \rightarrow 0^{+}} \inf\left\{\sum_{i = 1}^{\infty}(2r_i)^{\beta}: F \subseteq \bigcup_{i = 1}^{\infty}B(x_i, r_i), \sup_{i \geq 1}r_i \leq \epsilon\right\}.
\end{align}
When $\beta < 0$, we define $\mathscr{H}_{\beta}(F)$ to be infinite.

 Using Theorems \ref{theorem2017-08-17-1}, \ref{th2018-04-26-1}, Remark \ref{remark2017-11-20-1}
 together with results from Dalang, Khoshnevisan and Nualart \cite{DKN07}, we shall prove the following results for the hitting probabilities of the solution.
\begin{theorem} \label{theorem2017-08-18-1}
Assume \textbf{P1'} and \textbf{P2}. Fix $T > 0, M > 0$ and $\eta > 0$. Let $I \subset \, ]0, T]$ and $J \subset \mathbb{R}$ be two fixed non-trivial compact intervals.
\begin{enumerate}
  \item [(a)] There exists $c_1 > 0$ depending on $I, J$ and $M$, and $c_2 > 0$ depending on $I, J$ and $\eta$ such that for all compact sets $A \subseteq [-M, M]^d$,
   \begin{align*}
   c_1\, \mbox{Cap}_{d - \frac{2(\alpha + 1)}{\alpha - 1}}(A) \leq \mathrm{P}\{u(I \times J) \cap A \neq \emptyset\} \leq c_2 \, \mathscr{H}_{d - \frac{2(\alpha + 1)}{\alpha - 1} - \eta}(A).
  \end{align*}
  \item [(b)] For all $t \in \, ]0, T]$, there exists $c_1 > 0$ depending on $J$ and $M$, and $c_2 > 0$ depending on $J$ and $\eta$ such that for all compact sets $A \subseteq [-M, M]^d$,
   \begin{align*}
   c_1\, \mbox{Cap}_{d - \frac{2}{\alpha - 1}}(A) \leq \mathrm{P}\{u(\{t\} \times J) \cap A \neq \emptyset\} \leq c_2 \, \mathscr{H}_{d - \frac{2}{\alpha - 1} - \eta}(A).
  \end{align*}
  \item [(c)] For all $x \in \mathbb{R}$, there exists $c_1 > 0$ depending on $I$ and $M$, and $c_2 > 0$ depending on $I$ and $\eta$ such that for all compact sets $A \subseteq [-M, M]^d$,
   \begin{align*}
   c_1\, \mbox{Cap}_{d - \frac{2\alpha}{\alpha - 1}}(A) \leq \mathrm{P}\{u(I \times \{x\}) \cap A \neq \emptyset\} \leq c_2 \, \mathscr{H}_{d - \frac{2\alpha}{\alpha - 1} - \eta}(A).
  \end{align*}
\end{enumerate}
\end{theorem}
The optimal lower bounds for the hitting probabilities on the left-hand sides of Theorem \ref{theorem2017-08-18-1} are mainly the consequence of the sharp upper bound on the two-point density function in \eqref{eq2017-11-20-2} (or the sharp Gaussian-type upper bound \eqref{eq2017-11-20-1} under the slightly stronger condition \textbf{P1}). And for $\alpha = 2$, these results also extend to systems of classical stochastic heat equations on a bounded interval with Dirichlet or Neumann boundary conditions; see Remark \ref{remark2017-11-21-1}. The upper bounds on hitting probabilities on the right-hand sides of Theorem \ref{theorem2017-08-18-1} are an extension to $1 < \alpha \leq 2$ of the corresponding results of \cite[Theorem 1.2]{DKN09} for $\alpha = 2$.

If $\sigma \equiv \mbox{Id}$ and  $ b \equiv 0$, by \cite[Theorem 7.6]{Xia09}, the upper bounds in Theorem \ref{theorem2017-08-18-1} can be improved to the best result available for the Gaussian case.

\begin{theorem} \label{theorem2017-08-18-2}
 Denote by $v$ the solution of (\ref{eq2017-11-17-1}) with $\sigma \equiv \mbox{Id}$ and  $ b \equiv 0$. Fix $T > 0$. Let $I \subset \, ]0, T]$ and $J \subset \mathbb{R}$ be two fixed non-trivial compact intervals.
 \begin{itemize}
   \item [(a)] There exists $C > 0$ depending on $I$ and $J$ such that for all compact sets $A \subseteq \mathbb{R}^d$,
   \begin{align*}
    \mathrm{P}\{v(I \times J) \cap A \neq \emptyset\} \leq C\mathscr{H}_{d - \frac{2(\alpha + 1)}{\alpha - 1}}(A).
    \end{align*}
    \item [(b)] For all $t \in \, ]0, T]$,  there exists $C > 0$ depending on $J$ such that for all compact sets $A \subseteq \mathbb{R}^d$,
   \begin{align*}
    \mathrm{P}\{v(\{t\} \times J) \cap A \neq \emptyset\} \leq C\mathscr{H}_{d - \frac{2}{\alpha - 1}}(A).
    \end{align*}
    \item [(c)] For all $x \in \mathbb{R}$, there exists $C > 0$ depending on $I$ such that for all compact sets $A \subseteq \mathbb{R}^d$,
   \begin{align*}
    \mathrm{P}\{v(I \times \{x\}) \cap A \neq \emptyset\} \leq C\mathscr{H}_{d - \frac{2\alpha}{\alpha - 1}}(A).
    \end{align*}
 \end{itemize}
\end{theorem}

Theorems \ref{theorem2017-08-18-1} and \ref{theorem2017-08-18-2} will be  proved in Section \ref{section1111}
.

\section{Elements of Malliavin calculus}

In this section, we introduce, following Nualart \cite{Nua06} (see also \cite{San05}), some elements of Malliavin calculus. Let $W = \{W(h), h \in \mathscr{H}\}$ denote the isonormal Gaussian process (see \cite[Definition 1.1.1]{Nua06}) associated with space-time white noise, where $\mathscr{H}$ is the Hilbert space $L^2([0, T] \times \mathbb{R}, \mathbb{R}^d)$. Let $\mathscr{S}$ denote the class of smooth random variables of the form
\begin{align*}
G = g(W(h_1), \ldots, W(h_n)),
\end{align*}
where $n \geq 1$, $g \in \mathscr{C}_p^{\infty}(\mathbb{R}^n)$, the set of real-valued functions $g$ such that $g$ and all its partial derivatives have at most polynomial growth and $h_i \in \mathscr{H}$. Given $G \in \mathscr{S}$, its derivative is defined to be the $\mathbb{R}^d$-valued stochastic process $DG = (D_{t, x}G = (D^{(1)}_{t, x}G, \ldots, D^{(d)}_{t, x}G), (t, x) \in [0, T] \times \mathbb{R})$ given by
\begin{align*}
D_{t, x}G = \sum_{i = 1}^n \partial_i g (W(h_1), \ldots, W(h_n))h_i(t, x).
\end{align*}
More generally, we can define the derivative $D^kG$ of order $k$ of $G$ by setting
 \begin{align*}
D_{\alpha}^kG = \sum_{i_1, \ldots, i_k = 1}^n \frac{\partial}{\partial x_{i_1}}\cdots \frac{\partial}{\partial x_{i_k}}\, g(W(h_1), \ldots, W(h_n))h_{i_1}(\alpha_1)\otimes \cdots \otimes h_{i_k}(\alpha_k),
\end{align*}
where $\alpha = (\alpha_1, \ldots, \alpha_k)$, $\alpha_i = (t_i, x_i), 1 \leq i \leq k$ and the notation $\otimes$ denotes the tensor product of functions.

For $p, k \geq 1$, the space $\mathbb{D}^{k, p}$ is the closure of $\mathscr{S}$ with respect to the seminorm $\|\cdot\|_{k, p}$ defined by
\begin{align*}
\|G\|_{k, p}^p = \mbox{E}[|G|^p] + \sum_{j = 1}^k\mbox{E}\left[\|D^jG\|_{\mathscr{H}^{\otimes j}}^p\right],
\end{align*}
where
\begin{align*}
\|D^jG\|_{\mathscr{H}^{\otimes j}}^2 = \sum_{i_1, \ldots, i_j = 1}^d \int_0^Tdt_1\int_{\mathbb{R}}dx_1 \cdots \int_0^Tdt_j\int_{\mathbb{R}}dx_j\left(D^{(i_1)}_{t_1, x_1} \cdots D^{(i_j)}_{t_j, x_j}G\right)^2.
\end{align*}
We set $\mathbb{D}^{\infty} = \cap_{p \geq 1}\cap_{k \geq 1}\mathbb{D}^{k, p}$.

The derivative operator $D$ on $L^2(\Omega)$ has an adjoint, termed the Skorohod integral and denoted by $\delta$, which is an unbounded and closed operator on $L^2(\Omega, \mathscr{H})$; see \cite[Section 1.3]{Nua06}. Its domain, denoted by  $\mbox{Dom}\ \delta$,  is the set of elements $u \in L^2(\Omega, \mathscr{H})$ such that there exists a constant $c$ such that $|\mbox{E}[\langle DG, u \rangle_{\mathscr{H}}]| \leq c \|G\|_{0, 2}$, for any $G \in \mathbb{D}^{1, 2}$. If $u \in \mbox{Dom}\ \delta$, then $\delta(u)$ is the element of $L^2(\Omega)$ characterized by the following duality relation:
\begin{align*}
\mbox{E}[G\, \delta(u)] = \mbox{E}\left[\sum_{j = 1}^d \int_0^T\int_\mathbb{R} D_{t, x}^{(j)}G \, \, u_j(t, x)dtdx\right], \quad \mbox{for all }\ G \in \mathbb{D}^{1, 2}.
\end{align*}

A first application of Malliavin calculus is the following global criterion for existence and smoothness of densities of probability laws.

\begin{theorem}[{{\cite[Proposition 2.1.5]{Nua06} or \cite[Theorem 5.2]{San05}}}]\label{th2017-08-24-1}
Let $F = (F^1, \ldots, F^d)$ be an $\mathbb{R}^d$-valued random vector satisfying the following two conditions:
\begin{enumerate}
  \item [(i)] $F \in (\mathbb{D}^{\infty})^d$;
  \item [(ii)] The Malliavin matrix of $F$ defined by $\gamma_F = (\langle DF^i, DF^j \rangle_{\mathscr{H}})_{1 \leq i, j \leq d}$ is invertible a.s. and $(\det \gamma_F)^{-1} \in L^p(\Omega)$ for all $p \geq 1$.
\end{enumerate}
Then the probability law of $F$ has an infinitely differentiable density function.
\end{theorem}

A random vector $F$ that satisfies conditions (i) and (ii) of Theorem \ref{th2017-08-24-1} is said to be {\em nondegenerate}. For a nondegenerate random vector, the following integration by parts formula plays a key role.

\begin{prop}[{{\cite[Proposition 3.2.1]{Nua98} or \cite[Propostion 5.4]{San05}}}] \label{prop2017-08-24-1}
Let $F = (F^1, \ldots, F^d) \in (\mathbb{D}^{\infty})^d$ be a nondegenerate random vector, let $G \in \mathbb{D}^{\infty}$ and $g \in \mathscr{C}_p^{\infty}(\mathbb{R}^d)$. Fix $k \geq 1$. For any multi-index $\alpha = (\alpha_1, \ldots, \alpha_k) \in \{1, \ldots, d\}^k$, there is an element $H_{\alpha}(F, G) \in \mathbb{D}^{\infty}$ such that
\begin{align*}
\mbox{E}[(\partial_{\alpha}g(F)G)] = \mbox{E}[g(F)H_{\alpha}(F, G)].
\end{align*}
In fact, the random variables $H_{\alpha}(F, G)$ are recursively given by
\begin{align*}
H_{\alpha}(F, G) &= H_{(\alpha_k)}(F, H_{(\alpha_1, \ldots, \alpha_{k -1})}(F, G)) \quad \mbox{and} \quad
H_{(i)}(F, G) = \sum_{i = 1}^d \delta(G(\gamma_F^{-1})_{i, j}DF^j).
\end{align*}
\end{prop}

Proposition \ref{prop2017-08-24-1} with $G = 1$ and $\alpha = (1, \ldots, d)$ implies the following expression for the density of a nondegenerate random vector.
\begin{corollary}[{{\cite[Corollary 3.2.1]{Nua98}}}]\label{cor2017-08-24-1}
Let $F = (F^1, \ldots, F^d) \in (\mathbb{D}^{\infty})^d$ be a nondegenerate random vector and let $p_F(z)$ denote the density of $F$. Then for every subset $\sigma$ of the set of indices $\{1, \ldots, d\}$,
\begin{align*}
p_F(z) = (-1)^{d - |\sigma|}\mbox{E}[1_{\{F^i > z^i, i \in \sigma, F^i < z^i, i \not\in \sigma\}}H_{(1, \ldots, d)}(F, 1)],
\end{align*}
where $|\sigma|$ is the cardinality of $\sigma$, and, in agreement with Proposition \ref{prop2017-08-24-1},
\begin{align*}
H_{(1, \ldots, d)}(F, 1) = \delta((\gamma_F^{-1}DF)^d\delta((\gamma_F^{-1}DF)^{d - 1}\delta(\cdots \delta((\gamma_F^{-1}DF)^1) \cdots))).
\end{align*}
\end{corollary}

The next result gives a criterion for uniform boundedness of the density of a nondegenerate random vector.
\begin{prop}[{{\cite[Proposition 3.4]{DKN09}}}] \label{prop2017-08-24-2}
For all $p > 1$ and $l \geq 1$, fix $c_1(p) > 0$ and $c_2(l, p) \geq 0$. Let $F \in (\mathbb{D}^{\infty})^d$ be a nondegenerate random vector such that for all $p > 1$,
\begin{enumerate}
  \item [(a)] $\mbox{E}[(\det \gamma_F)^{-p}] \leq c_1(p)$;
  \item [(b)] $\mbox{E}\left[\|D^l(F^i)\|_{\mathscr{H}^{\otimes l}}^p\right] \leq c_2(l, p), \quad i = 1, \ldots, d$, \mbox{for all} $l \geq 1$.
\end{enumerate}
Then the density of $F$ is uniformly bounded, and the bound does not depend on $F$ but only on the constants $c_1(p)$ and $c_2(l, p)$.
\end{prop}

\section{Existence, smoothness and uniform boundedness of the one-point density} \label{section2.3}

In \cite{BEM10}, the Malliavin differentiability and smoothness of the density of the solution to fractional SPDEs driven by spatially correlated noise was established when $d = 1$. These can also be applied to SPDEs driven by space-time white noise and the extension to $d > 1$ under \textbf{P1'} and \textbf{P2} can easily be done by working coordinate by coordinate. In particular, for any $(t, x) \in [0, T] \times \mathbb{R}$, $i, k \in \{1, \ldots, d\}$, the derivative of $u_i(t, x)$ satisfies the system of equations
\begin{align}\label{eq2017-09-05-3}
D_{r, v}^{(k)}(u_i(t, x)) = G_{\alpha}(t - r, x - v)\sigma_{ik}(u(r, v)) + a_i(k, r, v, t, x),
\end{align}
where
\begin{align}\label{eq2017-09-05-4001122}
a_i(k, r, v, t, x) &= \sum_{j = 1}^d\int_{r}^t\int_{\mathbb{R}}G_{\alpha}(t - \theta, x - \eta)D_{r, v}^{(k)}(\sigma_{ij}(u(\theta, \eta)))W^j(d\theta, d\eta) \nonumber \\
& \quad + \int_{r}^t\int_{\mathbb{R}}G_{\alpha}(t - \theta, x - \eta)D_{r, v}^{(k)}(b_i(u(\theta, \eta)))d\theta d\eta,
\end{align}
if $r < t$ and $D_{r, v}^{(k)}(u_i(t, x)) = 0$ when $r > t$. By iterating the calculation which leads to (\ref{eq2017-09-05-3}), we see that $D^mu_i(t, x)$ also satisfies the system of stochastic partial differential equations which are analogous to the equations in Proposition 4.1 of \cite{DKN07}; see also \cite[(6.29)]{NuQ07}. Moreover, for any $p > 1, m \geq 1$ and $i \in \{1, \ldots, d\}$, the order $m$ derivatives satisfies
\begin{align}\label{eq2017-09-05-4}
\sup_{(t, x) \in [0, T] \times \mathbb{R}}\mbox{E}\left[\left\|D^m(u_i(t, x))\right\|_{\mathscr{H}^{\otimes m}}^p\right] < \infty.
\end{align}
Furthermore, for all $(t, x) \in [0, T] \times \mathbb{R}$,
\begin{align}\label{eq2017-09-05-5}
u(t, x) \in (\mathbb{D}^{\infty})^d.
\end{align}

Our objective in this section is to prove Theorem \ref{theorem2017-08-17-1}(a) by using Proposition \ref{prop2017-08-24-2}. The next result proves property (a) in Proposition \ref{prop2017-08-24-2} when $F$ is replaced by $u(t, x)$.
\begin{prop}\label{prop2017-09-05-1}
Fix $T > 0$ and assume hypotheses \textbf{P1'} and \textbf{P2}. Then, for any $p \geq 1$,
\begin{align*}
\mbox{E}\left[(\det \gamma_{u(t, x)})^{-p}\right]
\end{align*}
is uniformly bounded over $(t, x)$ in any closed non-trivial rectangle $I \times J \subset \, ]0, T] \times \mathbb{R}$.
\end{prop}
\begin{proof}
The proof follows along the same lines as \cite[Proposition 4.2]{DKN09} by using \cite[Proposition 3.5]{DKN09}; see also \cite[Proposition 4.1]{DKN13}.  The main differences are the exponents appearing in the estimate. Let $(t, x) \in I \times J$ be fixed. We write
\begin{align*}
\det \gamma_{u(t, x)} \geq \left(\inf_{\xi \in \mathbb{R}^d: \|\xi\| = 1} \xi^T \gamma_{u(t, x)}\xi\right)^d.
\end{align*}
Let $\xi \in \mathbb{R}^d$ with $\|\xi\| = 1$ and fix $\epsilon \in \, ]0, 1[$. Using (\ref{eq2017-09-05-3}) and the inequality
\begin{align}\label{eq2017-09-06-2}
(a + b)^2 \geq \frac{2}{3}a^2 - 2b^2,
\end{align}
valid for all $a, b \in \mathbb{R}$, we see that
\begin{align*}
\xi^T \gamma_{u(t, x)}\xi &= \int_0^tdr\int_{\mathbb{R}}dv\left\|\sum_{i = 1}^dD_{r, v}(u_i(t, x))\xi_i\right\|^2\\
&\geq \int_{t(1 - \epsilon)}^tdr\int_{\mathbb{R}}dv\left\|\sum_{i = 1}^dD_{r, v}(u_i(t, x))\xi_i\right\|^2 \geq \frac{2}{3}I_1 - 2I_2,
\end{align*}
where
\begin{align*}
I_1 &= \int_{t(1 - \epsilon)}^tdr\int_{\mathbb{R}}dv\sum_{k = 1}^d\left(\sum_{i = 1}^d G_{\alpha}(t - r, x - v)\sigma_{ik}(u(r, v))\xi_i\right)^2,\\
I_2 &= \int_{t(1 - \epsilon)}^tdr\int_{\mathbb{R}}dv\sum_{k = 1}^d\left(\sum_{i = 1}^d a_i(k, r, v, t, x)\xi_i\right)^2,
\end{align*}
and $a_i(k, r, v, t, x)$ is defined in (\ref{eq2017-09-05-4001122}). By hypothesis \textbf{P2} and semi-group property of the Green kernel \cite[Lemma 4.1(iii)]{ChD15},
\begin{align}\label{eq2017-09-06-3}
I_1 &\geq c \int_{t(1 - \epsilon)}^t\int_{\mathbb{R}}G_{\alpha}^2(t - r, x - v)dvdr = c \int_{t(1 - \epsilon)}^tG_{\alpha}(2(t - r), 0)dr \nonumber \\
  &= \frac{c}{2} \int_{0}^{2t\epsilon}G_{\alpha}(r, 0)dr = c'(2t\epsilon)^{\frac{\alpha - 1}{\alpha}} \geq c''\epsilon^{\frac{\alpha - 1}{\alpha}},
\end{align}
where in the third equality we use the scaling property of the Green kernel \cite[Lemma 4.1(iv)]{ChD15}, and the constants $c$, $c'$ and $c''$ are uniform over $(t, x) \in I \times J$.

Next we apply the Cauchy-Schwarz inequality to find that, for any $q \geq 1$,
\begin{align*}
\mbox{E}\left[\sup_{\xi \in \mathbb{R}^d: \|\xi\| = 1}|I_2|^q\right] \leq c(\mbox{E}[|I_{21}|^q] + \mbox{E}[|I_{22}|^q]),
\end{align*}
where
\begin{align*}
I_{21} &= \sum_{i, j, k = 1}^d \int_{t(1 - \epsilon)}^tdr\int_{\mathbb{R}}dv\left(\int_r^t\int_{\mathbb{R}}G_{\alpha}(t - \theta, x - \eta)D_{r, v}^{(k)}(\sigma_{ij}(u(\theta, \eta)))W^j(d\theta, d\eta)\right)^2, \\
I_{22} &= \sum_{i, k = 1}^d \int_{t(1 - \epsilon)}^tdr\int_{\mathbb{R}}dv\left(\int_r^t\int_{\mathbb{R}}G_{\alpha}(t - \theta, x - \eta)D_{r, v}^{(k)}(b_i(u(\theta, \eta)))d\theta d\eta\right)^2.
\end{align*}
The term $I_{21}$ is bounded in the same way as $A_1$ in \cite[(4.5)]{DKN09}, with $G$ there replaced by our $G_{\alpha}$. Instead of using their Lemmas $7.6$, $7.3$ and $7.5$, we use Lemma \ref{lemma2017-09-05-5}, \eqref{eq2017-09-05-1} below and Lemma \ref{lemma2017-09-06-1}. This leads to
\begin{align*}
\mbox{E}[|I_{21}|^q] &\leq  C_{T}\epsilon^{2(\alpha - 1)q/\alpha},
\end{align*}
where the constant $C_T$ is uniform over $(t, x) \in I \times J$. For details, see \cite[Proof of Prop. 2.3.1]{Pu18}.

We next derive a similar bound for $I_{22}$. First, we use the Cauchy-Schwarz inequality with respect to the measure $G_{\alpha}(t - \theta, x - \eta)d\theta d\eta$ to see that
\begin{align*}
I_{22} &\leq \sum_{i, k = 1}^d \int_{t(1 - \epsilon)}^t(t - r)dr\int_{\mathbb{R}}dv\int_r^t\int_{\mathbb{R}}G_{\alpha}(t - \theta, x - \eta)\left(D_{r, v}^{(k)}(b_i(u(\theta, \eta)))\right)^2d\theta d\eta \\
&\leq \sum_{i, k = 1}^d t\epsilon\int_{t(1 - \epsilon)}^tdr\int_{\mathbb{R}}dv\int_r^t\int_{\mathbb{R}}G_{\alpha}(t - \theta, x - \eta)\left(D_{r, v}^{(k)}(b_i(u(\theta, \eta)))\right)^2d\theta d\eta.
\end{align*}
Since the partial derivatives of $b_i$ are bounded, by Fubini's theorem,
\begin{align*}
\mbox{E}[|I_{22}|^q] &\leq c\sum_{l, k = 1}^d (t\epsilon)^q \mbox{E}\Bigg[\bigg|\int_{t(1 - \epsilon)}^tdr\int_{\mathbb{R}}dv\int_r^t\int_{\mathbb{R}}G_{\alpha}(t - \theta, x - \eta)\left(D_{r, v}^{(k)}(u_l(\theta, \eta))\right)^2d\theta d\eta\bigg|^q\Bigg] \\
&= c\sum_{l, k = 1}^d (t\epsilon)^q \mbox{E}\Bigg[\bigg|\int_{t(1 - \epsilon)}^td\theta \int_{\mathbb{R}}d\eta \, G_{\alpha}(t - \theta, x - \eta)\int_{t(1 - \epsilon)}^{t\wedge \theta}dr\int_{\mathbb{R}}dv\left(D_{r, v}^{(k)}(u_l(\theta, \eta))\right)^2 \bigg|^q\Bigg].
\end{align*}
Applying H\"{o}lder's inequality with respect to the measure $G_{\alpha}(t - \theta, x - \eta)d\theta d\eta$,
\begin{align*}
\mbox{E}[|I_{22}|^q] &\leq c\sum_{l, k = 1}^d (t\epsilon)^q \bigg|\int_{t(1 - \epsilon)}^td\theta \int_{\mathbb{R}}d\eta\, G_{\alpha}(t - \theta, x - \eta)\bigg|^{q - 1}\\
& \quad  \times  \int_{t(1 - \epsilon)}^td\theta \int_{\mathbb{R}}d\eta\, G_{\alpha}(t - \theta, x - \eta)
\mbox{E}\Bigg[\bigg|\int_{t(1 - \epsilon)}^{t \wedge \theta}dr\int_{\mathbb{R}}dv
 \left(D_{r, v}^{(k)}(u_l(\theta, \eta))\right)^2 \bigg|^q\Bigg].
\end{align*}
Using Lemma \ref{lemma2017-09-06-1}, this yields $\mbox{E}[|I_{22}|^q] \leq C_T (t\epsilon)^q(t\epsilon)^q(t\epsilon)^{(\alpha - 1)q/\alpha} = C_T(t\epsilon)^{(3 - 1/\alpha)q}$.

Thus, we have proved that
\begin{align}\label{eq2017-09-07-2}
\mbox{E}\left[\sup_{\xi \in \mathbb{R}^d: \|\xi\| = 1}|I_2|^q\right] \leq C_T\epsilon^{2(\alpha - 1)q/\alpha},
\end{align}
where the constant $C_T$ is clearly uniform over $(t, x) \in I \times J$.

Finally, we apply \cite[Prop. 3.5]{DKN09} with $Z:= \inf_{\|\xi\| = 1}(\xi^T\gamma_{u(t, x)}\xi)$, $Y_{1, \epsilon} = Y_{2, \epsilon} = \sup_{\|\xi\| = 1}I_2$, $\epsilon_0 = 1$, $\alpha_1 = \alpha_2 = (\alpha - 1)/\alpha$ and $\beta_1 = \beta_2 = 2(\alpha - 1)/\alpha$, to get
$
\mbox{E}\left[(\det \gamma_{u(t, x)})^{-p}\right] \leq C_T,
$
where all the constants are independent of $(t, x) \in I \times J$.
\end{proof}

In \cite{BEM10}, the authors established the existence and smoothness of the density of the solution of one single stochastic fractional partial differential equation driven by spatially correlated noise. For a system of $d$ equations driven by space-time white noise, we have the following.

\begin{prop} \label{prop2017-09-07-1}
Assume \textbf{P1'} and  \textbf{P2}. Fix $T > 0$ and let $I$ and $J$ be compact intervals as in Theorem \ref{theorem2017-08-17-1}. Then for any $(t, x) \in \, ]0, T] \times \mathbb{R}$, $u(t, x)$ is a nondegenerate random vector and its density function is infinitely differentiable and uniformly bounded over $z \in \mathbb{R}^d$ and $(t, x) \in I \times J$.
\end{prop}
\begin{proof}
The conclusions follow from Proposition \ref{prop2017-09-05-1} and (\ref{eq2017-09-05-5}) together with Theorem \ref{th2017-08-24-1}, \eqref{eq2017-09-05-4} and Proposition \ref{prop2017-08-24-2}.
\end{proof}
\begin{proof}[Proof of Theorem \ref{theorem2017-08-17-1}(a)]
This is an immediate consequence of Proposition \ref{prop2017-09-07-1}.
\end{proof}

\section{Gaussian-type upper bound on the two-point density}\label{section2.5}

The aim of this section is to prove Theorem \ref{theorem2017-08-17-1}(b) and Remark \ref{remark2017-11-20-1}(b). We will follow the general approach in \cite[Section 6]{DKN09}; see also \cite[Section 5]{DKN13}.

\subsection{Technical lemmas and propositions}

In this subsection, we present several technical lemmas and propositions which will be used for the analysis of the Malliavin matrix.

\begin{lemma}[{{\cite[Proposition 4.4]{ChD15}}}]\label{lemma2017-09-13-1}
 For any $s, t \in [0, T], s \leq t$, and $x, y \in \mathbb{R}$, there exists a constant $C_T > 0$ such that
\begin{align*}
\int_0^T\int_{\mathbb{R}}(g_{\alpha}(r, v))^2drdv \leq C_T(|t - s|^{\frac{\alpha - 1}{\alpha}} + |x - y|^{\alpha - 1}),
\end{align*}
where
\begin{align*}
g_{\alpha}(r, v):= g_{t, x, s, y}^{\alpha}(r, v) = 1_{\{r < t\}}G_{\alpha}(t - r, x - v) - 1_{\{r < s\}}G_{\alpha}(s - r, y - v).
\end{align*}
\end{lemma}
The following identity, which follows from a simple calculation by using the semigroup property and scaling property of Green kernel \cite[Lemma 4.1(iii), (iv)]{ChD15}, will be used several times later on:
\begin{equation}\label{eq2017-09-05-1}
\int_a^b\int_{\mathbb{R}}G_{\alpha}^2(t - r, x - v)dvdr = c_{\alpha}\left((t - a)^{\frac{\alpha - 1}{\alpha}} - (t - b)^{\frac{\alpha - 1}{\alpha}}\right), \quad a \leq b \leq t,
\end{equation}
where $c_{\alpha}$ is a positive constant depending on $\alpha$.

We next give an estimate on the $L^p$-modulus of continuity of the derivative of the increment, analogous to \cite[Proposition 6.2]{DKN09}, which is comparable to (\ref{eq2017-08-18-1}).

\begin{prop}\label{prop2017-09-13-1}
For any $p \geq 2, m \geq 1$, there exists a constant $C_{p, T}$ such that for all $s, t \in [0, T], s \leq t, x, y \in \mathbb{R}$,
\begin{align}\label{eq2017-09-13-1}
\mbox{E}\left[\left\|D^m(u_i(t, x) - u_i(s, y))\right\|^p_{\mathscr{H}^{\otimes m}}\right] \leq C_{p, T}(|t - s|^{\frac{\alpha - 1}{\alpha}} + |x - y|^{\alpha - 1})^{p/2}, \quad i = 1, \ldots, d.
\end{align}
\end{prop}
\begin{proof}
The proof is slightly different from that of \cite[Proposition 6.2]{DKN09} since the estimate for $I_3$ in \cite[Proposition 6.2]{DKN09} requires the Cauchy-Schwartz inequality, which is not applicable in our situation because the Lebesgue measure of $\mathbb{R}$ is infinite.

Assume $m = 1$. Using (\ref{eq2017-09-05-3}), we see that, for any $p \geq 2$,
\begin{align}\label{eq2018-05-29-1}
\mbox{E}\left[\left\|D(u_i(t, x) - u_i(s, y))\right\|^p_{\mathscr{H}}\right] \leq c\left(\mbox{E}\left[|I_1|^{p/2}\right] + \mbox{E}\left[|I_2|^{p/2}\right] + \mbox{E}\left[|I_3|^{p/2}\right] + \mbox{E}\left[|I_4|^{p/2}\right]\right),
\end{align}
where
\begin{align*}
I_1 &= \sum_{k = 1}^d\int_0^T dr\int_{\mathbb{R}}dv\left(g_{\alpha}(r, v)\sigma_{ik}(u(r, v))\right)^2, \\
I_2 &= \sum_{j, k = 1}^d\int_0^T dr\int_{\mathbb{R}}dv\left(\int_0^T\int_{\mathbb{R}}g_{\alpha}(\theta, \eta)D_{r, v}^{(k)}(\sigma_{ij}(u(\theta, \eta)))W^j(d\theta, d\eta)\right)^2, \\
I_3 &= \sum_{k = 1}^d\int_0^T dr\int_{\mathbb{R}}dv\left(\int_0^{t - s}\int_{\mathbb{R}}G_{\alpha}(t - \theta, x - \eta)D_{r, v}^{(k)}(b_i(u(\theta, \eta)))d\theta d\eta\right)^2,\\
I_4 &= \sum_{k = 1}^d\int_0^T dr\int_{\mathbb{R}}dv\Big(\int_0^{s}\int_{\mathbb{R}}G_{\alpha}(s - \theta, y - \eta)\\
& \quad \quad \quad \quad \quad \quad \quad \quad \quad \quad \times D_{r, v}^{(k)}(b_i(u(t - s + \theta, x - y + \eta)) - b_i(u(\theta, \eta)))d\theta d\eta\Big)^2.
\end{align*}
By hypothesis \textbf{P1'} and Lemma \ref{lemma2017-09-13-1},
\begin{align}\label{eq2018-05-29-2}
\mbox{E}\left[|I_1|^{p/2}\right] \leq C_{p,T}(|t - s|^{\frac{\alpha - 1}{\alpha}} + |x - y|^{\alpha - 1})^{p/2}.
\end{align}
For the term $I_2$, we proceed as in \cite[Proof of Prop. 6.2]{DKN09}, using Lemma \ref{lemma2017-09-05-5},  \eqref{eq2017-09-05-4} and Lemma \ref{lemma2017-09-13-1} instead of their Lemma 7.6, (4.1) and Lemma 6.1, and we obtain
\begin{align}\label{eq2018-05-29-3}
\mbox{E}\left[|I_2|^{p/2}\right] \leq C_{p,T}(|t - s|^{\frac{\alpha - 1}{\alpha}} + |x - y|^{\alpha - 1})^{p/2}.
\end{align}
To estimate $I_3$, denoting $\Theta_{k, l} := D_{r, v}^{(k)}(u_l(\theta, \eta))$, we use H\"{o}lder's inequality with respect to the measure $G_{\alpha}(t - \theta, x - \eta)d\theta d\eta$ twice to get that
\begin{align}\label{eq2018-05-29-4}
\mbox{E}\left[|I_3|^{p/2}\right]
 &\leq C_{p, T}\sum_{k, l = 1}^d(t - s)^{p/2}\mbox{E}\bigg[\Big(\int_0^{t - s}d\theta\int_{\mathbb{R}}d\eta \, G_{\alpha}(t - \theta, x - \eta)\int_0^Tdr\int_{\mathbb{R}}dv \, \Theta_{k, l}^2\Big)^{p/2}\bigg] \nonumber\\
  &\leq C_{p, T}\sum_{k, l = 1}^d(t - s)^{p/2}\left(\int_0^{t - s}d\theta\int_{\mathbb{R}}d\eta \, G_{\alpha}(t - \theta, x - \eta)\right)^{\frac{p}{2} - 1}\nonumber \\
  & \quad  \times\int_0^{t - s}d\theta\int_{\mathbb{R}}d\eta \, G_{\alpha}(t - \theta, x - \eta)\sup_{(\theta, \eta) \in [0, T] \times \mathbb{R}}\mbox{E}\left[\left[\int_0^Tdr\int_{\mathbb{R}}dv \, \Theta_{k, l}^2\right]^{p/2}\right] \nonumber\\
   &\leq C_{p, T}(t - s)^{p},
\end{align}
where in the last inequality we use (\ref{eq2017-09-05-4}).
Using H\"{o}lder's inequality with respect to the measure $G_{\alpha}(t - \theta, x - \eta)d\theta d\eta$,
\begin{align*}
 I_4 &\leq c\sum_{k = 1}^d\int_0^T dr\int_{\mathbb{R}}dv\int_0^{s}d\theta\int_{\mathbb{R}}d\eta \, G_{\alpha}(s - \theta, y - \eta)\\
 &  \quad \quad \quad \quad \quad \quad  \quad \quad \quad \quad \times \left(D_{r, v}^{(k)}\left(b_i(u(t - s + \theta, x - y + \eta)) - b_i(u(\theta, \eta))\right)\right)^2.
 \end{align*}
We apply the chain rule to compute $D_{r, v}^{(k)}b_i(u(t - s + \theta, x - y + \eta)) - D_{r, v}^{(k)}b_i(u(\theta, \eta))$,  subtract and add the term  $\sum_{l = 1}^d\frac{\partial b_i}{\partial x_l}(u(t - s + \theta, x - y + \eta))D_{r, v}^{(k)}u_l(\theta, \eta)$. Then by hypothesis \textbf{P1'}, this is bounded above by
 \begin{align*}
 & c\sum_{k, l = 1}^d\int_0^T dr\int_{\mathbb{R}}dv\int_0^{s}d\theta \int_{\mathbb{R}}d\eta \, G_{\alpha}(s - \theta, y - \eta)\\
 & \quad \quad \quad  \quad  \quad \quad \quad  \quad \quad \quad \times   \left(D_{r, v}^{(k)}(u_l(t - s + \theta, x - y + \eta) - u_l(\theta, \eta))\right)^2 \\
  & \quad   + c\sum_{k, l = 1}^d\int_0^T dr\int_{\mathbb{R}}dv\int_0^{s}d\theta\int_{\mathbb{R}}d\eta \, G_{\alpha}(s - \theta, y - \eta) \\
 &  \quad \quad \quad  \quad  \quad \quad \quad  \quad \quad\quad \quad \times\left(u_l(t - s + \theta, x - y + \eta) - u_l(\theta, \eta)\right)^2\Theta_{k, l}^2 \\
 &:= I_{41} + I_{42}.
\end{align*}
Using H\"{o}lder's inequality  with respect to the measure $G_{\alpha}(t - \theta, x - \eta)d\theta d\eta$, we have
\begin{align}
\mbox{E}\left[|I_{42}|^{p/2}\right]&\leq c\sum_{k, l = 1}^d\left(\int_0^{s}d\theta\int_{\mathbb{R}}d\eta \, G_{\alpha}(s - \theta, y - \eta)\right)^{\frac{p}{2} - 1}\int_0^{s}d\theta\int_{\mathbb{R}}d\eta \, G_{\alpha}(s - \theta, y - \eta)\nonumber\\
& \quad \quad \quad \quad \times \mbox{E}\left[\left|u_l(t - s + \theta, x - y + \eta) - u_l(\theta, \eta)\right|^p
 \left(\int_0^Tdr\int_{\mathbb{R}}dv \, \Theta_{k, l}^2\right)^{p/2}\right].\nonumber
 \end{align}
 By  the Cauchy-Schwartz inequality, this is bounded above by
 \begin{align}\label{eq2018-05-29-5}
& cs^{p/2}\sum_{k, l = 1}^d\sup_{(\theta, \eta) \in [0, T] \times \mathbb{R}}\mbox{E}\left[\left(\int_0^Tdr\int_{\mathbb{R}}dv \, \Theta_{k, l}^2\right)^{p}\right]^{1/2}\nonumber\\
& \qquad \times\sup_{(\theta, \eta) \in [0, T] \times \mathbb{R}}\mbox{E}\left[\left|u_l(t - s + \theta, x - y + \eta) - u_l(\theta, \eta)\right|^{2p}\right]^{1/2}\nonumber\\
& \quad\leq C_{p, T}s^{p/2}(|t - s|^{\frac{\alpha - 1}{\alpha}} + |x - y|^{\alpha - 1})^{p/2}
\end{align}
where we use (\ref{eq2017-09-05-4}) and (\ref{eq2017-08-18-1}).

Denote
 \begin{align*}
 \varphi(h, z, \theta) := \sup_{\eta \in \mathbb{R}}\sum_{k, l = 1}^d\mbox{E}\Big[\Big(\int_0^T\int_{\mathbb{R}}\left(D_{r, v}^{(k)}(u_l(h + \theta, z + \eta) - u_l(\theta, \eta))\right)^2drdv\Big)^{\frac{p}{2}}\Big].
 \end{align*}
By H\"{o}lder's inequality,
\begin{align}\label{eq2018-05-29-6}
\mbox{E}\left[|I_{41}|^{p/2}\right]
&\leq c\sum_{k, l = 1}^d \left(\int_0^{s}\int_{\mathbb{R}}G_{\alpha}(s - \theta, y - \eta)d\theta d\eta\right)^{\frac{p}{2} - 1} \int_0^{s}d\theta\int_{\mathbb{R}}d\eta \, G_{\alpha}(s - \theta, y - \eta)\nonumber\\
& \quad \quad \quad \quad \times \mbox{E}\Big[\Big(\int_0^T\int_{\mathbb{R}}(D_{r, v}^{(k)}(u_l(t - s + \theta, x - y + \eta) - u_l(\theta, \eta)))^2drdv\Big)^{\frac{p}{2}}\Big] \nonumber\\
&\leq C_{p, T}\int_0^s \varphi(t - s, x - y, \theta)d\theta.
\end{align}
Denote $h = t - s$ and $z = x - y$. From \eqref{eq2018-05-29-1}--\eqref{eq2018-05-29-6}, we conclude that for all $h \geq 0$, $z \in \mathbb{R}$, $s \in [0, T]$, $y \in \mathbb{R}$ and $1 \leq i \leq d$,
\begin{align*}
\mbox{E}\left[\left\|D(u_i(h + s, z + y) - u_i(s, y))\right\|^p_{\mathscr{H}}\right]  \leq C_{p, T}(|h|^{\frac{\alpha - 1}{\alpha}} + |z|^{\alpha - 1})^{p/2} + C_{p, T}\int_0^s \varphi(h, z, \theta)d\theta.
\end{align*}
Taking the supremum over $y \in \mathbb{R}$ on the left-hand side of the above inequality, we obtain that  for all $h \geq 0$, $z \in \mathbb{R}$ and $s \in [0, T]$,
\begin{align*}
\varphi(h, z, s) \leq C_{p, T}(|h|^{\frac{\alpha - 1}{\alpha}} + |z|^{\alpha - 1})^{p/2} + C_{p, T}\int_0^s \varphi(h, z, \theta)d\theta.
\end{align*}
By Gronwall's lemma (see \cite[p.543]{ReY99}), we obtain that
\begin{align*}
\sup_{s \in [0, T]}\varphi(h, z, s) \leq C_{p, T}(|h|^{\frac{\alpha - 1}{\alpha}} + |z|^{\alpha - 1})^{p/2},
\end{align*}
which implies \eqref{eq2017-09-13-1} with $m = 1$.

The case $m > 1$ follows along the same lines by using (\ref{eq2017-09-05-4}) and the stochastic partial differential equations satisfied by the iterated derivatives (see \cite[Proposition 4.1]{DKN09}).
\end{proof}

The following lemma is another version of \cite[Lemma 6.11]{DKN09}.
\begin{lemma}\label{lemma2017-09-15-1}
Assume \textbf{P1'}. Fix $T > 0, q \geq 1$. There exists a constant $c = c(q, T) \in \, ]0, \infty[$ such that for every $0 < 2\epsilon \leq s \leq t \leq T$ and $x \in \mathbb{R}$,
\begin{align*}
\mbox{E}\left[\left(\sum_{k = 1}^d\int_{s - \epsilon}^sdr\int_{\mathbb{R}}dv\sum_{i = 1 }^da_i^2(k, r, v, t, x)\right)^q\right] \leq c(t - s +\epsilon)^{(\alpha - 1)q/\alpha}\epsilon^{(\alpha - 1)q/\alpha}.
\end{align*}
\end{lemma}
\begin{proof}
The proof follows the same lines as \cite[Lemma 6.11]{DKN09}. Define
\begin{align*}
A := \sum_{k = 1}^d\int_{s - \epsilon}^sdr\int_{\mathbb{R}}dv\sum_{i = 1 }^da_i^2(k, r, v, t, x).
\end{align*}
From (\ref{eq2017-09-05-4001122}), we write
\begin{align*}
\mbox{E}\left[|A|^q\right] \leq c\left(\mbox{E}\left[|A_1|^q\right] + \mbox{E}\left[|A_2|^q\right]\right),
\end{align*}
where
\begin{align}
A_1 &:= \sum_{i, j, k = 1}^d\int_{s - \epsilon}^sdr\int_{\mathbb{R}}dv\left|\int_r^t\int_{\mathbb{R}}G_{\alpha}(t - \theta, x - \eta)D_{r, v}^{(k)}(\sigma_{ij}(u(\theta, \eta)))W^j(d\theta, d\eta)\right|^2, \label{eq2018-04-30-2} \\
A_2 &:= \sum_{i, k = 1}^d\int_{s - \epsilon}^sdr\int_{\mathbb{R}}dv\left|\int_r^t\int_{\mathbb{R}}G_{\alpha}(t - \theta, x - \eta)D_{r, v}^{(k)}(b_{i}(u(\theta, \eta)))d\theta d\eta\right|^2.
\end{align}

We bound the $q$-th moment of $A_1$ and $A_2$ separately.
As regards $A_1$, we follow the calculation in \cite[p.416-417]{DKN09}, with their $G$ replaced by our $G_{\alpha}$, and we use \eqref{eq2017-09-05-1} instead of their Lemma 7.3 and our Lemma \ref{lemma2017-09-06-1} instead of their Lemma 7.5. This replaces their exponent $\frac{1}{2}$ with $\frac{\alpha - 1}{\alpha}$, and we obtain
\begin{align}\label{eq2017-09-15-10}
\mbox{E}\left[|A_{1}|^q\right] &\leq c(t - s + \epsilon)^{\frac{\alpha - 1}{\alpha}q}\epsilon^{\frac{\alpha - 1}{\alpha}q}.
\end{align}

Next we derive a similar bound for $A_2$.  By the Cauchy-Schwartz inequality with respect to the measure $G_{\alpha}(t - \theta, x - \eta)d\theta d\eta$,
\begin{align*}
A_2 &\leq \sum_{i, k = 1}^d\int_{s - \epsilon}^sdr\int_{\mathbb{R}}dv \, (t - r)\int_r^t\int_{\mathbb{R}}G_{\alpha}(t - \theta, x - \eta)\left(D_{r, v}^{(k)}(b_{i}(u(\theta, \eta)))\right)^2d\theta d\eta \\
 & \leq \sum_{i, k = 1}^d(t - s + \epsilon)\int_{s - \epsilon}^sdr\int_{\mathbb{R}}dv \int_r^t\int_{\mathbb{R}}G_{\alpha}(t - \theta, x - \eta)\left(D_{r, v}^{(k)}(b_{i}(u(\theta, \eta)))\right)^2d\theta d\eta.
\end{align*}
By hypothesis \textbf{P1'} and Fubini's theorem,
\begin{align}\label{eq2017-09-15-14}
\mbox{E}\left[|A_2|^q\right] &\leq c(t - s + \epsilon)^q\sum_{k, l = 1}^d \mbox{E}\bigg[\Big|\int_{s - \epsilon}^sdr\int_{\mathbb{R}}dv\int_r^td\theta\int_{\mathbb{R}}d\eta \, G_{\alpha}(t - \theta, x - \eta)(D_{r, v}^{(k)}(u_l(\theta, \eta)))^2\Big|^q\bigg] \nonumber\\
&=  c(t - s + \epsilon)^q\sum_{k, l = 1}^d \mbox{E}\Bigg[\Big|\int_{s - \epsilon}^td\theta\int_{\mathbb{R}}d\eta \, G_{\alpha}(t - \theta, x - \eta)\nonumber\\
& \quad \quad \quad \quad  \quad \quad \quad \quad \quad \quad \quad \quad \quad \quad \times  \int_{s - \epsilon}^{s \wedge \theta}dr\int_{\mathbb{R}}dv\left(D_{r, v}^{(k)}(u_l(\theta, \eta))\right)^2\Big|^q\Bigg].
\end{align}
We apply H\"{o}lder's inequality with respect to the measure $G_{\alpha}(t - \theta, x - \eta)d\theta d\eta$ to find that
\begin{align}\label{eq2017-09-15-12}
\mbox{E}\left[|A_{2}|^q\right] &\leq c(t - s + \epsilon)^q\sum_{k, l = 1}^d\left|\int_{s - \epsilon}^{t}d\theta\int_{\mathbb{R}}d\eta \, G_{\alpha}(t - \theta, x - \eta)\right|^{q - 1}\nonumber\\
& \quad  \times \int_{s - \epsilon}^{t}d\theta\int_{\mathbb{R}}d\eta \, G_{\alpha}(t - \theta, x - \eta)\mbox{E}\left[\left|\int_{s - \epsilon}^{s \wedge \theta}dr\int_{\mathbb{R}}dv\left(D_{r, v}^{(k)}(u_l(\theta, \eta))\right)^2\right|^q\right] \nonumber \\
 &\leq c(t - s + \epsilon)^q\left|\int_{s - \epsilon}^{t}d\theta\int_{\mathbb{R}}d\eta \, G_{\alpha}(t - \theta, x - \eta)\right|^{q}\epsilon^{\frac{\alpha - 1}{\alpha}q}\nonumber\\
 &= c(t - s + \epsilon)^{2q}\epsilon^{\frac{\alpha - 1}{\alpha}q},
\end{align}
where in the second inequality we use Lemma \ref{lemma2017-09-06-1}.
Hence \eqref{eq2017-09-15-10} and \eqref{eq2017-09-15-12} prove the lemma.
\end{proof}

The following lemma  improves Lemma \ref{lemma2017-09-15-1} by using Lemma \ref{lemma2017-09-14-1}. As we mentioned in Section \ref{section1}, this is a key ingredient in our improvement of the lower bound in \eqref{eq2017-10-30-1000}.
\begin{lemma} \label{lemma2017-09-19-1}
Assume \textbf{P1'}. Fix $T > 0, c_0 > 1$ and $0 < \gamma_0 < 1$.  For all $q \geq 1$, there exists a constant $c = c(c_0, q, T) \in \,]0, \infty[$ such that for every $0 < 2\epsilon \leq s \leq t \leq T$ with $t - s > c_0\epsilon^{\gamma_0}$ and $x \in \mathbb{R}$,
\begin{align*}
\mbox{E}\left[\left(\sum_{k = 1}^d\int_{s - \epsilon}^sdr\int_{\mathbb{R}}dv\sum_{i = 1 }^da_i^2(k, r, v, t, x)\right)^q\right] \leq c\epsilon^{\min((1 + \gamma_0)\frac{\alpha - 1}{\alpha}, 1 - \frac{\gamma_0}{\alpha})q}.
\end{align*}
\end{lemma}
\begin{proof}
We  use again the notations from the proof of Lemma \ref{lemma2017-09-15-1}. From \eqref{eq2018-04-30-2} and Burkholder's inequality for Hilbert-space-valued martingales (Lemma \ref{lemma2017-09-05-5}), we have
\begin{align*}
\mbox{E}\left[|A_1|^q\right] &\leq c \sum_{k, l = 1}^d \mbox{E}\left[\left|\int_{s - \epsilon}^td\theta\int_{\mathbb{R}}d\eta \, G_{\alpha}^2(t - \theta, x - \eta)\int_{s - \epsilon}^{s\wedge \theta}dr\int_{\mathbb{R}}dv\left(D_{r, v}^{(k)}(u_l(\theta, \eta))\right)^2\right|^q\right]\\
&\leq A_{11} + A_{12} + A_{13},
\end{align*}
with
\begin{align*}
 A_{11} &:= c \sum_{k, l = 1}^d \mbox{E}\left[\left|\int_{s - \epsilon}^{s}d\theta\int_{\mathbb{R}}d\eta \, G_{\alpha}^2(t - \theta, x - \eta)\int_{s - \epsilon}^{s\wedge \theta}dr\int_{\mathbb{R}}dv\left(D_{r, v}^{(k)}(u_l(\theta, \eta))\right)^2\right|^q\right],\\
 A_{12} &:= c \sum_{k, l = 1}^d \mbox{E}\left[\left|\int_{s}^{s + c_0\epsilon^{\gamma_0}}d\theta\int_{\mathbb{R}}d\eta \, G_{\alpha}^2(t - \theta, x - \eta)\int_{s - \epsilon}^{s\wedge \theta}dr\int_{\mathbb{R}}dv\left(D_{r, v}^{(k)}(u_l(\theta, \eta))\right)^2\right|^q\right],\\
 A_{13} &:= c \sum_{k, l = 1}^d \mbox{E}\left[\left|\int_{s + c_0\epsilon^{\gamma_0}}^td\theta\int_{\mathbb{R}}d\eta \, G_{\alpha}^2(t - \theta, x - \eta)\int_{s - \epsilon}^{s}dr\int_{\mathbb{R}}dv\left(D_{r, v}^{(k)}(u_l(\theta, \eta))\right)^2\right|^q\right],
 \end{align*}
 and from \eqref{eq2017-09-15-14},
 \begin{align*}
\mbox{E}\left[|A_2|^q\right] &\leq c\sum_{k, l = 1}^d \mbox{E}\Bigg[\Big|\int_{s - \epsilon}^td\theta\int_{\mathbb{R}}d\eta  \, G_{\alpha}(t - \theta, x - \eta)\int_{s - \epsilon}^{s \wedge \theta}dr\int_{\mathbb{R}}dv\left(D_{r, v}^{(k)}(u_l(\theta, \eta))\right)^2\Big|^q\Bigg]\\
&\leq A_{21} + A_{22}  + A_{23},
\end{align*}
where $A_{2j}$ is defined in the same way as $A_{1j}$, but with $G_{\alpha}^2$ replaced by $G_{\alpha}$, $j = 1, 2, 3$.

We first bound $\mbox{E}\left[|A_1|^q\right]$. We apply H\"{o}lder's inequality with respect to the measure $G_{\alpha}^2(t - \theta, x - \eta)d\theta d\eta$ to find that
\begin{align}
\mbox{E}\left[|A_{11}|^q\right] &\leq c \sum_{k, l = 1}^d \left(\int_{s - \epsilon}^{s}d\theta\int_{\mathbb{R}}d\eta \, G_{\alpha}^2(t - \theta, x - \eta)\right)^{q - 1}  \nonumber \\
& \quad \times \int_{s - \epsilon}^{s}d\theta\int_{\mathbb{R}}d\eta \, G_{\alpha}^2(t - \theta, x - \eta) \mbox{E}\left[\left|\int_{s - \epsilon}^{ \theta}dr\int_{\mathbb{R}}dv\left(D_{r, v}^{(k)}(u_l(\theta, \eta))\right)^2\right|^q\right]. \nonumber
\end{align}
For $\theta \in [s - \epsilon, s]$, we have $s - \epsilon \geq \theta - \epsilon \geq 0$. Hence by Lemma \ref{lemma2017-09-06-1},
\begin{align} \label{eq2017-09-15-4}
\mbox{E}\left[\left|\int_{s - \epsilon}^{\theta}dr\int_{\mathbb{R}}dv\left(D_{r, v}^{(k)}(u_l(\theta, \eta))\right)^2\right|^q\right] & \leq \mbox{E}\left[\left|\int_{\theta - \epsilon}^{\theta}dr\int_{\mathbb{R}}dv\left(D_{r, v}^{(k)}(u_l(\theta, \eta))\right)^2\right|^q\right] \leq c \epsilon^{\frac{\alpha - 1}{\alpha}q},
\end{align}
where $c \in \, ]0, \infty[$ does not depend on $(\theta, \eta, s, t, \epsilon, x)$.
 Therefore, by \eqref{eq2017-09-05-1},
 \begin{align} \label{eq2017-09-15-2}
\mbox{E}\left[|A_{11}|^q\right] &\leq c \epsilon^{\frac{\alpha - 1}{\alpha}q} \left(\int_{s - \epsilon}^{s}d\theta\int_{\mathbb{R}}d\eta \, G_{\alpha}^2(t - \theta, x - \eta)\right)^{q}  \nonumber \\
& = c \epsilon^{\frac{\alpha - 1}{\alpha}q}\left((t - s + \epsilon)^{(\alpha - 1)/\alpha} - (t - s)^{(\alpha - 1)/\alpha}\right)^q\nonumber \\
& \leq c \epsilon^{\frac{\alpha - 1}{\alpha}q}\epsilon^{(1 - \frac{\gamma_0}{\alpha})q},
\end{align}
 where, in the last inequality, we perform the same calculation as in \eqref{eq2017-09-14-1} under the assumption $t - s > c_0\epsilon^{\gamma_0}$.
Again,we apply H\"{o}lder's inequality with respect to the measure $G_{\alpha}^2(t - \theta, x - \eta)d\theta d\eta$ to find that
\begin{align}
\mbox{E}\left[|A_{12}|^q\right] &\leq c \sum_{k, l = 1}^d \left(\int_{s}^{s + c_0\epsilon^{\gamma_0}}d\theta\int_{\mathbb{R}}d\eta \, G_{\alpha}^2(t - \theta, x - \eta)\right)^{q - 1}  \nonumber \\
& \quad \times \int_{s}^{s + c_0\epsilon^{\gamma_0}}d\theta\int_{\mathbb{R}}d\eta \, G_{\alpha}^2(t - \theta, x - \eta) \mbox{E}\left[\left|\int_{s - \epsilon}^{s}dr\int_{\mathbb{R}}dv\left(D_{r, v}^{(k)}(u_l(\theta, \eta))\right)^2\right|^q\right]. \nonumber
\end{align}
The expectation is bounded as in  \eqref{eq2017-09-15-4} by Lemma \ref{lemma2017-09-06-1}.
Consequently, using in the second inequality below, the fact that $t - s > c_0\epsilon^{\gamma_0}$ and the function $x \mapsto x ^{\frac{\alpha - 1}{\alpha}} - (x - c_0\epsilon^{\gamma_0})^{\frac{\alpha - 1}{\alpha}}$ is decreasing on $[c_0\epsilon^{\gamma_0}, \infty[$,
\begin{align} \label{eq2017-09-15-1}
\mbox{E}\left[|A_{12}|^q\right] &\leq c\left(\int_{s }^{s + c_0\epsilon^{\gamma_0}}d\theta\int_{\mathbb{R}}d\eta \, G_{\alpha}^2(t - \theta, x - \eta)\right)^{q}\epsilon^{\frac{\alpha - 1}{\alpha}q} \nonumber\\
&= c\left((t - s)^{\frac{\alpha - 1}{\alpha}} - (t - s - c_0\epsilon^{\gamma_0})^{\frac{\alpha - 1}{\alpha}}\right)^q\epsilon^{\frac{\alpha - 1}{\alpha}q}\nonumber\\
&\leq c\left((c_0\epsilon^{\gamma_0})^{\frac{\alpha - 1}{\alpha}} - (c_0\epsilon^{\gamma_0} - c_0\epsilon^{\gamma_0})^{\frac{\alpha - 1}{\alpha}}\right)^q\epsilon^{\frac{\alpha - 1}{\alpha}q}\nonumber\\
&= c(c_0\epsilon^{\gamma_0})^{\frac{\alpha - 1}{\alpha}q}\epsilon^{\frac{\alpha - 1}{\alpha}q} = c'\epsilon^{(1 + \gamma_0)\frac{\alpha - 1}{\alpha}q}.
\end{align}

For $A_{13}$, we have, by H\"{o}lder's inequality with respect to the measure $G_{\alpha}^2(t - \theta, x - \eta)d\theta d\eta$,
\begin{align*}
\mbox{E}\left[|A_{13}|^q\right] &\leq c \sum_{k, l = 1}^d \left(\int_{s + c_0\epsilon^{\gamma_0}}^td\theta\int_{\mathbb{R}}d\eta \, G_{\alpha}^2(t - \theta, x - \eta)\right)^{q - 1} \nonumber \\
& \quad \times \int_{s + c_0\epsilon^{\gamma_0}}^td\theta\int_{\mathbb{R}}d\eta \, G_{\alpha}^2(t - \theta, x - \eta) \mbox{E}\left[\left|\int_{s - \epsilon}^{s }dr\int_{\mathbb{R}}dv\left(D_{r, v}^{(k)}(u_l(\theta, \eta))\right)^2\right|^q\right]. \nonumber
\end{align*}
Lemma \ref{lemma2017-09-14-1} implies that for any $\theta \in \, ]s + c_0\epsilon^{\gamma_0}, t[$,
\begin{align*}
\sum_{k, l = 1}^d \mbox{E}\left[\left|\int_{s - \epsilon}^{s }dr\int_{\mathbb{R}}dv\left(D_{r, v}^{(k)}(u_l(\theta, \eta))\right)^2\right|^q\right] \leq c \epsilon^{(1 - \frac{\gamma_0}{\alpha})q},
\end{align*}
where $c \in \, ]0, \infty[$ does not depend on $(\theta, \eta, s, t, \epsilon, x)$. Thus,  by \eqref{eq2017-09-05-1},
\begin{align}\label{eq2017-09-15-3}
\mbox{E}\left[|A_{13}|^q\right] &\leq c\left(\int_{s + c_0\epsilon^{\gamma_0}}^td\theta\int_{\mathbb{R}}d\eta \, G_{\alpha}^2(t - \theta, x - \eta)\right)^{q}\epsilon^{(1 - \frac{\gamma_0}{\alpha})q} \nonumber\\
&= c(t - s - c_0\epsilon^{\gamma_0})^{\frac{\alpha - 1}{\alpha}q}\epsilon^{(1 - \frac{\gamma_0}{\alpha})q} \leq c'\epsilon^{(1 - \frac{\gamma_0}{\alpha})q}.
\end{align}

We proceed to derive a similar bound for $\mbox{E}\left[|A_2|^q\right]$.
We apply H\"{o}lder's inequality with respect to the measure $G_{\alpha}(t - \theta, x - \eta)d\theta d\eta$ to find that
\begin{align}\label{eq2017-09-15-6}
\mbox{E}\left[|A_{21}|^q\right] &\leq c\sum_{k, l = 1}^d\left|\int_{s - \epsilon}^{s}d\theta\int_{\mathbb{R}}d\eta \, G_{\alpha}(t - \theta, x - \eta)\right|^{q - 1}\nonumber\\
& \quad  \times  \int_{s - \epsilon}^{s }d\theta\int_{\mathbb{R}}d\eta \, G_{\alpha}(t - \theta, x - \eta)\mbox{E}\left[\left|\int_{s - \epsilon}^{\theta}dr\int_{\mathbb{R}}dv\left(D_{r, v}^{(k)}(u_l(\theta, \eta))\right)^2\right|^q\right] \nonumber \\
 &\leq c\left|\int_{s - \epsilon}^{s}d\theta\int_{\mathbb{R}}d\eta \, G_{\alpha}(t - \theta, x - \eta)\right|^{q}\epsilon^{\frac{\alpha - 1}{\alpha}q} =  c\epsilon^{(\frac{\alpha - 1}{\alpha} + 1)q},
\end{align}
where in the second inequality we use (\ref{eq2017-09-15-4}).
Similarly, we apply H\"{o}lder's inequality with respect to the measure $G_{\alpha}(t - \theta, x - \eta)d\theta d\eta$ to find that
\begin{align}\label{eq2017-09-15-7}
\mbox{E}\left[|A_{22}|^q\right] &\leq c\sum_{k, l = 1}^d\left|\int_{s}^{s + c_0\epsilon^{\gamma_0}}d\theta\int_{\mathbb{R}}d\eta \, G_{\alpha}(t - \theta, x - \eta)\right|^{q - 1} \nonumber\\
& \quad \times \int_{s}^{s + c_0\epsilon^{\gamma_0}}d\theta\int_{\mathbb{R}}d\eta \, G_{\alpha}(t - \theta, x - \eta)\mbox{E}\left[\left|\int_{s - \epsilon}^{s}dr\int_{\mathbb{R}}dv\left(D_{r, v}^{(k)}(u_l(\theta, \eta))\right)^2\right|^q\right] \nonumber\\
 &\leq c\left|\int_{s}^{s + c_0\epsilon^{\gamma_0}}d\theta\int_{\mathbb{R}}d\eta \, G_{\alpha}(t - \theta, x - \eta)\right|^{q}\epsilon^{\frac{\alpha - 1}{\alpha}q}\nonumber\\
 &= c(c_0\epsilon^{\gamma_0})^{q}\epsilon^{\frac{\alpha - 1}{\alpha}q} =  c'\epsilon^{(\frac{\alpha - 1}{\alpha} + \gamma_0)q},
\end{align}
where in the second inequality we use Lemma \ref{lemma2017-09-06-1}. For the last term, we use H\"{o}lder's inequality with respect to the measure $G_{\alpha}(t - \theta, x - \eta)d\theta d\eta$ to see that
\begin{align}\label{eq2017-09-15-5}
\mbox{E}\left[|A_{23}|^q\right] &\leq c\sum_{k, l = 1}^d\left|\int_{s + c_0\epsilon^{\gamma_0}}^td\theta\int_{\mathbb{R}}d\eta \, G_{\alpha}(t - \theta, x - \eta)\right|^{q - 1} \nonumber\\
& \quad \times \int_{s + c_0\epsilon^{\gamma_0}}^td\theta\int_{\mathbb{R}}d\eta \, G_{\alpha}(t - \theta, x - \eta)\mbox{E}\left[\left|\int_{s - \epsilon}^{s}dr\int_{\mathbb{R}}dv\left(D_{r, v}^{(k)}(u_l(\theta, \eta))\right)^2\right|^q\right]\nonumber \\
 &\leq c\left|\int_{s + c_0\epsilon^{\gamma_0}}^td\theta\int_{\mathbb{R}}d\eta \, G_{\alpha}(t - \theta, x - \eta)\right|^{q}\epsilon^{(1 - \frac{\gamma_0}{\alpha})q}\nonumber\\
 &= c(t - s - c_0\epsilon^{\gamma_0})^{q}\epsilon^{(1 - \frac{\gamma_0}{\alpha})q}\leq c'\epsilon^{(1 - \frac{\gamma_0}{\alpha})q},
\end{align}
where in the second inequality we use Lemma \ref{lemma2017-09-14-1}.

Finally, from \eqref{eq2017-09-15-2}, \eqref{eq2017-09-15-1}, \eqref{eq2017-09-15-3}, \eqref{eq2017-09-15-6}, \eqref{eq2017-09-15-7} and \eqref{eq2017-09-15-5}, together with the choice of $\gamma_0$, we obtain the desired result.
\end{proof}

\begin{remark}\label{remark2017-11-21-3}
The result of Lemma \ref{lemma2017-09-19-1} is also true for solutions of stochastic heat equations with Neumann or Dirichlet boundary conditions since we can still apply the result of Lemma \ref{lemma2017-09-14-1}; see Remark \ref{remark2017-11-21-2}.
\end{remark}

\subsection{Study of the Malliavin matrix}

Fix $T > 0$. For $s, t \in [0, T], s \leq t$, and $x, y \in \mathbb{R}$, consider the $2d$-dimensional random vector
\begin{align}\label{eq2017-09-15-15}
Z &:= (u(s, y), u(t, x) - u(s, y)).
\end{align}
Let $\gamma_Z$ be the Malliavin matrix of $Z$. Note that
$\gamma_Z = ((\gamma_Z)_{m, l})_{m, l = 1, \ldots, 2d}$ is a symmetric $2d \times 2d$ random matrix with four $d \times d$ blocs of the form
\begin{align*}
\gamma_Z = \left(
  \begin{array}{ccc}
    \gamma_Z^{(1)} & \vdots  & \gamma_Z^{(2)} \\
    \cdots  & \vdots  & \cdots  \\
    \gamma_Z^{(3)} &  \vdots & \gamma_Z^{(4)} \\
  \end{array}
\right)
\end{align*}
where
\begin{align*}
\gamma_Z^{(1)} &= \left(\left\langle D(u_i(s, y)), D(u_j(s, y))\right\rangle_{\mathscr{H}}\right)_{i, j = 1, \dots, d},\\
\gamma_Z^{(2)} &=   \left(\left\langle D(u_i(s, y)), D(u_j(t, x)- u_j(s, y))\right\rangle_{\mathscr{H}}\right)_{i, j = 1, \dots, d},\\
\gamma_Z^{(3)} &=  \left(\left\langle D(u_i(t, x)- u_i(s, y)), D(u_j(s, y))\right\rangle_{\mathscr{H}}\right)_{i, j = 1, \dots, d},\\
\gamma_Z^{(4)} &=  \left(\left\langle D(u_i(t, x) - u_i(s, y)), D(u_j(t, x)- u_j(s, y))\right\rangle_{\mathscr{H}}\right)_{i, j = 1, \dots, d}.
\end{align*}
We let ($\mathbf{1}$) denote the couples of $\{1, \ldots, d\} \times \{1, \ldots, d\}$, ($\mathbf{2}$) denote the couples of $\{1, \ldots, d\} \times \{d + 1, \ldots, 2d\}$, ($\mathbf{3}$) denote the couples of $\{d + 1, \ldots, 2d\} \times \{1, \ldots, d\}$ and ($\mathbf{4}$) denote the couples of $\{d + 1, \ldots, 2d\} \times \{d + 1, \ldots, 2d\}$.

The next two results follow exactly along the same lines as \cite[Propositions 6.5 and 6.7]{DKN09} using (\ref{eq2017-09-05-4}) and Proposition \ref{prop2017-09-13-1}, with $\Delta$ there replaced by $\Delta^2_{\alpha}$. We omit the proofs.

\begin{prop}\label{prop2017-09-19-1}
Fix $T > 0$ and let $I$ and $J$ be compact intervals as in Theorem \ref{theorem2017-08-17-1}. Let $A_Z$ denote the cofactor matrix of $\gamma_Z$. Assuming \textbf{P1'}, for any $(s, y), (t, x) \in I \times J, (s, y) \neq (t, x), p > 1$,
\begin{align*}
\mbox{E}\left[|(A_Z)_{m, l}|^p\right]^{1/p}
 \leq
\left\{\begin{array}{ll}
    c_{p, T} (|t - s|^{\frac{\alpha - 1}{\alpha}} + |x - y|^{\alpha - 1})^d & \hbox{if $(m, l) \in (\mathbf{1})$,} \\
    c_{p, T} (|t - s|^{\frac{\alpha - 1}{\alpha}} + |x - y|^{\alpha - 1})^{d - \frac{1}{2}}   & \hbox{if $(m, l) \in (\mathbf{2})$ or $(\mathbf{3})$,} \\
    c_{p, T} (|t - s|^{\frac{\alpha - 1}{\alpha}} + |x - y|^{\alpha - 1})^{d - 1} & \hbox{if $(m, l) \in (\mathbf{4})$.}
  \end{array}
\right.
\end{align*}
\end{prop}
\begin{prop}\label{prop2017-09-19-2}
Fix $T > 0$ and let $I$ and $J$ be compact intervals as in Theorem \ref{theorem2017-08-17-1}. Assuming \textbf{P1'}, for any $(s, y), (t, x) \in I \times J, (s, y) \neq (t, x), p > 1$,
\begin{align*}
\mbox{E}\left[\|D^k(\gamma_Z)_{m, l}\|_{\mathscr{H}^{\otimes k}}^p\right]^{1/p}
 \leq
\left\{\begin{array}{ll}
    c_{k, p, T} & \hbox{if $(m, l) \in (\mathbf{1})$,} \\
    c_{k, p, T} (|t - s|^{\frac{\alpha - 1}{\alpha}} + |x - y|^{\alpha - 1})^{\frac{1}{2}}   & \hbox{if $(m, l) \in (\mathbf{2})$ or $(\mathbf{3})$,} \\
    c_{k, p, T} (|t - s|^{\frac{\alpha - 1}{\alpha}} + |x - y|^{\alpha - 1}) & \hbox{if $(m, l) \in (\mathbf{4})$.}
  \end{array}
\right.
\end{align*}
\end{prop}

The main technical effort in this subsection is the proof of the following proposition, which improves \cite[Proposition 6.6(a)]{DKN09} and is why the $\eta$ can be removed in the lower bound on hitting probabilities.
\begin{prop}\label{prop2017-09-19-3}
Fix $T > 0$ and let $I$ and $J$ be compact intervals as in Theorem \ref{theorem2017-08-17-1}. Assume \textbf{P1'} and \textbf{P2}. There exists $C$ depending on $T$ such that for any $(s, y), (t, x) \in I \times J, (s, y) \neq (t, x), p > 1$,
\begin{align}\label{eq2017-09-19-1}
\mbox{E}\left[(\det \gamma_Z)^{-p}\right]^{1/p} \leq C (|t - s|^{\frac{\alpha - 1}{\alpha}} + |x - y|^{\alpha - 1})^{-d}.
\end{align}
\end{prop}
\begin{proof}
The proof has the same structure as that of \cite[Proposition 6.6]{DKN09}; see also \cite[Propositon 5.5]{DKN13}. We write
\begin{align}\label{eq2017-09-19-2}
\det \gamma_Z = \prod_{i = 1}^{2d}\big(\xi^i\big)^T\gamma_Z\xi^i,
\end{align}
where $\xi = \{\xi^1, \ldots, \xi^{2d}\}$ is an orthogonal basis of $\mathbb{R}^{2d}$ consisting of eigenvectors of $\gamma_Z$.

We use the perturbation argument of \cite[Proposition 6.6]{DKN09}. Let $\mathbf{0} \in \mathbb{R}^d$. Consider the spaces $E_1 = \{(\lambda, \mathbf{0}): \lambda \in \mathbb{R}^d\}$ and $E_2 = \{(\mathbf{0}, \mu): \mu \in \mathbb{R}^d\}$. Each $\xi^i$ can be written
\begin{align}\label{eq2017-09-19-3}
\xi^i = (\lambda^i, \mu^i) = \beta_i(\tilde{\lambda}^i, \mathbf{0}) + \sqrt{1 - \beta_i^2}(\mathbf{0}, \tilde{\mu}^i),
\end{align}
where $\lambda^i, \mu^i \in \mathbf{R}^d, (\tilde{\lambda}^i, \mathbf{0}) \in E_1, (\mathbf{0}, \tilde{\mu}^i) \in E_2$, with $\|\tilde{\lambda}^i\| = \|\tilde{\mu}^i\| = 1$ and $0 \leq \beta_i \leq 1$. In particular, $\|\xi^i\|^2 = \|\lambda^i\|^2 + \|\mu^i\|^2 = 1$.

For a fixed small $\beta_0$, the result of \cite[Lemma 6.8]{DKN09} gives us at least $d$ eigenvectors $\xi^1, \ldots, \xi^d$ satisfying $\beta_i \geq \beta_0, i = 1, \ldots, d$, which we say have a "large projection on $E_1$". We will show that these will contribute a factor of order $1$ to the product in (\ref{eq2017-09-19-2}). The at most $d$ other eigenvectors will each contribute a factor of order $|t - s|^{\frac{\alpha - 1}{\alpha}} + |x - y|^{\alpha - 1}$, which we say have a "small projection on $E_1$".

Hence, by \cite[Lemma 6.8]{DKN09} and the Cauchy-Schwarz inequality, we can write
\begin{align}\label{eq2017-09-19-4}
\mbox{E}\left[(\det \gamma_Z)^{-p}\right]^{1/p} &\leq \sum_{K \subset \{1, \ldots, 2d\}, |K| = d}\Bigg(\mbox{E}\Bigg[\mathbf{1}_{A_K}\bigg(\prod_{i \in K}(\xi^i)^T\gamma_Z\xi^i\bigg)^{-2p}\Bigg]\Bigg)^{1/(2p)} \nonumber \\
& \quad \times \Bigg(\mbox{E}\Bigg[\bigg(\inf_{\scriptsize\begin{array}{c}  \xi = (\lambda, \mu) \in \mathbb{R}^{2d}:\\ \|\lambda\|^2 + \|\mu\|^2 = 1 \end{array}}\xi^T\gamma_Z\xi\bigg)^{-2dp}\Bigg]\Bigg)^{1/(2p)},
\end{align}
where $A_K = \cap_{i \in K}\{\beta_i \geq \beta_0\}$.

With this, Propositions \ref{prop2017-09-19-4} and \ref{prop2017-09-19-5} below will conclude the proof of Proposition \ref{prop2017-09-19-3}.
\end{proof}

\begin{remark}\label{remark2017-11-21-5}
As a consequence of Remark \ref{remark2017-11-21-4}, we see that the result of Proposition \ref{prop2017-09-19-3} is also true for the solutions of stochastic heat equations with Neumann or Dirichlet boundary conditions.
\end{remark}

\begin{prop}\label{prop2017-09-19-4}
Fix $T > 0$. Assume \textbf{P1'} and \textbf{P2}. There exists $C$ depending on $T$ such that for all $s, t \in I, 0 \leq t - s < 1, x, y \in J, (s, y) \neq (t, x)$, and $p > 1$,
      \begin{align}\label{eq2017-09-19-5}
      \mbox{E}\Bigg[\bigg(\inf_{\scriptsize\begin{array}{c} \xi = (\lambda, \mu) \in \mathbb{R}^{2d}:\\ \|\lambda\|^2 + \|\mu\|^2 = 1 \end{array}}\xi^T\gamma_Z\xi\bigg)^{-2dp}\Bigg] &\leq C(|t - s|^{\frac{\alpha - 1}{\alpha}} + |x - y|^{\alpha - 1})^{-2dp}.
            \end{align}
\end{prop}

We are going to apply Lemma \ref{lemma2017-09-19-1} to prove this proposition. This is an improvement over the proof of \cite[Proposition 6.9]{DKN09} in which an extra exponent $\eta$ appears.

\begin{prop}\label{prop2017-09-19-5}
 Assume \textbf{P1'} and \textbf{P2}. Fix $T > 0$ and $p > 1$. Then there exists $C = C(p, T)$ such that for all $s, t \in I$ with $t \geq s, x, y \in J, (s, y) \neq (t, x)$,
\begin{align}\label{eq2017-09-19-6}
\mbox{E}\left[\mathbf{1}_{A_K}\left(\prod_{i \in K}(\xi^i)^T\gamma_Z\xi^i\right)^{-p}\right] \leq C,
\end{align}
where $A_K$ is defined just below (\ref{eq2017-09-19-4}).
\end{prop}

\begin{proof}[Proof of Proposition \ref{prop2017-09-19-4}]

Since $\gamma_Z$ is a matrix of inner products, we can write
\begin{align*}
\xi^T\gamma_Z\xi &= \sum_{k = 1}^d\int_0^Tdr\int_{\mathbb{R}}dv\Big(\sum_{i = 1}^d\big(\lambda_i D_{r,v}^{(k)}(u_i(s, y))  + \mu_i(D_{r,v}^{(k)}(u_i(t, x)) - D_{r,v}^{(k)}(u_i(s, y)))\big)\Big)^2.
\end{align*}
From here on, the proof is divided into two cases.

\textbf{Case 1}. In the first case, we assume that $t - s > 0$ and $|x - y|^{\alpha} \leq t - s$. Choose and fix an $\epsilon \in\,  ]0, \delta(t - s)[$, where $0 < \delta < 1$ is small but fixed; its specific value will be decided later on (see the line above \eqref{eq2017-09-20-2}). Then we may write
\begin{align*}
\xi^T\gamma_Z\xi \geq J_1 + J_2,
\end{align*}
where
\begin{align*}
J_1 &:= \sum_{k = 1}^d\int_{s - \epsilon}^sdr\int_{\mathbb{R}}dv\left(\sum_{i = 1}^d(\lambda_i - \mu_i)[G_{\alpha}(s - r, y - v)\sigma_{ik}(u(r, v)) + a_i(k, r, v, s, y)] + W\right)^2,\\
J_2 &:= \sum_{k = 1}^d\int_{t - \epsilon}^tdr\int_{\mathbb{R}}dv \, W^2,
\end{align*}
$a_i(k, r, v, s, y)$ is defined in (\ref{eq2017-09-05-4001122}) and
\begin{align*}
W := \sum_{i = 1}^d[\mu_i G_{\alpha}(t - r, x - v)\sigma_{ik}(u(r, v)) + \mu_ia_i(k, r, v, t, x)].
\end{align*}

\textit{Sub-case A}: $\epsilon \leq \delta(t - s)^{1/\gamma_0}$ with $0 < \gamma_0 < 1$. In this sub-case, by the elementary inequality (\ref{eq2017-09-06-2}),
\begin{align*}
J_2 \geq \hat{Y}_{1,\epsilon} - Y_{1,\epsilon},
\end{align*}
where
\begin{align*}
\hat{Y}_{1,\epsilon} &:= \frac{2}{3}\sum_{k = 1}^d\int_{t - \epsilon}^tdr\int_{\mathbb{R}}dv\left(\sum_{i = 1}^d\mu_i\sigma_{ik}(u(r, v))\right)^2G_{\alpha}^2(t - r, x - v),\\
Y_{1,\epsilon} &:= 2\sup\limits_{\|\mu\| \leq 1}\sum_{k = 1}^d\int_{t - \epsilon}^tdr\int_{\mathbb{R}}dv\left(\sum_{i = 1}^d\mu_ia_i(k, r, v, t, x)\right)^2.
\end{align*}
In agreement with hypothesis \textbf{P2} and by \eqref{eq2017-09-05-1},
\begin{align*}
\hat{Y}_{1,\epsilon} &\geq c \|\mu\|^2\int_{t - \epsilon}^tdr\int_{\mathbb{R}}dv \, G_{\alpha}^2(t - r, x - v) = c'\|\mu\|^2\epsilon^{\frac{\alpha - 1}{\alpha}}.
\end{align*}
Next we apply Lemma \ref{lemma2017-09-15-1} [with $s := t$] to find that $\mbox{E}\left[|Y_{1,\epsilon}|^q\right] \leq c\epsilon^{\frac{2\alpha - 2}{\alpha}q}$, for any $q \geq 1$.

For $J_1$, we find that
\begin{align*}
J_1 \geq \hat{Y}_{2, \epsilon} - Y_{2, \epsilon},
\end{align*}
where
\begin{align*}
 \hat{Y}_{2, \epsilon} := \frac{2}{3}\sum_{k = 1}^d\int_{s - \epsilon}^sdr\int_{\mathbb{R}}dv\left(\sum_{i = 1}^d(\lambda_i - \mu_i)\sigma_{ik}(u(r, v))\right)^2G_{\alpha}^2(s - r, y - v),
\end{align*}
and
\begin{align*}
Y_{2, \epsilon} := 6(W_1 + W_2 + W_3),
\end{align*}
where
\begin{align}
W_1 &:= \sup\limits_{\|\xi\| = 1}\sum_{k = 1}^d\int_{s - \epsilon}^sdr\int_{\mathbb{R}}dv\left(\sum_{i = 1}^d\mu_iG_{\alpha}(t - r, x - v)\sigma_{ik}(u(r, v))\right)^2,\nonumber\\
W_2 &:= \sup\limits_{\|\xi\| = 1}\sum_{k = 1}^d\int_{s - \epsilon}^sdr\int_{\mathbb{R}}dv\left(\sum_{i = 1}^d(\lambda_i - \mu_i)a_i(k, r, v, s, y)\right)^2,\label{eq2018-07-09-1}\\
W_3 &:= \sup\limits_{\|\xi\| = 1}\sum_{k = 1}^d\int_{s - \epsilon}^sdr\int_{\mathbb{R}}dv\left(\sum_{i = 1}^d\mu_ia_i(k, r, v, t, x)\right)^2.\label{eq2018-07-09-2}
\end{align}
Hypothesis \textbf{P2} implies that $\hat{Y}_{2, \epsilon} \geq c \|\lambda - \mu\|\epsilon^{\frac{\alpha - 1}{\alpha}}$.
We next give an estimate on the $q$-th moment of $W_1$, which is better than in \cite{DKN09}. We apply the Cauchy-Schwartz inequality to find that, for any $q \geq 1$,
\begin{align*}
\mbox{E}\left[|W_1|^q\right] &\leq  \sup\limits_{\|\xi\| = 1}\|\mu\|^{2q}\times \mbox{E}\left[\left|\sum_{k = 1}^d\int_{s - \epsilon}^sdr\int_{\mathbb{R}}dv\sum_{i = 1}^d(\sigma_{ik}(u(r, v)))^2G_{\alpha}^2(t - r, x - v)\right|^q\right].
\end{align*}
Thanks to hypothesis \textbf{P1'} and \eqref{eq2017-09-05-1}, this is bounded above by
\begin{align*}
 c \left|\int_{s - \epsilon}^sdr\int_{\mathbb{R}}dv \, G_{\alpha}^2(t - r, x - v)\right|^q &= c ((t - s + \epsilon)^{\frac{\alpha - 1}{\alpha}} - (t - s)^{\frac{\alpha - 1}{\alpha}})^q\\
 & \leq  c\epsilon^{(1 - \frac{\gamma_0}{\alpha})q},
\end{align*}
where, in the inequality, we perform the same calculation as in \eqref{eq2017-09-14-1} under the assumption $t - s > c_0\epsilon^{\gamma_0}$ of this Sub-case A.

We bound the $q$-th moment of $W_2$ similarly as in \cite{DKN09}: By the Cauchy-Schwarz inequality,
\begin{align*}
\mbox{E}\left[|W_2|^q\right] &\leq  \sup\limits_{\|\xi\| = 1}\|\lambda - \mu\|^{2q}\times \mbox{E}\left[\left|\sum_{k = 1}^d\int_{s - \epsilon}^sdr\int_{\mathbb{R}}dv\sum_{i = 1}^da_i^2(k, r, v, s, y)\right|^q\right]\\
&\leq c \, \mbox{E}\left[\left|\sum_{k = 1}^d\int_{s - \epsilon}^sdr\int_{\mathbb{R}}dv\sum_{i = 1}^da_i^2(k, r, v, s, y)\right|^q\right].
\end{align*}
We apply Lemma \ref{lemma2017-09-15-1} [with $t := s$] to find that $\mbox{E}\left[|W_2|^q\right] \leq c\epsilon^{\frac{2\alpha - 2}{\alpha}q}$.

Furthermore, different from the estimate of the $q$-th moment of $W_3$ in \cite{DKN09},  under the assumption of this Sub-case A, by Lemma \ref{lemma2017-09-19-1} we find that, for any $q \geq 1$,
\begin{align*}
\mbox{E}\left[|W_3|^q\right] &\leq  \sup\limits_{\|\xi\| = 1}\|\mu\|^{2q}\times \mbox{E}\left[\left|\sum_{k = 1}^d\int_{s - \epsilon}^sdr\int_{\mathbb{R}}dv\sum_{i = 1}^da_i^2(k, r, v, t, x)\right|^q\right]\\
&\leq c\epsilon^{\min((1 + \gamma_0)\frac{\alpha - 1}{\alpha}, 1 - \frac{\gamma_0}{\alpha})q}.
\end{align*}
The preceding bounds for $W_1, W_2$ and $W_3$ prove, in conjunction, that
\begin{align*}
\mbox{E}\left[|Y_{2,\epsilon}|^q\right] \leq c\epsilon^{\min((1 + \gamma_0)\frac{\alpha - 1}{\alpha}, 1 - \frac{\gamma_0}{\alpha})q}.
\end{align*}
 Thus we have
\begin{align}\label{eq2017-09-20-1}
J_1 + J_2 &\geq \hat{Y}_{1, \epsilon} + \hat{Y}_{2, \epsilon} - Y_{1, \epsilon} - Y_{2, \epsilon} \nonumber\\
&\geq c(\|\mu\|^2 + \|\lambda - \mu\|^2)\epsilon^{\frac{\alpha - 1}{\alpha}} - Y_{1, \epsilon} - Y_{2, \epsilon} \nonumber\\
&\geq c\epsilon^{\frac{\alpha - 1}{\alpha}} - Y_{\epsilon},
\end{align}
where $Y_{\epsilon} := Y_{1, \epsilon} + Y_{2, \epsilon}$ satisfies
\begin{align}
\mbox{E}\left[|Y_{\epsilon}|^q\right] \leq c\epsilon^{\min((1 + \gamma_0)\frac{\alpha - 1}{\alpha}, 1 - \frac{\gamma_0}{\alpha})q}.
\end{align}

\textit{Sub-case B}: $\delta(t - s)^{1/\gamma_0} < \epsilon < \delta(t - s)$. In this sub-case, we are going to give a different estimate on $J_1$:
\begin{align*}
J_1 \geq \tilde{Y}_{\epsilon} - 4(W_2 + W_3),
\end{align*}
where
\begin{align*}
\tilde{Y}_{\epsilon}:= \frac{2}{3}\sum_{k = 1}^d\int_{s - \epsilon}^sdr\int_{\mathbb{R}}dv\left(\sum_{i = 1}^d[(\lambda_i - \mu_i)G_{\alpha}(s - r, y - v) + \mu_iG_{\alpha}(t - r, x - v)]\sigma_{ik}(u(r, v))\right)^2
\end{align*}
and $W_2$ and $W_3$ are defined in \eqref{eq2018-07-09-1} and \eqref{eq2018-07-09-2}.
Using the inequality
$
(a + b)^2 \geq a^2  - 2|ab|,
$
we see that
\begin{align*}
\tilde{Y}_{\epsilon} \geq \hat{Y}_{2, \epsilon} - \frac{4}{3}B_1^{(3)},
\end{align*}
where as above, $\hat{Y}_{2, \epsilon} \geq c \|\lambda - \mu\|\epsilon^{\frac{\alpha - 1}{\alpha}}$, and
\begin{align}
B_1^{(3)} &:= \sum_{k = 1}^d\int_{s - \epsilon}^sdr\int_{\mathbb{R}}dv\left|\sum_{i = 1}^d(\lambda_i - \mu_i)G_{\alpha}(s - r, y - v) \sigma_{ik}(u(r, v))\right| \nonumber\\
& \quad \times \left|\sum_{i = 1}^d\mu_iG_{\alpha}(t - r, x - v)\sigma_{ik}(u(r, v))\right|. \label{eq2018-07-09-3}
\end{align}
Hypothesis \textbf{P1'} assures us that
\begin{align*}
\left|B_1^{(3)}\right| &\leq c\int_{s - \epsilon}^sdr\int_{\mathbb{R}}dv \, G_{\alpha}(s - r, y - v)G_{\alpha}(t - r, x - v)\\
 &= c\int_{s - \epsilon}^sdr \, G_{\alpha}(t + s - 2r, x - y) = c\int_{0}^{\epsilon}dr \, G_{\alpha}(t - s + 2r, x - y),
 \end{align*}
 where, in the first equality, we use the semi-group property of the Green kernel \cite[Lemma 4.1(iii)]{ChD15}. Since for any $t > 0$, the function $x \mapsto G_{\alpha}(t, x)$ attains its maximum at $0$, this is bounded above by
 \begin{align*}
 c\int_{0}^{\epsilon}dr\, G_{\alpha}(t - s + 2r, 0)  &= c'\int_{0}^{\epsilon}dr(t - s + 2r)^{-\frac{1}{\alpha}}\\
    &= c'((t - s + 2\epsilon)^{\frac{\alpha - 1}{\alpha}} - (t - s)^{\frac{\alpha - 1}{\alpha}})\\
    &= c'\epsilon^{\frac{\alpha - 1}{\alpha}}\left((\frac{t - s}{\epsilon} + 2)^{\frac{\alpha - 1}{\alpha}} - (\frac{t - s}{\epsilon})^{\frac{\alpha - 1}{\alpha}}\right)\\
    &\leq c'\epsilon^{\frac{\alpha - 1}{\alpha}}((1/\delta + 2)^{\frac{\alpha - 1}{\alpha}} - (1/\delta)^{\frac{\alpha - 1}{\alpha}}),
\end{align*}
where the first equality is due to the scaling property of Green kernel \cite[Lemma 4.1(iv)]{ChD15} and in the inequality we use the assumption $\epsilon < \delta(t - s)$ and the fact that the function $x \mapsto (x + 2)^{\frac{\alpha - 1}{\alpha}} - x^{\frac{\alpha - 1}{\alpha}}$ is decreasing on $[0, \infty[$.
 Hence we have
\begin{align}
J_1 + J_2 &\geq \hat{Y}_{1, \epsilon} + \hat{Y}_{2, \epsilon} - \frac{4}{3}B_1^{(3)} - 4W_2 - 4W_3 - Y_{1, \epsilon} \nonumber\\
&\geq c(\|\mu\|^2 + \|\lambda - \mu\|^2)\epsilon^{\frac{\alpha - 1}{\alpha}} - c'\epsilon^{\frac{\alpha - 1}{\alpha}}((1/\delta + 2)^{\frac{\alpha - 1}{\alpha}} - (1/\delta)^{\frac{\alpha - 1}{\alpha}}) - 4W_2 - 4W_3 - Y_{1, \epsilon}\nonumber\\
&\geq c_0\epsilon^{\frac{\alpha - 1}{\alpha}} - c'\epsilon^{\frac{\alpha - 1}{\alpha}}((1/\delta + 2)^{\frac{\alpha - 1}{\alpha}} - (1/\delta)^{\frac{\alpha - 1}{\alpha}}) - 4W_2 - 4W_3 - Y_{1, \epsilon}\nonumber
\end{align}
We can choose $\delta$ small so that $c_0 > c'((1/\delta + 2)^{\frac{\alpha - 1}{\alpha}} - (1/\delta)^{\frac{\alpha - 1}{\alpha}})$ and therefore,
\begin{align}\label{eq2017-09-20-2}
J_1 + J_2 &\geq c\epsilon^{\frac{\alpha - 1}{\alpha}} - 4W_2 - 4W_3 - Y_{1, \epsilon}.
\end{align}
In this sub-case,
\begin{align*}
\mbox{E}\left[|W_2|^q\right] &\leq  \sup\limits_{\|\xi\| = 1}\|\lambda - \mu\|^{2q}\times \mbox{E}\left[\left|\sum_{k = 1}^d\int_{s - \epsilon}^sdr\int_{\mathbb{R}}dv\sum_{i = 1}^da_i^2(k, r, v, s, y)\right|^q\right]\\
&\leq c\, \mbox{E}\left[\left|\sum_{k = 1}^d\int_{s - \epsilon}^sdr\int_{\mathbb{R}}dv\sum_{i = 1}^da_i^2(k, r, v, s, y)\right|^q\right].
\end{align*}
We apply Lemma \ref{lemma2017-09-15-1} to find that $\mbox{E}\left[|W_2|^q\right] \leq c\epsilon^{\frac{2\alpha - 2}{\alpha}q}$.
Similarly, we find using Lemma \ref{lemma2017-09-15-1} and the assumption $\delta(t - s)^{1/\gamma_0} < \epsilon$ that
\begin{align*}
\mbox{E}\left[|W_3|^q\right] &\leq  \sup\limits_{\|\xi\| = 1}\|\mu\|^{2q}\times \mbox{E}\left[\left|\sum_{k = 1}^d\int_{s - \epsilon}^sdr\int_{\mathbb{R}}dv\sum_{i = 1}^da_i^2(k, r, v, t, x)\right|^q\right]\\
&\leq c(t - s + \epsilon)^{\frac{\alpha - 1}{\alpha}q}\epsilon^{\frac{\alpha - 1}{\alpha}q}\\
&\leq c(\delta^{-\gamma_0}\epsilon^{\gamma_0} + \epsilon)^{\frac{\alpha - 1}{\alpha}q}\epsilon^{\frac{\alpha - 1}{\alpha}q}\leq c'\epsilon^{(1 + \gamma_0)\frac{\alpha - 1}{\alpha}q}.
\end{align*}
Combining \eqref{eq2017-09-20-1} and \eqref{eq2017-09-20-2}, we have for $\epsilon \in \, ]0, \delta(t - s)[$,
\begin{align}
\inf_{\|\xi\| = 1}\xi^T\gamma_Z\xi \geq c\epsilon^{\frac{\alpha - 1}{\alpha}} - \tilde{Z}_{\epsilon},
\end{align}
where
\begin{align*}
\tilde{Z}_{\epsilon} := Y_{\epsilon}1_{\{\epsilon \leq \delta(t - s)^{1/\gamma_0}\}} + 4(W_2 + W_3 + Y_{1, \epsilon})1_{\{\delta(t - s)^{1/\gamma_0} < \epsilon < \delta(t - s)\}}
\end{align*}
and for all $q \geq 1$,
\begin{align}\label{eq2017-09-20-7}
\mbox{E}\left[|Y_{\epsilon}1_{\{\epsilon \leq \delta(t - s)^{1/\gamma_0}\}}|^q\right] \leq c\epsilon^{\min((1 + \gamma_0)\frac{\alpha - 1}{\alpha}, 1 - \frac{\gamma_0}{\alpha})q},
\end{align}
and
\begin{align}\label{eq2017-09-20-9}
 \mbox{E}\left[|4(W_2 + W_3 + Y_{1, \epsilon})1_{\{\delta(t - s)^{1/\gamma_0} < \epsilon < \delta(t - s)\}}|^q\right] \leq c\epsilon^{(1 + \gamma_0)\frac{\alpha - 1}{\alpha}q}.
\end{align}
We use \cite[Proposition 3.5]{DKN09} to find that
\begin{align*}
\mbox{E}\left[\left(\inf_{\|\xi\| = 1}\xi^T\gamma_Z\xi\right)^{-2pd}\right] &\leq c\left(\delta(t - s)\right)^{-2pd\frac{\alpha - 1}{\alpha}} = c'(t - s)^{-2pd\frac{\alpha - 1}{\alpha}}\\
  &\leq \tilde{c}\left[|t - s|^{\frac{\alpha - 1}{\alpha}} + |x - y|^{\alpha - 1}\right]^{-2pd},
\end{align*}
whence follows the result in the case that $|x - y|^{\alpha} \leq t - s < 1$.

\textbf{Case 2}. Now we work on the second case where $|x - y| > 0$ and $|x - y|^{\alpha} \geq t - s \geq 0$. Let $\epsilon > 0$ be such that $(1 + \beta)\epsilon^{1/\alpha} < \frac{1}{2}|x - y|$, where $\beta > 0$ is large but fixed; its specific value will be decided on later (see the explanation for \eqref{eq2017-11-19-1} and \eqref{eq2017-11-19-2}). Then
\begin{align*}
\xi^T\gamma_Z\xi \geq I_1 + I_2,
\end{align*}
where
\begin{align*}
I_1 &:= \sum_{k = 1}^d\int_{s - \epsilon}^sdr\int_{\mathbb{R}}dv \, (\varphi_1 + \varphi_2)^2, \quad
I_2 := \sum_{k = 1}^d\int_{(t - \epsilon)\vee s}^tdr\int_{\mathbb{R}}dv \, \varphi_2^2,
\end{align*}
and
\begin{align*}
\varphi_1 &:= \sum_{i = 1}^d(\lambda_i - \mu_i)[G_{\alpha}(s - r, y - v)\sigma_{ik}(u(r, v)) + a_i(k, r, v, s, y)],\\
\varphi_2 &:= \sum_{i = 1}^d[\mu_i G_{\alpha}(t - r, x - v)\sigma_{ik}(u(r, v)) + \mu_ia_i(k, r, v, t, x)].
\end{align*}

From here on, Case 2 is divided into two further sub-cases.

\textit{Sub-Case A}. Suppose, in addition, that $\epsilon \geq \delta(t - s)$, where $\delta$ is chosen as in Case 1. In this sub-case, we are going to prove that
\begin{align}\label{eq2017-09-20-3}
\inf_{\|\xi\| = 1}\xi^T\gamma_Z\xi \geq c\epsilon^{\frac{\alpha - 1}{\alpha}} - Z_{1,\epsilon},
\end{align}
where for all $q \geq 1$,
\begin{align} \label{eq2017-09-20-6}
\mbox{E}\left[|Z_{1,\epsilon}|^q\right] \leq c(q)\epsilon^{\frac{2\alpha - 2}{\alpha}q}.
\end{align}
Indeed, by the elementary inequality (\ref{eq2017-09-06-2}) we find that
\begin{align*}
I_1 \geq \frac{2}{3}\tilde{A}_1 - B_1^{(1)} - B_1^{(2)},
\end{align*}
where
\begin{align}
\tilde{A}_1 &:= \sum_{k = 1}^d\int_{s - \epsilon}^sdr\int_{\mathbb{R}}dv\left(\sum_{i = 1}^d[(\lambda_i - \mu_i)G_{\alpha}(s - r, y - v) + \mu_iG_{\alpha}(t - r, x - v)]\sigma_{ik}(u(r, v))\right)^2, \nonumber\\
B_1^{(1)} &:= 4 \|\lambda - \mu\|^2 \sum_{k = 1}^d\int_{s - \epsilon}^sdr\int_{\mathbb{R}}dv\sum_{i = 1}^da_i^2(k, r, v, s, y),\label{eq2017-09-20-13}\\
B_1^{(2)} &:= 4 \|\mu\|^2 \sum_{k = 1}^d\int_{s - \epsilon}^sdr\int_{\mathbb{R}}dv\sum_{i = 1}^da_i^2(k, r, v, t, x).\label{eq2017-09-20-14}
\end{align}
Using the inequality
$
(a + b)^2 \geq a^2 + b^2 - 2|ab|,
$
we see that
\begin{align*}
\tilde{A}_1 \geq A_1 + A_2 - 2B_1^{(3)},
\end{align*}
where
\begin{align*}
A_1 &:= \sum_{k = 1}^d\int_{s - \epsilon}^sdr\int_{\mathbb{R}}dv\left(\sum_{i = 1}^d(\lambda_i - \mu_i)G_{\alpha}(s - r, y - v) \sigma_{ik}(u(r, v))\right)^2, \\
A_2 &:= \sum_{k = 1}^d\int_{s - \epsilon}^sdr\int_{\mathbb{R}}dv\left(\sum_{i = 1}^d\mu_iG_{\alpha}(t - r, x - v)\sigma_{ik}(u(r, v))\right)^2
\end{align*}
and $B_1^{(3)}$ has the same expression as in \eqref{eq2018-07-09-3}.
We can combine terms to find that
\begin{align*}
I_1 \geq \frac{2}{3}(A_1 + A_2) - (B_1^{(1)} + B_1^{(2)} + 2B_1^{(3)}).
\end{align*}
Moreover, we appeal to the elementary inequality (\ref{eq2017-09-06-2}) to find that
\begin{align*}
I_2 \geq \frac{2}{3}A_3 - B_2,
\end{align*}
where
\begin{align}
A_3 &:= \sum_{k = 1}^d\int_{(t - \epsilon)\vee s}^tdr\int_{\mathbb{R}}dv\left(\sum_{i = 1}^d\mu_i G_{\alpha}(t - r, x - v)\sigma_{ik}(u(r, v)\right)^2, \nonumber \\
B_2 &:= 2\sum_{k = 1}^d\int_{(t - \epsilon)\vee s}^tdr\int_{\mathbb{R}}dv\left(\sum_{i = 1}^d\mu_i a_i(k, r, v, t, x)\right)^2.\label{eq2017-09-20-16}
\end{align}
By hypothesis \textbf{P2} and using \eqref{eq2017-09-05-1} three times,
\begin{align*}
&A_1 + A_2 + A_3 \\
&\quad \geq \rho^2\Big(\|\lambda - \mu\|^{2}\int_{s - \epsilon}^sdr\int_{\mathbb{R}}dv \, G_{\alpha}^2(s - r, y - v) + \|\mu\|^2\int_{s - \epsilon}^sdr\int_{\mathbb{R}}dv \, G_{\alpha}^2(t - r, x - v) \\
&\quad\quad \quad \quad \quad + \|\mu\|^2\int_{(t - \epsilon)\vee s}^tdr\int_{\mathbb{R}}dv \, G_{\alpha}^2(t - r, x - v)\Big)\\
&\quad= c\rho^2\Big(\|\lambda - \mu\|^{2}\epsilon^{\frac{\alpha - 1}{\alpha}} + \|\mu\|^2\left((t - s + \epsilon)^{\frac{\alpha - 1}{\alpha}} - (t - s)^{\frac{\alpha - 1}{\alpha}} + (t - ((t - \epsilon)\vee s))^{\frac{\alpha - 1}{\alpha}}\right) \Big)\\
&\quad= c\rho^2\Big(\|\lambda - \mu\|^{2}\epsilon^{\frac{\alpha - 1}{\alpha}} + \|\mu\|^2\left((t - s + \epsilon)^{\frac{\alpha - 1}{\alpha}} - (t - s)^{\frac{\alpha - 1}{\alpha}} + ((t - s)\wedge \epsilon)^{\frac{\alpha - 1}{\alpha}}\right) \Big)\\
&\quad= c\rho^2\epsilon^{\frac{\alpha - 1}{\alpha}}\left(\|\lambda - \mu\|^{2} + \|\mu\|^2\left((\frac{t - s}{\epsilon} + 1)^{\frac{\alpha - 1}{\alpha}} - (\frac{t - s}{\epsilon})^{\frac{\alpha - 1}{\alpha}} + ((\frac{t - s}{\epsilon})\wedge 1)^{\frac{\alpha - 1}{\alpha}}\right) \right).
\end{align*}
Denote $\zeta(x):= (x + 1)^{\frac{\alpha - 1}{\alpha}} - x^{\frac{\alpha - 1}{\alpha}} + (x\wedge 1)^{\frac{\alpha - 1}{\alpha}}, \, x \in [0, \infty[$. Then it is clear that
\begin{align}\label{eq2017-09-21-1}
\hat{c}_0 := \min\limits_{0 \leq x < \infty}\zeta(x) > 0.
\end{align}
Thus we have
\begin{align*}
A_1 + A_2 + A_3 &\geq c\rho^2\epsilon^{\frac{\alpha - 1}{\alpha}}\Big(\|\lambda - \mu\|^{2} + \hat{c}_0\|\mu\|^2\Big) \geq c'\epsilon^{\frac{\alpha - 1}{\alpha}}.
\end{align*}
We are aiming for (\ref{eq2017-09-20-3}), and will bound the absolute moments of $B_1^{(i)}, i = 1, 2, 3$ and $B_2$, separately. According to Lemma \ref{lemma2017-09-15-1} with $s = t$,
\begin{align}\label{eq2017-09-20-4}
\mbox{E}\left[\sup_{\|\xi\| = 1}|B_2|^q\right] \leq c(q)\epsilon^{\frac{2\alpha - 2}{\alpha}q} \quad \mbox{and} \quad \mbox{E}\left[\sup_{\|\xi\| = 1}|B_1^{(1)}|^q\right] \leq c(q)\epsilon^{\frac{2\alpha - 2}{\alpha}q}.
\end{align}
In the same way, we see that
\begin{align}\label{eq2017-09-20-15}
\mbox{E}\left[\sup_{\|\xi\| = 1}|B_1^{(2)}|^q\right] \leq c(t - s + \epsilon)^{\frac{\alpha - 1}{\alpha}q}\epsilon^{\frac{\alpha - 1}{\alpha}q}.
\end{align}
Since we are in the Sub-case A where $t - s \leq \delta^{-1}\epsilon$, we obtain
\begin{align}\label{eq2017-09-20-5}
\mbox{E}\left[\sup_{\|\xi\| = 1}|B_1^{(2)}|^q\right] \leq c(q)\epsilon^{\frac{2\alpha - 2}{\alpha}q}.
\end{align}
We can combine (\ref{eq2017-09-20-4}) and (\ref{eq2017-09-20-5}) as follows:
\begin{align}
\mbox{E}\left[\sup_{\|\xi\| = 1}\left(B_1^{(1)} + B_1^{(2)}\right)^q\right] \leq c(q)\epsilon^{\frac{2\alpha - 2}{\alpha}q}.
\end{align}
Finally, we turn to bounding the absolute moments of $B_1^{(3)}$. Hypothesis \textbf{P1'} assures us that
\begin{align*}
|B_1^{(3)}| &\leq c\int_{s - \epsilon}^sdr\int_{\mathbb{R}}dv \, G_{\alpha}(s - r, y - v)G_{\alpha}(t - r, x - v)\\
 &= c\int_{s - \epsilon}^sdr \, G_{\alpha}(t + s - 2r, x - y)  = c\int_{0}^{\epsilon}dr \, G_{\alpha}(t - s + 2r, x - y),
\end{align*}
thanks to the semi-group property.

When $\alpha = 2$, we can follow the arguments of \cite[p.414]{DKN09} to find that
\begin{align}
\left|B_1^{(3)}\right| \leq c\epsilon^{1/2}\Psi(\beta), \quad \mbox{where} \quad \Psi(\beta) := \beta\int_0^{6/\beta^2}z^{-1/2}e^{-1/z}dz.
\end{align}
Thus,
\begin{align*}
\inf_{\|\xi\| = 1}\xi^T\gamma_Z\xi &\geq \frac{2}{3}(A_1 + A_2 + A_3) - \left(B_1^{(1)} + B_1^{(2)} + \left|B_1^{(3)}\right| + B_2\right)\\
 &\geq c_1\epsilon^{1/2} - c_2\Psi(\beta)\epsilon^{1/2} - Z_{1, \epsilon},
\end{align*}
where $Z_{1, \epsilon} := B_1^{(1)} + B_1^{(2)} + B_2$ satisfies $\mbox{E}[|Z_{1, \epsilon}|^q] \leq c_1(q)\epsilon^q$.
Because $\lim_{\nu \rightarrow \infty}\Psi(\nu) = 0$, we can choose $\beta$ so large that $c_2\Psi(\beta) \leq c_1/4$ for the $c_1$ and $c_2$ of the preceding displayed equation. This yields,
\begin{align} \label{eq2017-11-19-1}
\inf_{\|\xi\| = 1}\xi^T\gamma_Z\xi &\geq c\epsilon^{1/2} - Z_{1, \epsilon}.
\end{align}

When $1 < \alpha < 2$, by the scaling property of the Green kernel \cite[Lemma 4.1(iv)]{ChD15}, and \cite[Lemma 4.1(vi)]{ChD15}, we have
\begin{align*}
\left|B_1^{(3)}\right| &\leq c\int_{0}^{\epsilon}dr (t - s + 2r)^{-1/\alpha}G_{\alpha}(1, (x - y)(t - s + 2r)^{-1/\alpha})\\
&\leq  cK_{\alpha}\int_{0}^{\epsilon}\frac{(t - s + 2r)^{-1/\alpha}}{1 + \left|(x - y)(t - s + 2r)^{-1/\alpha}\right|^{1 + \alpha}} dr\\
&\leq  cK_{\alpha}\int_{0}^{\epsilon}\frac{(t - s + 2r)^{-1/\alpha}}{\left|(x - y)(t - s + 2r)^{-1/\alpha}\right|^{1 + \alpha}} dr\\
&= cK_{\alpha}|x - y|^{-1 - \alpha}\int_{0}^{\epsilon}(t - s + 2r)dr = cK_{\alpha}|x - y|^{-1 - \alpha}[(t - s)\epsilon + \epsilon^2].
\end{align*}
Since $t - s \leq |x - y|^{\alpha}$ and $(1 + \beta)\epsilon^{1/\alpha} < \frac{1}{2}|x - y|$ (since we are in Case 2), this is bounded above by
\begin{align*}
& cK_{\alpha}(|x - y|^{-1}\epsilon + |x - y|^{-1 - \alpha}\epsilon^2)\\
&\quad \leq cK_{\alpha}\left(\frac{1}{(1 + \beta)}\epsilon^{\frac{\alpha - 1}{\alpha}} + \frac{1}{(1 + \beta)^{1 + \alpha}}\epsilon^{2 - (1 + \alpha)/\alpha}\right)\\
&\quad = cK_{\alpha}\left(\frac{1}{(1 + \beta)}+ \frac{1}{(1 + \beta)^{1 + \alpha}}\right)\epsilon^{\frac{\alpha - 1}{\alpha}}.
\end{align*}
Therefore, for $1 < \alpha \leq 2$, we can choose and fix $\beta$ large enough so that
\begin{align}\label{eq2017-11-19-2}
\inf_{\|\xi\| = 1}\xi^T\gamma_Z\xi &\geq c\epsilon^{\frac{\alpha - 1}{\alpha}} - Z_{1, \epsilon},
\end{align}
where for all $q \geq 1$,
\begin{align*}
\mbox{E}\left[|Z_{1,\epsilon}|^q\right] \leq c(q)\epsilon^{\frac{2\alpha - 2}{\alpha}q},
\end{align*}
as in (\ref{eq2017-09-20-3}) and (\ref{eq2017-09-20-6}).

\textit{Sub-case B}. In this final (sub-) case we suppose that $\epsilon < \delta(t - s) \leq \delta|x - y|^{\alpha}$. Choose and fix $0 < \epsilon < \delta(t - s)$. During the course of our proof of Case 1, we established the following:
\begin{align}\label{eq2017-09-20-8}
\inf_{\|\xi\| = 1}\xi^T\gamma_Z\xi &\geq c\epsilon^{\frac{\alpha - 1}{\alpha}} - \tilde{Z}_{\epsilon},
\end{align}
where, for all $q \geq 1$,
\begin{align*}
\mbox{E}\left[|\tilde{Z}_{\epsilon}|^q\right] \leq c\epsilon^{\min((1 + \gamma_0)\frac{\alpha - 1}{\alpha}, 1 - \frac{\gamma_0}{\alpha})q}
\end{align*}
(see (\ref{eq2017-09-20-7}) and (\ref{eq2017-09-20-9})). This inequality remains valid in this Sub-case B.

Combine Sub-Cases A and B, and, in particular, (\ref{eq2017-09-20-3}) and (\ref{eq2017-09-20-8}), to find that for all $0 < \epsilon < 2^{-\alpha}(1 + \beta)^{-\alpha}|x - y|^{\alpha}$,
\begin{align*}
\inf_{\|\xi\| = 1}\xi^T\gamma_Z\xi &\geq c\epsilon^{\frac{\alpha - 1}{\alpha}} -(\tilde{Z}_{\epsilon}\mathbf{1}_{\{\epsilon < \delta(t - s)\}} + Z_{1, \epsilon}\mathbf{1}_{\{t - s \leq \delta^{-1}\epsilon\}}).
\end{align*}
Because of this and (\ref{eq2017-09-20-6}), by \cite[Proposition 3.5]{DKN09}, this implies that
\begin{align*}
\mbox{E}\left[\left(\inf_{ \|\xi\| = 1}\xi^T\gamma_Z\xi\right)^{-2pd}\right] &\leq c|x - y|^{\alpha(-2dp)(\frac{\alpha - 1}{\alpha})}\\
&\leq c\left(|x - y|^{\alpha} + |t - s|\right)^{(\frac{\alpha - 1}{\alpha})(-2dp)}\\
&\leq c\left(|t - s|^{\frac{\alpha - 1}{\alpha}} + |x - y|^{\alpha - 1} \right)^{-2dp}.
\end{align*}
This completes the proof of Proposition \ref{prop2017-09-19-4}.
\end{proof}

\begin{remark}\label{remark2017-11-21-4}
From the proof of Proposition \ref{prop2017-09-19-4}, we see that \eqref{eq2017-09-19-5} is also valid for the solutions of stochastic heat equations with Neumann or Dirichlet boundary conditions, since we can still apply the result of Lemma \ref{lemma2017-09-19-1}; see Remark \ref{remark2017-11-21-3}.
\end{remark}

\begin{proof}[Proof of Proposition \ref{prop2017-09-19-5}]
The proof follows along the same lines as those of \cite[Proposition 6.13]{DKN09}.
Let $0 < \epsilon < s \leq t$. We fix $i_0 \in \{1, \ldots, 2d\}$ and write $\tilde{\lambda}^{i_0} = (\tilde{\lambda}_1^{i_0}, \ldots, \tilde{\lambda}_d^{i_0})$ and $\tilde{\mu}^{i_0} = (\tilde{\mu}_1^{i_0}, \ldots, \tilde{\mu}_d^{i_0})$. We look at $(\xi^{i_0})^T\gamma_Z\xi^{i_0}$ on the event $\{\beta_{i_0} \geq \beta_0\}$. As in the proof of Proposition \ref{prop2017-09-19-4} and using the notation from (\ref{eq2017-09-19-3}), this is bounded below by
\begin{align}\label{eq2017-09-20-10}
&  \sum_{k = 1}^d\int_{s - \epsilon}^sdr\int_{\mathbb{R}}dv\Bigg(\sum_{i = 1}^d\Big[\Big(\beta_{i_0}\tilde{\lambda}_i^{i_0}G_{\alpha}(s - r, y - v) \nonumber\\
& \quad\quad + \tilde{\mu}_i^{i_0}\sqrt{1 - \beta_{i_0}^2}(G_{\alpha}(t - r, x - v) - G_{\alpha}(s - r, y - v))\Big)\sigma_{ik}(u(r, v))\nonumber\\
& \quad \quad+ \beta_{i_0}\tilde{\lambda}_i^{i_0}a_i(k, r, v, s, y)\nonumber\\
& \quad \quad+ \tilde{\mu}_i^{i_0}\sqrt{1 - \beta_{i_0}^2}(a_i(k, r, v, t, x) - a_i(k, r, v, s, y))\Big]\Bigg)^2\nonumber\\
&  \quad+ \sum_{k = 1}^d\int_{s \vee (t - \epsilon)}^tdr\int_{\mathbb{R}}dv\Bigg(\sum_{i = 1}^d\Big[\tilde{\mu}_i^{i_0}\sqrt{1 - \beta_{i_0}^2}G_{\alpha}(t - r, x - v)\sigma_{ik}(u(r, v))\nonumber\\
& \quad \quad + \tilde{\mu}_i^{i_0}\sqrt{1 - \beta_{i_0}^2}a_i(k, r, v, t, x)\Big]\Bigg)^2.
\end{align}
We seek lower bounds for this expression for $0 < \epsilon < \epsilon_0$ where $\epsilon_0 \in \, ]0, \frac{1}{2}[$ is fixed.
In the remainder of this proof, we will use the generic notation $\beta, \tilde{\lambda}$ and $\tilde{\mu}$ for the realizations $\beta_{i_0}(\omega), \tilde{\lambda}^{i_0}(\omega)$, and $\tilde{\mu}^{i_0}(\omega)$.

\textbf{Case 1} $t - s \leq \epsilon$. Then, by the elementary inequality (\ref{eq2017-09-06-2}), the expression in (\ref{eq2017-09-20-10}) is bounded below by
\begin{align*}
\frac{2}{3}(f_1(s, t, \epsilon, \beta, \tilde{\lambda}, \tilde{\mu}, x, y) + f_2(s, t, \epsilon,  \beta, \tilde{\lambda}, \tilde{\mu}, x, y)) - 2I_{\epsilon},
\end{align*}
where, from hypothesis \textbf{P2},
\begin{align}
f_1 &\geq c\rho^2\int_{s - \epsilon}^sdr\int_{\mathbb{R}}dv\Big\|\beta\tilde{\lambda}G_{\alpha}(s - r, y - v) \nonumber \\
& \quad \quad \quad \quad \quad \quad  \quad \quad \quad + \sqrt{1 - \beta^2}\tilde{\mu}(G_{\alpha}(t - r, x - v) - G_{\alpha}(s - r, y - v))\Big\|^2, \label{eq2017-09-20-11}\\
f_2 &\geq c\rho^2\int_{s \vee (t - \epsilon)}^tdr\int_{\mathbb{R}}dv\left\|\tilde{\mu}\sqrt{1 - \beta^2}G_{\alpha}(t - r, x - v)\right\|^2\label{eq2017-09-20-12}
\end{align}
and $I_{\epsilon} = 3(I_{1, \epsilon} + I_{2, \epsilon} + I_{3, \epsilon})$, where
\begin{align*}
I_{1, \epsilon} &:= \sum_{k = 1}^d\int_{s - \epsilon}^sdr\int_{\mathbb{R}}dv\left(\sum_{i = 1}^d\left[\beta\tilde{\lambda}_i - \tilde{\mu}_i\sqrt{1 - \beta^2}\right]a_i(k, r, v, s, y)\right)^2,\\
I_{2, \epsilon} &:= \sum_{k = 1}^d\int_{s - \epsilon}^sdr\int_{\mathbb{R}}dv\left(\sum_{i = 1}^d \tilde{\mu}_i\sqrt{1 - \beta^2}a_i(k, r, v, t, x)\right)^2,\\
I_{3, \epsilon} &:= \sum_{k = 1}^d\int_{t - \epsilon}^tdr\int_{\mathbb{R}}dv\left(\sum_{i = 1}^d \tilde{\mu}_i\sqrt{1 - \beta^2}a_i(k, r, v, t, x)\right)^2.
\end{align*}
There are obvious similarities between the term $I_{1, \epsilon}$ and $B_1^{(1)}$ in (\ref{eq2017-09-20-13}). However, we must keep in mind that $\beta, \tilde{\lambda}$ and $\tilde{\mu}$ are the realizations of $\beta_{i_0}, \tilde{\lambda}^{i_0}$, and $\tilde{\mu}^{i_0}$. Therefore,
\begin{align*}
I_{1, \epsilon} &:= \sum_{k = 1}^d\int_{s - \epsilon}^sdr\int_{\mathbb{R}}dv\left(\sum_{i = 1}^d\left[\beta\tilde{\lambda}_i - \tilde{\mu}_i\sqrt{1 - \beta^2}\right]a_i(k, r, v, s, y)\right)^2\\
 &\leq C\sum_{k = 1}^d\int_{s - \epsilon}^sdr\int_{\mathbb{R}}dv\sum_{i = 1}^da_i^2(k, r, v, s, y).
\end{align*}
Then, we apply the same method that was used to bound $\mbox{E}[|B_1^{(1)}|^q]$ to deduce that $\mbox{E}[|I_{1, \epsilon}|^q] \leq c(q)\epsilon^{\frac{2\alpha - 2}{\alpha}q}$. Similarly, since $I_{2, \epsilon}$ is similar to $B_1^{(2)}$ from \eqref{eq2017-09-20-14} and $t - s \leq \epsilon$, we see using \eqref{eq2017-09-20-15} that $\mbox{E}[|I_{2, \epsilon}|^q] \leq c(q)\epsilon^{\frac{2\alpha - 2}{\alpha}q}$. Finally, using the similarity between $I_{3, \epsilon}$ and $B_2$ in \eqref{eq2017-09-20-16}, we see that $\mbox{E}[|I_{3, \epsilon}|^q] \leq c(q)\epsilon^{\frac{2\alpha - 2}{\alpha}q}$.

We claim that for every $\beta_0 > 0$, there exists $\epsilon_0 > 0$ and $c_0 > 0$ such that
\begin{align} \label{eq2017-09-21-3}
f_1 + f_2 \geq c_0\epsilon^{\frac{\alpha - 1}{\alpha}} \ \mbox{for  all} \  \beta \in [\beta_0, 1], \, \epsilon \in \, ]0, \epsilon_0], \, s, t \in [0, T], \, x, y \in \mathbb{R}.
\end{align}
Using this for the $\beta_0$ from \cite[Lemma 6.8]{DKN09} with $\alpha_0$ there replace by $\beta_0$, this will imply in particular that for $\epsilon \geq t - s$,
\begin{align}
(\xi^{i_0})^T\gamma_Z\xi^{i_0} \geq c_0\epsilon^{\frac{\alpha - 1}{\alpha}} - 2I_{\epsilon},
\end{align}
where $\mbox{E}[|I_{\epsilon}|^q] \leq c(q)\epsilon^{\frac{2\alpha - 2}{\alpha}q}.$

We now prove \eqref{eq2017-09-21-3}. Because $\|\tilde{\lambda}\| = \|\tilde{\mu}\| = 1$, $f_1$ is bounded below by
\begin{align*}
&  c\rho^2\int_{s - \epsilon}^sdr\int_{\mathbb{R}}dv\Big(\beta^2G_{\alpha}^2(s - r, y - v)) + \left(1 - \beta^2\right)(G_{\alpha}(t - r, x - v) - G_{\alpha}(s - r, y - v))^2 \\
& \quad \quad \quad \quad\quad \quad   + 2\beta\sqrt{1 - \beta^2}G_{\alpha}(s - r, y - v))(G_{\alpha}(t - r, x - v) - G_{\alpha}(s - r, y - v))(\tilde{\lambda}\cdot\tilde{\mu})\Big)\\
& \quad =  c\rho^2\int_{s - \epsilon}^sdr\int_{\mathbb{R}}dv\Big(\Big(\beta - \sqrt{1 - \beta^2}\Big)^2G_{\alpha}^2(s - r, y - v) + \left(1 - \beta^2\right)G_{\alpha}^2(t - r, x - v) \\
&  \quad \quad  \quad\quad \quad   + 2\Big(\beta - \sqrt{1 - \beta^2}\Big)\sqrt{1 - \beta^2}G_{\alpha}(s - r, y - v)G_{\alpha}(t - r, x - v)\\
&  \quad \quad  \quad\quad \quad  + 2\beta\sqrt{1 - \beta^2}G_{\alpha}(s - r, y - v))(G_{\alpha}(t - r, x - v) - G_{\alpha}(s - r, y - v))(\tilde{\lambda}\cdot\tilde{\mu} - 1)\Big).
\end{align*}
By the semi-group property \cite[Lemma 4.1(iii)]{ChD15}, we set $h := t - s$ and change the variables to obtain the following bound:
\begin{align*}
f_1 &\geq c\rho^2\int_{0}^{\epsilon}dr\Big(\Big(\beta - \sqrt{1 - \beta^2}\Big)^2G_{\alpha}(2r, 0) + \left(1 - \beta^2\right)G_{\alpha}(2h + 2r, 0) \\
&  \qquad\quad \quad + 2\Big(\beta - \sqrt{1 - \beta^2}\Big)\sqrt{1 - \beta^2}G_{\alpha}(h + 2r, x - y)\\
&  \qquad\quad \quad + 2\beta\sqrt{1 - \beta^2}(G_{\alpha}(h + 2r, x - y) - G_{\alpha}(2r, 0))(\tilde{\lambda}\cdot\tilde{\mu} - 1)\Big).
\end{align*}
Since by the scaling property of Green kernel \cite[Lemma 4.1(iv)]{ChD15}, and \cite[Lemma 4.1(ii)]{ChD15},
\begin{align*}
G_{\alpha}(h + 2r, x - y)& = (h + 2r)^{-1/\alpha}G_{\alpha}(1, (h + 2r)^{-1/\alpha}(x - y)) \\
&\leq  (h + 2r)^{-1/\alpha}G_{\alpha}(1, 0) \leq (2r)^{-1/\alpha}G_{\alpha}(1, 0) = G_{\alpha}(2r, 0),
\end{align*}
 together with $\tilde{\lambda}\cdot\tilde{\mu} - 1 \leq 0$, we see that
\begin{align*}
f_1 \geq c\rho^2\hat{g}_1,
\end{align*}
where
\begin{align*}
\hat{g}_1 &:= \hat{g}_1(h, \epsilon, \beta, x, y)\\
 &= \int_{0}^{\epsilon}dr\Big(\Big(\beta - \sqrt{1 - \beta^2}\Big)^2G_{\alpha}(2r, 0) + \left(1 - \beta^2\right)G_{\alpha}(2h + 2r, 0) \\
 &  \quad + 2\Big(\beta - \sqrt{1 - \beta^2}\Big)\sqrt{1 - \beta^2}G_{\alpha}(h + 2r, x - y)\Big).
\end{align*}
Therefore,
\begin{align*}
\hat{g}_1 &= \int_{0}^{\epsilon}dr\Big(\Big(\beta - \sqrt{1 - \beta^2}\Big)^2r^{-\frac{1}{\alpha}}2^{-\frac{1}{\alpha}}G_{\alpha}(1, 0) + \left(1 - \beta^2\right)(h + r)^{-\frac{1}{\alpha}}2^{-\frac{1}{\alpha}}G_{\alpha}(1, 0)\\
& \quad \quad \quad + 2\Big(\beta - \sqrt{1 - \beta^2}\Big)\sqrt{1 - \beta^2}G_{\alpha}(h + 2r, x - y)\Big).
\end{align*}
On the other hand,
\begin{align*}
f_2 &\geq c\rho^2\int_0^{\epsilon \wedge (t - s)}dr \, (1 - \beta^2)G_{\alpha}(2r, 0) = c\rho^2\hat{g}_2,
\end{align*}
where
\begin{align*}
\hat{g}_2 &:= \int_0^{\epsilon \wedge h}dr \, (1 - \beta^2)G_{\alpha}(2r, 0) = (1 - \beta^2)\frac{\alpha}{\alpha - 1}2^{-\frac{1}{\alpha}}G_{\alpha}(1, 0)(\epsilon \wedge h)^{\frac{\alpha - 1}{\alpha}}.
\end{align*}
Finally, we conclude that
\begin{align}  \label{eq2017-11-19-3}
f_1 + f_2 &\geq c\rho^2(\hat{g}_1 + \hat{g}_2)\nonumber \\
&= c\rho^2\Bigg(\frac{\alpha}{\alpha - 1}2^{-\frac{1}{\alpha}}G_{\alpha}(1, 0)\bigg(\Big(\beta - \sqrt{1 - \beta^2}\Big)^2\epsilon^{\frac{\alpha - 1}{\alpha}} \nonumber\\
& \quad \quad \quad \quad  + \left(1 - \beta^2\right)\left((h + \epsilon)^{\frac{\alpha - 1}{\alpha}} - h^{\frac{\alpha - 1}{\alpha}} + (\epsilon \wedge h)^{\frac{\alpha - 1}{\alpha}}\right)\bigg)\nonumber\\
&  \quad \quad \quad \quad + 2\Big(\beta - \sqrt{1 - \beta^2}\Big)\sqrt{1 - \beta^2}\int_0^{\epsilon}dr\, G_{\alpha}(h + 2r, x - y)\Bigg).
\end{align}
Now we consider two different sub-cases.

\textbf{Sub-case (i).} Suppose $\beta - \sqrt{1 - \beta^2} \geq 0$, that is, $\beta \geq 2^{-1/2}$. Then
\begin{align*}
\epsilon^{-\frac{\alpha - 1}{\alpha}}(\hat{g}_1 + \hat{g}_2) \geq \phi_1\left(\beta, \frac{h}{\epsilon}\right),
\end{align*}
where
\begin{align*}
\phi_1(\beta, z) &:= \frac{\alpha2^{-\frac{1}{\alpha}}}{\alpha - 1}G_{\alpha}(1, 0)\Big((\beta - \sqrt{1 - \beta^2})^2 + \left(1 - \beta^2\right)\left((z + 1)^{\frac{\alpha - 1}{\alpha}} - z^{\frac{\alpha - 1}{\alpha}} + (z \wedge 1)^{\frac{\alpha - 1}{\alpha}}\right)\Big).
\end{align*}
Clearly,
\begin{align*}
\inf_{\beta \geq 2^{-1/2}}\inf_{z \geq 0}\phi_1(\beta, z) &\geq \inf_{\beta \geq 2^{-1/2}}\frac{\alpha}{\alpha - 1}2^{-\frac{1}{\alpha}}G_{\alpha}(1, 0)\left(\left(\beta - \sqrt{1 - \beta^2}\right)^2 + \hat{c}_0\left(1 - \beta^2\right)\right)\\
&> \phi_0 > 0,
\end{align*}
where the value of $\hat{c}_0$ is specified in (\ref{eq2017-09-21-1}).
Thus,
\begin{align*}
\inf_{\beta \geq 2^{-1/2}, h \geq 0, 0 < \epsilon \leq \epsilon_0} \epsilon^{-\frac{\alpha - 1}{\alpha}}(\hat{g}_1 + \hat{g}_2) > 0.
\end{align*}

\textbf{Sub-case (ii).} Now we consider the case where $\beta - \sqrt{1 - \beta^2} < 0$, that is, $\beta < 2^{-1/2}$. In this case, from \eqref{eq2017-11-19-3}, we see that
\begin{align*}
\epsilon^{-\frac{\alpha - 1}{\alpha}}(\hat{g}_1 + \hat{g}_2) \geq \psi_1\left(\beta, \frac{h}{\epsilon}\right),
\end{align*}
where
\begin{align*}
\psi_1(\beta, z) &:= \frac{\alpha2^{-\frac{1}{\alpha}}}{\alpha - 1}G_{\alpha}(1, 0)\bigg((\beta - \sqrt{1 - \beta^2})^2  + (1 - \beta^2)\Big((z + 1)^{\frac{\alpha - 1}{\alpha}} - z^{\frac{\alpha - 1}{\alpha}} + (z \wedge 1)^{\frac{\alpha - 1}{\alpha}}\Big)\\
&  \quad \quad \quad \quad \quad \quad \quad \quad \quad \quad - 2(\sqrt{1 - \beta^2} - \beta)\sqrt{1 - \beta^2}\Big(\Big(\frac{z}{2} + 1\Big)^{\frac{\alpha - 1}{\alpha}}- \Big(\frac{z}{2}\Big)^{\frac{\alpha - 1}{\alpha}}\Big)\bigg).
\end{align*}
We observe that $\psi_1(\beta, z) > 0$ if $\beta \neq 0$ (this observation is similar to that in the lines following \cite[(6.39)]{DKN09}). Denote $c_{\alpha} := \frac{\alpha}{\alpha - 1}2^{-\frac{1}{\alpha}}G_{\alpha}(1, 0)$. For $z \geq 1$, we have
\begin{align*}
\psi_1(\beta, z) &\geq c_{\alpha}\bigg[(\beta - \sqrt{1 - \beta^2})^2 + (1 - \beta^2) - 2(\sqrt{1 - \beta^2} - \beta)\sqrt{1 - \beta^2}\Big(\Big(\frac{3}{2}\Big)^{\frac{\alpha - 1}{\alpha}}- \Big(\frac{1}{2}\Big)^{\frac{\alpha - 1}{\alpha}}\Big)\bigg]\\
 &\geq c_{\alpha}\bigg(1 - \Big(\frac{3}{2}\Big)^{\frac{\alpha - 1}{\alpha}} + \Big(\frac{1}{2}\Big)^{\frac{\alpha - 1}{\alpha}}\bigg)\Big[\Big(\beta - \sqrt{1 - \beta^2}\Big)^2 + \Big(1 - \beta^2\Big)\Big] \geq \bar{c}_0,
\end{align*}
where in the second inequality we use the elementary inequality $2ab \leq a^2 + b^2$.
Then
\begin{align*}
\inf_{\beta \in [\beta_0, 2^{-1/2}]}\inf_{z \geq 0}\psi_1(\beta, z) &\geq \min\left\{\bar{c}_0, \inf_{\beta \in [\beta_0, 2^{-1/2}]}\inf_{z \in [0, 1]}\psi_1(\beta, z)\right\} \geq c_{\beta_0} > 0.
\end{align*}
This concludes the proof of the claim \eqref{eq2017-09-21-3}.

\textbf{Case 2} $t - s > \epsilon$. Choose and fix $\eta > 0$. Following the same lines as in \cite[ p.424-425]{DKN09}, we see that, when $t - s > \epsilon$,
\begin{align*}
1_{\{\beta_{i_0} \geq \beta_0\}}(\xi^{i_0})^T\gamma_Z\xi^{i_0} \geq 1_{\{\beta_{i_0} \geq \beta_0\}} \min\left(c\rho^2\epsilon^{\frac{\alpha - 1}{\alpha} + \eta} - 2I_{3, \epsilon}, c\epsilon^{\frac{\alpha - 1}{\alpha}} - \tilde{J}_{\epsilon}\right),
\end{align*}
where
\begin{align*}
\mbox{E}[|I_{3, \epsilon}|^q] \leq c(q)\epsilon^{\frac{2\alpha - 2}{\alpha}q} \quad \mbox{and} \quad \mbox{E}[|\tilde{J}_{\epsilon}|^q] \leq c(q)\epsilon^{(\frac{\alpha - 1}{\alpha} + \eta)q}.
\end{align*}

Putting together the results of Cases 1 and 2, we see that for $0 < \epsilon \leq \epsilon_0$,
\begin{align*}
1_{\{\beta_{i_0} \geq \beta_0\}}(\xi^{i_0})^T\gamma_Z\xi^{i_0} \geq 1_{\{\beta_{i_0} \geq \beta_0\}}Z,
\end{align*}
where
\begin{align*}
Z = \min\left(c\rho^2\epsilon^{\frac{\alpha - 1}{\alpha} + \eta} - 2I_{3, \epsilon}, c\epsilon^{\frac{\alpha - 1}{\alpha}} - 2I_{ \epsilon}\mathbf{1}_{\{\epsilon \geq t - s\}} - \tilde{J}_{\epsilon}\mathbf{1}_{\{\epsilon < t - s\}}\right).
\end{align*}
Note that all the constants are independent of $i_0$. Taking into account the bounds on moments of $I_{3, \epsilon}, I_{ \epsilon}$ and $\tilde{J}_{\epsilon}$, and then using \cite[Proposition 3.5]{DKN09}, we deduce that for all $p \geq 1$, there is $C > 0$ such that
\begin{align*}
\mbox{E}\left[1_{\{\beta_{i_0} \geq \beta_0\}}\left((\xi^{i_0})^T\gamma_Z\xi^{i_0}\right)^{-p}\right] \leq \mbox{E}\left[1_{\{\beta_{i_0} \geq \beta_0\}}Z^{-p}\right] \leq \mbox{E}[Z^{-p}] \leq C.
\end{align*}
Since this applies to any $p \geq 1$, we can use H\"{o}lder's inequality to deduce (\ref{eq2017-09-19-6}). This proves Proposition \ref{prop2017-09-19-5}.
\end{proof}

\subsection{Proof of Theorem \ref{theorem2017-08-17-1}(b) and Remark \ref{remark2017-11-20-1}(b)}\label{section2.5.2}

Fix two compact intervals $I$ and $J$ as in Theorem \ref{theorem2017-08-17-1}. Let $(s, y), (t, x) \in I \times J, s \leq t, (s, y) \neq (t, x)$, and $z_1, z_2 \in \mathbb{R}^d$. Let $Z$ be as in (\ref{eq2017-09-15-15}) and let $p_Z$ be the density of $Z$. Then
\begin{align*}
p_{s, y; t, x}(z_1, z_2) = p_Z(z_1, z_2 - z_1).
\end{align*}
Use Corollary \ref{cor2017-08-24-1} with $\sigma = \{i \in \{1, \ldots, d\}: z_2^i - z_1^i \geq 0\}$ and H\"{o}lder's inequality to see that
\begin{align}\label{eq2017-09-21-5}
p_Z(z_1, z_2 - z_1) &\leq \prod_{i = 1}^d\left(\mbox{P}\left\{|u_i(t, x) - u_i(s, y)| > |z_1^i - z_2^i|\right\}\right)^{\frac{1}{2d}}\nonumber \\
& \quad \times \|H_{(1, \ldots, 2d)}(Z, 1)\|_{0, 2}.
\end{align}
Therefore, in order to prove the desired results of Theorem \ref{theorem2017-08-17-1}(b) and Remark \ref{remark2017-11-20-1}(b), it suffices to prove that:
\begin{align}
& \|H_{(1, \ldots, 2d)}(Z, 1)\|_{0, 2} \leq c_T(|t - s|^{\frac{\alpha - 1}{\alpha}} + |x - y|^{\alpha - 1})^{-d/2},\label{eq2017-09-21-7}
\end{align}
and
\begin{align}
\prod_{i = 1}^d\left(\mbox{P}\left\{|u_i(t, x) - u_i(s, y)| > |z_1^i - z_2^i|\right\}\right)^{\frac{1}{2d}} \leq c \exp\left(-\frac{\|z_1 - z_2\|^2}{c_T(|t - s|^{\frac{\alpha - 1}{\alpha}} + |x - y|^{\alpha - 1})}\right)\label{eq2017-09-21-6}
\end{align}
under the hypothesis \textbf{P1}, and
\begin{align}
\prod_{i = 1}^d\left(\mbox{P}\left\{|u_i(t, x) - u_i(s, y)| > |z_1^i - z_2^i|\right\}\right)^{\frac{1}{2d}} \leq c \left[\frac{|t - s|^{\frac{\alpha - 1}{\alpha}} + |x - y|^{\alpha - 1}}{\|z_1 - z_2\|^2} \wedge 1\right]^{p/(4d)}\label{eq2017-09-21-60}
\end{align}
under the hypothesis \textbf{P1'}.

The proof of (\ref{eq2017-09-21-6}) under the hypothesis \textbf{P1} is essentially the same as that of \cite[(6.2)]{DKN09}, with $\Delta$ there replaced by $\Delta^2_{\alpha}$, by using Lemma \ref{lemma2017-09-13-1}, the exponential martingale inequality \cite[(A.5)]{Nua06} and Girsanov's theorem. As for the proof of \eqref{eq2017-09-21-60} under the hypothesis \textbf{P1'}, it is analogous to that of \cite[Theorem 1.6(b)]{DKN13}, with $\frac{\gamma}{2}$ there replaced by $\frac{\alpha - 1}{\alpha}$ and $\gamma$ there replaced by $\alpha - 1$. Details can be found in \cite[Section 2.5.3]{Pu18}.

We turn to proving (\ref{eq2017-09-21-7}), which requires the following estimate on the inverse of the matrix $\gamma_Z$.

\begin{theorem}\label{th2017-09-22-1}
Fix $T > 0$.  Assume \textbf{P1'} and \textbf{P2}. Let $I$ and $J$ be compact intervals as in Theorem \ref{theorem2017-08-17-1}. For any $(s, y), (t, x) \in I \times J, s \leq t, (s, y) \neq (t, x), k \geq 0$ and  $p > 1$,
\begin{align}\label{eq2017-09-22-1}
\mbox{E}\left[\|(\gamma_Z)^{-1}_{m, l}\|_{k, p}\right]
 \leq
\left\{\begin{array}{ll}
    c_{k, p, T} & \hbox{if $(m, l) \in (\mathbf{1})$,} \\
    c_{k, p, T} (|t - s|^{\frac{\alpha - 1}{\alpha}} + |x - y|^{\alpha - 1})^{-\frac{1}{2}}   & \hbox{if $(m, l) \in (\mathbf{2})$ or $(\mathbf{3})$,} \\
    c_{k, p, T} (|t - s|^{\frac{\alpha - 1}{\alpha}} + |x - y|^{\alpha - 1})^{-1} & \hbox{if $(m, l) \in (\mathbf{4})$.}
  \end{array}
\right.
\end{align}
\end{theorem}
\begin{proof}
As in the proof of \cite[Theorem 6.3]{DKN09}, we shall use Propositions \ref{prop2017-09-19-1}, \ref{prop2017-09-19-2} and \ref{prop2017-09-19-3}.

When $k = 0$, the result is a consequence of the estimates of Propositions \ref{prop2017-09-19-1} and \ref{prop2017-09-19-3}, using the fact that the inverse of a matrix is the inverse of its determinant multiplied by its cofactor matrix. Comparing to the proof of \cite[Theorem 6.3(a)]{DKN09}, the extra exponent $\eta$ does not appear due to the optimal estimate of negative moments of $\det \gamma_Z$ in Proposition \ref{prop2017-09-19-3}.

For $k \geq 1$, we proceed recursively as in the proof of \cite[Theorem 6.3]{DKN09}, using Proposition  \ref{prop2017-09-19-2} instead of Proposition \ref{prop2017-09-19-1}.
\end{proof}

\begin{proof}[Proof of (\ref{eq2017-09-21-7})]
The proof is similar to that of \cite[(6.3)]{DKN09} by using the continuity of the Skorohod integral $\delta$ (see \cite[Proposition 3.2.1]{Nua06} and \cite[(1.11) and p.131]{Nua98}) and H\"{o}lder's inequality for Malliavin norms (see \cite[Proposition 1.10, p.50]{Wat84}); the main difference is that $\Delta$ there is replaced by $\Delta^2_{\alpha}$. Comparing with the estimate in \cite[(6.3)]{DKN09}, we are able to remove the extra exponent $\eta$ because of the correct estimate on the inverse of the matrix $\gamma_Z$ in Theorem \ref{th2017-09-22-1}.
\end{proof}

\begin{remark}\label{remark2017-11-21-6}
We conclude this section by remarking that \eqref{eq2017-09-21-7} is also valid for the solutions of stochastic heat equations with Neumann or Dirichlet boundary conditions, since the result of Theorem \ref{th2017-09-22-1} is true in that case by applying Proposition \ref{prop2017-09-19-3}; see Remark \ref{remark2017-11-21-5}.
\end{remark}

\section{Proof of Theorems \ref{theorem2017-08-18-1} and \ref{theorem2017-08-18-2}} \label{section1111}

In this section, we give the proof of Theorems \ref{theorem2017-08-18-1} and \ref{theorem2017-08-18-2}. The organization of the proof is similar to \cite[Section 2.3]{DKN13}.

\subsection{Proof of Theorem \ref{theorem2017-08-18-1}: upper bounds}

For all positive integers $n$, set
\begin{align*}
t_k^n := k2^{-\frac{2n\alpha}{\alpha - 1}},  \quad x_l^n := l2^{-\frac{2n}{\alpha - 1}}
\end{align*}and
\begin{align*}
I_k^n = [t_k^n, t_{k + 1}^n], \quad  J_l^n = [x_l^n, x_{l + 1}^n], \quad  R_{k, l}^n = I_k^n \times J_l^n.
\end{align*}
By \eqref{eq2017-08-21-1}, we have
\begin{align}\label{eq2017-08-21-2}
\mbox{E}\Big[\sup_{(t, x) \in R_{k, l}^n}\|u(t, x) - u(t_k^n, x_l^n)\|^p\Big] \leq C2^{-n\beta p},
\end{align}
where $\beta$ is chosen as in \eqref{eq2017-08-21-1}.

\begin{lemma} \label{lemma2017-08-18-1}
Fix $\eta > 0$. There exists $c > 0$ such that for all $z \in \mathbb{R}^d$, $n$ large and $R_{k, l}^n \subset I \times J$,
\begin{align}\label{eq2017-08-18-4}
\mbox{P}\left\{u(R_{k, l}^n) \cap B(z, 2^{-n}) \neq \emptyset \right\} \leq c 2^{-n(d - \eta)}.
\end{align}
\end{lemma}
\begin{proof}
The proof is a similar to that of \cite[Theorem 3.3]{DKN09}, using Theorem \ref{theorem2017-08-17-1}(a) and \eqref{eq2017-08-21-2}; see also \cite[Lemma 2.2]{DKN13}. The details are left to the reader.
\end{proof}
\begin{proof}[Proof of Theorem \ref{theorem2017-08-18-1}: upper bounds]
We start by proving the upper bound on hitting probability in Theorem \ref{theorem2017-08-18-1}(a).
Fix $\epsilon \in \, ]0, 1[$ and $n \in \mathbb{N}$ such that $2^{-n - 1} < \epsilon \leq 2^{-n}$, and write
\begin{align*}
\mbox{P}\left\{u(I \times J) \cap B(z, \epsilon) \neq \emptyset \right\} \leq \sum_{(k, l): R_{k, l}^n \cap I \times J \neq \emptyset}\mbox{P}\left\{u(R_{k, l}^n) \cap B(z, 2^{-n}) \neq \emptyset \right\}.
\end{align*}
The number of pairs $(k, l)$ involved in the sum is at most $2^{2n(\alpha + 1)/(\alpha - 1)}$ times a constant. Lemma \ref{lemma2017-08-18-1} implies that for all $z \in A$, $\eta > 0$ and large $n$,
\begin{align} \label{eq2017-12-15-1}
\mbox{P}\left\{u(I \times J) \cap B(z, \epsilon) \neq \emptyset \right\} &\leq \tilde{C}2^{-n(d - \eta)}2^{\frac{2n(\alpha + 1)}{\alpha - 1}} \leq C \epsilon^{d - \frac{2(\alpha + 1)}{\alpha - 1} - \eta}.
\end{align}
Note that $C$ does not depend on $(n, \epsilon)$. Therefore,  \eqref{eq2017-12-15-1} is valid for all $\epsilon \in \, ]0, 1[$.

Now we use the same covering argument as in the end of the proof of Theorem 1.2(a) in \cite[p.104]{DKN13} to conclude that the upper bound in Theorem \ref{theorem2017-08-18-1}(a) holds.

The proof of the upper bounds on hitting probabilities in Theorem \ref{theorem2017-08-18-1}(b) and (c) is similar; see also \cite[Theorem 3.1(2), (3)]{DKN07}.
\end{proof}

\subsection{Proof of Theorem \ref{theorem2017-08-18-1}: lower bounds} \label{section2.2.3}

The proof is similar to that of \cite[Theorem 2.1]{DKN07}; see also \cite[Section 2.4]{DKN13}, which requires the following lemma analogous to \cite[Lemma 2.3]{DKN13}.
\begin{lemma}\label{lemma2017-11-20-5}
 Fix $N > 0$ and $\beta > 0$.
\begin{itemize}
  \item [(a)]For $p > 4d(\frac{d}{2} - \frac{2}{\alpha - 1} - 1)$, there exists a finite and positive constant $C = C(I, J, d, N, p, \alpha)$ such that for all $a \in [0, N]$,
\begin{align}\label{eq2017-08-23-30}
& \int_Idt\int_Ids\int_Jdx\int_Jdy (\Delta_{\alpha}((t, x); (s, y)))^{-d}\left[\frac{(\Delta_{\alpha}((t, x); (s, y)))^{2}}{a^2} \wedge 1\right]^{p/(4d)}\nonumber \\
& \quad \quad \leq C \,  \mbox{K}_{d - \frac{2(\alpha + 1)}{\alpha - 1}}(a).
\end{align}
\item [(b)] For $p > 4d(\frac{d}{2} - \frac{1}{\beta})$, there exists a finite and positive constant $C = C(I, d, N, p, \beta)$ such that for all $a \in [0, N]$,
\begin{align}\label{eq2017-08-23-301100}
& \int_Idt\int_Ids\, |t - s|^{-\frac{d\beta}{2}}\left[\frac{|t - s|^{\beta}}{a^2} \wedge 1\right]^{p/(4d)}\leq C \,  \mbox{K}_{d - \frac{2}{\beta}}(a).
\end{align}
\end{itemize}
\end{lemma}
\begin{proof} We start by proving (a). Using the change of variables $\tilde{u} = t - s$ ($t$ fixed), $\tilde{v} = x - y$ ($x$ fixed), we see that  the integral on the left-hand side of \eqref{eq2017-08-23-30} is bounded above by
\begin{align*}
4|I||J|\int_0^{|I|}d\tilde{u}\int_0^{|J|}d\tilde{v} \, (\tilde{u}^{\frac{\alpha - 1}{2\alpha}} + \tilde{v}^{\frac{\alpha - 1}{2}})^{-d}\left[\frac{(\tilde{u}^{\frac{\alpha - 1}{2\alpha}} + \tilde{v}^{\frac{\alpha - 1}{2}})^{2}}{a^2} \wedge 1\right]^{p/(4d)}.
\end{align*}
Another change of variables $[\tilde{u} = (ua^2)^{\alpha/(\alpha - 1)}, \tilde{v} = (va^2)^{1/(\alpha - 1)}]$ implies that this is less than
\begin{align*}
Ca^{\frac{2\alpha + 2}{\alpha - 1} - d}\int_0^{|I|^{(\alpha - 1)/\alpha}a^{-2}}du\int_0^{|J|^{\alpha - 1}a^{-2}}dv \, \frac{u^{1/(\alpha - 1)}v^{(2 - \alpha)/(\alpha - 1)}}{(u + v)^{d/2}}[(u + v) \wedge 1]^{p/(4d)}.
\end{align*}
Passing to the polar coordinates, this is bounded above by
\begin{align}\label{eq2017-11-20-6}
Ca^{\frac{2\alpha + 2}{\alpha - 1} - d}(I_1 + I_2(a)),
\end{align}
where
\begin{align*}
I_1 &= \int_0^{\bar{K}N^{-2}}d\rho \, \rho^{\frac{2}{\alpha - 1} - \frac{d}{2}}\rho^{p/(4d)}\quad \mbox{and} \quad
I_2(a) = \int_{\bar{K}N^{-2}}^{\bar{K}a^{-2}}d\rho \, \rho^{\frac{2}{\alpha - 1} - \frac{d}{2}}
\end{align*}
with $\bar{K} = (|I|^{2(\alpha - 1)/\alpha} + |J|^{2(\alpha - 1)})^{1/2}$. Clearly, $I_1 \leq C < \infty$ since $\frac{2}{\alpha - 1} - \frac{d}{2} + \frac{p}{4d} > -1$ by the hypothesis on $p$. Moreover, if $\frac{2}{\alpha - 1} - \frac{d}{2} + 1 \neq 0$, i.e., $\frac{2(\alpha + 1)}{\alpha - 1} \neq d$, then
\begin{align}\label{eq2017-11-20-7}
I_2(a) = \bar{K}^{(\alpha + 1)/(\alpha - 1) - d/2}\frac{a^{d - 2(\alpha + 1)/(\alpha - 1)} - N^{d - 2(\alpha + 1)/(\alpha - 1)}}{(\alpha + 1)/(\alpha - 1) - d/2}.
\end{align}
There are three separate cases to consider. (i) If $\frac{2(\alpha + 1)}{\alpha - 1} < d$, then $I_2(a) \leq C < \infty$ for all $a \in [0, N]$. (ii) If $\frac{2(\alpha + 1)}{\alpha - 1} > d$, then $I_2(a) \leq c a^{d - 2(\alpha + 1)/(\alpha - 1)}$. (iii) If $\frac{2(\alpha + 1)}{\alpha - 1} = d$, then
\begin{align} \label{eq2017-12-13-1}
I_2(a) = 2(\log\frac{1}{a} + \log N) \leq (2\log N  + 2)\log_+(\frac{1}{a}).
\end{align}
We combine these observations to conclude that the expression in \eqref{eq2017-11-20-6} is bounded above by $C\, K_{d - \frac{2(\alpha + 1)}{\alpha - 1}}(a)$.

Next we prove (b). Fix $t$ and change variables $[u = t - s]$ to see that
\begin{align}\label{eq2017-08-23-30110011}
& \int_Idt\int_Ids\, |t - s|^{-\frac{d\beta}{2}}\left[\frac{|t - s|^{\beta}}{a^2} \wedge 1\right]^{p/(4d)}\leq 2\int_0^{|I|}du \, u^{-\frac{d\beta}{2}}\left[\frac{u^{\beta}}{a^2} \wedge 1\right]^{p/(4d)}.
\end{align}
Another change of variables $[u = a^{2/\beta}v^{1/\beta}]$ simplifies this expression to
\begin{align*}
C\, a^{\frac{2}{\beta} - d}\int_0^{|I|^{\beta}a^{-2}}dv \, v^{\frac{1}{\beta} - \frac{d}{2} - 1}\left[v \wedge 1\right]^{p/(4d)}.
\end{align*}
Observe that
\begin{align*}
\int_0^{|I|^{\beta}a^{-2}}dv \, v^{\frac{1}{\beta} - \frac{d}{2} - 1}\left[v \wedge 1\right]^{p/(4d)} \leq I_1 + I_2(a),
\end{align*}
where
\begin{align*}
I_1: = \int_0^{|I|^{\beta}N^{-2}}dv \, v^{\frac{1}{\beta} - \frac{d}{2} - 1 + \frac{p}{4d}} \quad \mbox{and} \quad I_2(a): = \int_{|I|^{\beta}N^{-2}}^{|I|^{\beta}a^{-2}}dv \, v^{\frac{1}{\beta} - \frac{d}{2} - 1}.
\end{align*}
Clearly, $I_1 \leq C < \infty$ provided that $p > 4d(\frac{d}{2} - \frac{1}{\beta})$. The remainder of the proof is the same as that of (a).
\end{proof}

\begin{proof}[Proof of Theorem \ref{theorem2017-08-18-1}: lower bounds]
We start by proving the lower bound on hitting probabilities in (a).
The proof follows along the same lines as the proof of \cite[Theorem 2.1(1)]{DKN07}, therefore we will only sketch the steps that differ; see also the proof of \cite[Theorem 1.2(b)]{DKN13}. We need to replace their $\beta - 6$ by $d - \frac{2(\alpha + 1)}{\alpha - 1}$.

We first note that our Theorems \ref{theorem2017-08-17-1}(a), \ref{th2018-04-26-1} and Remark \ref{remark2017-11-20-1} imply that
\begin{align} \label{eq2017-08-23-4}
\inf_{\|z\| \leq M}\int_Idt\int_Jdx\, p_{t, x}(z) \geq C >0,
\end{align}
which proves hypothesis \textbf{A1'} of \cite[Theorem 2.1(1)]{DKN07} (see \cite[Remark 2.5(a)]{DKN07}).

Let us now follow the proof of \cite[Theorem 2.1(1)]{DKN07}. Define, for all $z \in \mathbb{R}^d$ and $\epsilon > 0$, $\tilde{B}(z, \epsilon) := \{y \in \mathbb{R}^d: |y - z| < \epsilon\}$, where $|z| := \max_{1 \leq j \leq d}|z_j|$, and
\begin{align} \label{eq2017-08-23-5}
J_{\epsilon}(z) = \frac{1}{(2\epsilon)^d}\int_Idt\int_Jdx \, \mathbf{1}_{\tilde{B}(z, \epsilon)}(u(t, x)),
\end{align}
as in \cite[(2.28)]{DKN07}.

Assume first that $d < \frac{2(\alpha + 1)}{\alpha - 1}$. Using Remark \ref{remark2017-11-20-1}(b), we find, instead of \cite[(2.30)]{DKN07},
\begin{align*}
\mbox{E}[(J_{\epsilon}(z))^2] \leq c \int_Idt\int_Ids\int_Jdx\int_Jdy \, [\Delta_{\alpha}((t, x); (s, y))]^{-d}.
\end{align*}
The change of variables $u = t - s$ ($t$ fixed), $v = x - y$ ($x$ fixed), implies that the above integral is bounded above by
\begin{align} \label{eq2017-08-23-6}
 C\int_0^{|I|}du\int_0^{|J|}dv \left(u^{\frac{\alpha - 1}{2\alpha}} + v^{\frac{\alpha - 1}{2}}\right)^{-d} \leq C'\int_0^{|I|}du \, \Psi_{|J|, (\alpha - 1)d/2}(u^{(\alpha - 1)d/(2\alpha)}),
\end{align}
where $\Psi$ is defined by
$
\Psi_{a, \nu}(\rho) := \int_0^a\frac{dx}{\rho + x^{\nu}},
$
for all $a, \nu, \rho > 0$, as in (2.23) of \cite{DKN07}. Hence, by Lemma 2.3 of \cite{DKN07}, for all $\epsilon > 0$,
\begin{align*}
\mbox{E}\left[(J_{\epsilon}(z))^2\right] \leq C \int_0^{|I|}du \, \mbox{K}_{1 - \frac{2}{(\alpha - 1)d}}(u^{(\alpha - 1)d/(2\alpha)}).
\end{align*}
In order to bound the above integral, we consider three different cases: (i) If $0 < d < \frac{2}{\alpha - 1}$, then $1 - \frac{2}{(\alpha - 1)d} < 0$ and the integral equals $|I|$. (ii) If $\frac{2}{\alpha - 1} < d < \frac{2(\alpha + 1)}{\alpha - 1}$, then $\mbox{K}_{1 - \frac{2}{(\alpha - 1)d}}(u^{(\alpha - 1)d/(2\alpha)}) = u^{1/\alpha - (\alpha - 1)d/(2\alpha)}$ and the integral is finite. (iii) If $d = \frac{2}{\alpha - 1}$, then $\mbox{K}_0(u^{1/\alpha}) = \log_+(u^{-1/\alpha})$ and the integral is also finite. The remainder of the proof of Theorem \ref{theorem2017-08-18-1}(a) when $d < \frac{2(\alpha + 1)}{\alpha - 1}$ follows exactly as in \cite[Theorem 2.1(1) Case 1]{DKN07}.

Assume now that $d > \frac{2(\alpha + 1)}{\alpha - 1}$. Define, for all $\mu \in \mathscr{P}(A)$ and $\epsilon > 0$,
\begin{align} \label{eq2017-08-23-7}
J_{\epsilon}(\mu) = \frac{1}{(2\epsilon)^d}\int_{\mathbb{R}^d}\mu(dz)\int_Idt\int_Jdx \, \mathbf{1}_{\tilde{B}(z, \epsilon)}(u(t, x)),
\end{align}
as \cite[(2.35)]{DKN07}. Fix $\mu \in \mathscr{P}(A)$ such that
\begin{align*}
I_{d - \frac{2(\alpha + 1)}{\alpha - 1}}(\mu) \leq \frac{2}{\mbox{Cap}_{d - \frac{2(\alpha + 1)}{\alpha - 1}}(A)}.
\end{align*}
Analogous to the proof of \cite[(2.41)]{DKN07}, we use Remark \ref{remark2017-11-20-1}(b) and Lemma \ref{lemma2017-11-20-5}(a), to see that for all $\epsilon > 0$
\begin{align*}
\mbox{E}\left[(J_{\epsilon}(\mu))^2\right] \leq C_2 I_{d - \frac{2(\alpha + 1)}{\alpha - 1}}(\mu) \leq \frac{2C_2}{\mbox{Cap}_{d - \frac{2(\alpha + 1)}{\alpha - 1}}(A)}.
\end{align*}
The remainder of the proof of Theorem \ref{theorem2017-08-18-1}(a) when $d > \frac{2(\alpha + 1)}{\alpha - 1}$ follows as in \cite[Theorem 2.1(1) Case 2]{DKN07}.

The case $d = \frac{2(\alpha + 1)}{\alpha - 1}$ is proved exactly along the same lines as the proof of \cite[Theorem 2.1(1) Case 3]{DKN07}, appealing to (\ref{eq2017-08-23-4}), Remark \ref{remark2017-11-20-1}(b) and Lemma \ref{lemma2017-11-20-5}(a).

The proof of lower bounds on hitting probabilities in (b) and (c) follows similarly by using Remark \ref{remark2017-11-20-1}(b) and Lemma \ref{lemma2017-11-20-5}(b).
\end{proof}

\subsection{Proof of Theorem \ref{theorem2017-08-18-2}}

In the case $b \equiv 1$ and $\sigma \equiv I_d$, the components of $v = (v_1, \ldots, v_d)$ are independent and identically distributed.
\begin{prop}\label{prop2017-08-21-1}
For any $0 < t_0 < T$, $p > 1$ and $K$ a compact set, there exists $c_1 = c_1(p, t_0, K) > 0$ such that for any $t_0 \leq s \leq t \leq T, x, y \in K,$
\begin{align}\label{eq2017-08-21-4}
\mbox{E}\left[|v_1(t, x) - v_1(s, y)|^p\right] \geq c_1\left(|t - s|^{\frac{\alpha - 1}{\alpha}} + |x - y|^{\alpha - 1}\right)^{p/2}.
\end{align}
\end{prop}
\begin{proof}
The proof is similar to that of Proposition 2.1 of \cite{DKN13}. Since $v$ is Gaussian, it is sufficient to prove (\ref{eq2017-08-21-4}) for $p = 2$. By Ito's isometry, we have
\begin{align}\label{eq2017-08-21-5}
\mbox{E}\left[|v_1(t, x) - v_1(s, y)|^2\right] &= \int_s^t\int_{\mathbb{R}}G_{\alpha}^2(t - r, x - v)dvdr \nonumber \\
&  \quad + \int_0^s\int_{\mathbb{R}}(G_{\alpha}(t - r, x - v) - G_{\alpha}(s - r, y - v))^2dvdr \\
&:= I_1 + I_2 \nonumber.
\end{align}

\textbf{Case 1}: $t - s \geq |x - y|^{\alpha}$. In this case, by the semi-group property and the scaling property of the Green kernel \cite[Lemma 4.1(iii), (iv)]{ChD15}, we have
\begin{align*}
I_1 + I_2 \geq I_1 &= \int_s^tG_{\alpha}(2(t - r), 0)dr = \int_s^t(2(t - r))^{-1/\alpha}G_{\alpha}(1, 0)dr \\
&= c_{\alpha}(t - s)^{\frac{\alpha - 1}{\alpha}} \geq \frac{c_{\alpha}}{2}\left((t - s)^{\frac{\alpha - 1}{\alpha}} + |x - y|^{\alpha - 1}\right).
\end{align*}
\textbf{Case 2}: $t - s \leq |x - y|^{\alpha}$. In this case, by the Plancherel theorem,
\begin{align*}
I_1 + I_2 \geq I_2 &= \int_0^s\int_{\mathbb{R}}(G_{\alpha}(t - r, x -y + v) - G_{\alpha}(s - r, v))^2dvdr \\
&= \frac{1}{2\pi}\int_0^s\int_{\mathbb{R}}\left|e^{-(s - r)|\lambda|^{\alpha}} - e^{-(t - r)|\lambda|^{\alpha}}e^{i\lambda(x - y)}\right|^2d\lambda dr\\
&= \frac{1}{2\pi}\int_0^s\int_{\mathbb{R}}e^{-2(s - r)|\lambda|^{\alpha}}\left|1 - e^{-(t - s)|\lambda|^{\alpha}}e^{i\lambda(x - y)}\right|^2d\lambda dr.
\end{align*}
We use the elementary inequality $|1 - re^{i\theta}| \geq \frac{1}{2}|1 - e^{i\theta}|$, valid for all $r \in [0, 1]$ and $\theta \in \mathbb{R}$, to see that
\begin{align*}
I_2 \geq \int_{\mathbb{R}}\frac{1 - e^{-2s|\lambda|^{\alpha}}}{8\pi |\lambda|^{\alpha}}\left|1 - e^{i\lambda(x - y)}\right|^2d\lambda.
\end{align*}
Because $x - y \in K - K$ and $K$ is compact, fix $C > 0$ such that $|x - y| \leq C$. When $x \neq y$, we change the variable by letting $\xi = |x - y|\lambda$
 and write $e_0 = (x - y)/|x - y|$ to see that the right-hand side of the above inequality is equal to
 \begin{align*}
&|x - y|^{\alpha - 1}\int_{\mathbb{R}}\frac{1 - e^{-2s|\xi|^{\alpha}/|x - y|^{\alpha}}}{8\pi |\xi|^{\alpha}}\left|1 - e^{ie_0\xi}\right|^2d\xi \\
 &\quad \geq |x - y|^{\alpha - 1}\int_{\mathbb{R}}\frac{1 - e^{-2s|\xi|^{\alpha}/C^{\alpha}}}{8\pi |\xi|^{\alpha}}\left|1 - e^{ie_0\xi}\right|^2d\xi\\
& \quad \geq  |x - y|^{\alpha - 1}\int_{\mathbb{R}}\frac{1 - e^{-2t_0|\xi|^{\alpha}/C^{\alpha}}}{8\pi |\xi|^{\alpha}}\left|1 - e^{ie_0\xi}\right|^2d\xi.
\end{align*}
The integral above is a positive constant. Therefore, when $t - s \leq |x - y|^{\alpha}$,
\begin{align*}
\mbox{E}\left[|v_1(t, x) - v_1(s, y)|^2\right] \geq c|x - y|^{\alpha - 1} \geq \frac{c}{2}\left(|t - s|^{\frac{\alpha - 1}{\alpha}} + |x - y|^{\alpha - 1}\right).
\end{align*}
Cases $1$ and $2$ together imply (\ref{eq2017-08-21-4}).
\end{proof}

\begin{proof}[Proof of Theorem \ref{theorem2017-08-18-2}]
As in \cite[Theorem 1.5]{DKN13}, we first apply \cite[Theorem 7.6]{Xia09} to deduce Theorem \ref{theorem2017-08-18-2}(a). For this, it suffices to verify Conditions (C1) and (C2) of \cite[Section 2.4, p.158]{Xia09} with $N = 2, H_1 = \frac{\alpha - 1}{2\alpha}, H_2 = \frac{\alpha - 1}{2}$.

First, we observe that $\mbox{E}[v_1(t, x)^2] = c_{\alpha}t^{\frac{\alpha - 1}{\alpha}}$ (see (\ref{eq2017-09-05-1})), which implies that there are positive constants $c_1, c_2$ such that for all $(t, x), (s, y) \in I \times J$,
\begin{align}\label{eq2017-08-23-1}
c_1 \leq \mbox{E}[v_1(t, x)^2] \leq c_2.
\end{align}
By (\ref{eq2017-08-21-4}) and (\ref{eq2017-08-18-1}), there exist positive constants $c_3, c_4$ such that for all $(t, x), (s, y) \in I \times J$,
\begin{align}\label{eq2017-08-23-2}
c_3(\Delta_{\alpha}((t, x); (s, y)))^{2} \leq \mbox{E}\left[|v_1(t, x) - v_1(s, y)|^2\right] \leq c_4(\Delta_{\alpha}((t, x); (s, y)))^{2}.
\end{align}
Hence condition (C1) is satisfied by (\ref{eq2017-08-23-1}) and (\ref{eq2017-08-23-2}).  Condition (C2) holds by applying the fourth point of Remark 2.2 in \cite{Xia09}, since $(t, x) \mapsto \mbox{E}[v_1(t, x)] = c_{\alpha}t^{\frac{\alpha - 1}{\alpha}}$ is continuous in $I \times J$ with continuous partial derivatives.

The proof of Theorem \ref{theorem2017-08-18-2}(b) and (c) follows the same lines by using \eqref{eq2017-08-23-1}, \eqref{eq2017-08-23-2} and the fact that $(t, x) \mapsto \mbox{E}[v_1(t, x)] = c_{\alpha}t^{\frac{\alpha - 1}{\alpha}}$ is continuous in $I \times J$ with continuous partial derivatives.

Therefore we have finished the proof of Theorem \ref{theorem2017-08-18-2}.
\end{proof}

\begin{appendices}

\section{}

We first recall Burkholder's inequality for Hilbert-space-valued martingales; see also \cite[Eq.(4.18)]{BaP98} and \cite[Lemma 7.6]{DKN09}.
\begin{lemma}[{{\cite[E.2. p.212]{MeM82}}}] \label{lemma2017-09-05-5}
Let $H_{s, y}$ be a predictable $L^2(([0, t] \times \mathbb{R})^m, d\alpha)$-valued process, where $m \geq 1$ and $d\alpha$ denotes Lebesgue measure. Then, for any $p \geq 2$, there exists $C > 0$ such that
\begin{eqnarray*}
\mbox{E}\Bigg[\bigg|\int\limits_{([0, t] \times \mathbb{R})^m}\Big|\int_0^t\int_{\mathbb{R}}H_{s, y}(\alpha)W(ds, dy)\Big|^2d\alpha\bigg|^p\Bigg] \leq C\mbox{E}\Bigg[\bigg|\int_0^t\int_{\mathbb{R}}\Big|\int\limits_{([0, t] \times \mathbb{R})^m}H^2_{s, y}(\alpha)d\alpha\Big|dyds\bigg|^p\Bigg].
\end{eqnarray*}
\end{lemma}

The next result is an extension of Morien \cite[Lemma 4.2]{Mor99} for the solution of the fractional stochastic heat equation \eqref{eq2017-11-17-1}.
\begin{lemma}\label{lemma2017-09-06-1}
Assume \textbf{P1}. For all $p \geq 1, T > 0$ there exists $C > 0$ such that for all $T \geq t \geq s \geq \epsilon > 0$ and $x \in \mathbb{R}$,
\begin{equation*}
\sum_{k, i = 1}^d \mbox{E}\left[\left(\int_{s - \epsilon}^s dr \int_{\mathbb{R}}dv \left|D_{r, v}^{(k)}(u_i(t, x))\right|^2\right)^p\right] \leq C \epsilon^{(\alpha - 1)p/\alpha}.
\end{equation*}
\end{lemma}
\begin{proof}
The proof follows the same lines as \cite[Lemma 4.2]{Mor99}. We include it because the ingredients will be needed for Lemma \ref{lemma2017-09-14-1}. We define
\begin{equation}\label{eq2017-09-06-4}
H_i(t, x):= \mbox{E}\left[\left(\int_{s - \epsilon}^sdr\int_{\mathbb{R}}dv\left|D_{r, v}^{(k)}(u_i(t, x))\right|^2\right)^p\right],
\end{equation}
and
\begin{equation}\label{eq2017-09-06-6}
K_s(t) := \sum_{i = 1}^d \sup_{s \leq \lambda \leq t}\sup_{y \in \mathbb{R}}H_i(\lambda, y)
\end{equation}
which are finite by (\ref{eq2017-09-05-4}).
Thanks to formula (\ref{eq2017-09-05-3}), we have
\begin{align} \label{eq2017-09-06-7}
H_i(t, x) &\leq  c\left(\int_{s - \epsilon}^sdr\int_{\mathbb{R}}dv \, G_{\alpha}^2(t - r, x - v)\right)^p   \nonumber\\
&  \quad + c\sum_{j = 1}^d \mbox{E}\bigg[\Big[\int_{s - \epsilon}^sdr\int_{\mathbb{R}}dv\Big(\int_r^t\int_{\mathbb{R}}G_{\alpha}(t - \theta, x - \eta) D_{r, v}^{(k)}(\sigma_{ij}(u(\theta, \eta)))W^j(d\theta, d\eta)\Big)^2\Big]^p\bigg] \nonumber\\
& \quad + c \, \mbox{E}\bigg[\Big[\int_{s - \epsilon}^sdr\int_{\mathbb{R}}dv\Big(\int_r^t\int_{\mathbb{R}}G_{\alpha}(t - \theta, x - \eta)D_{r, v}^{(k)}(b_i(u(\theta, \eta)))d\theta d\eta\Big)^2\Big]^p\bigg] \nonumber\\
&:= A_1 + A_2 + A_3.
\end{align}
By \eqref{eq2017-09-05-1}, we see that
\begin{align}\label{eq2017-09-06-5}
\int_{s - \epsilon}^sdr\int_{\mathbb{R}}dv \, G_{\alpha}^2(t - r, x - v)  &= c((t - s + \epsilon)^{\frac{\alpha - 1}{\alpha}} - (t - s )^{\frac{\alpha - 1}{\alpha}}) \leq  c'\epsilon^{\frac{\alpha - 1}{\alpha}},
\end{align}
since the function $x \mapsto (x + \epsilon)^{(\alpha - 1)/\alpha} - x^{(\alpha - 1)/\alpha}$ is decreasing on $[0, \infty[$. This implies that
\begin{align}\label{eq2017-09-06-12}
A_1 \leq c_p \epsilon^{(\alpha - 1)p/\alpha}.
\end{align}
Using Burkholder's inequality for Hilbert-space-valued martingales (Lemma \ref{lemma2017-09-05-5}) first, and then the Cauchy-Schwarz inequality together with the fact that the partial derivatives of $\sigma_{ij}$ are bounded, we obtain
\begin{align} \label{eq2017-09-06-13}
A_2 &\leq c\sum_{l = 1}^d \mbox{E}\bigg[\Big[\int_{s - \epsilon}^sd\theta\int_{\mathbb{R}}d\eta\int_{s - \epsilon}^{s \wedge \theta}dr\int_{\mathbb{R}}dv \, G_{\alpha}^2(t - \theta, x - \eta)\left(D_{r, v}^{(k)}(u_l(\theta, \eta))\right)^2\Big]^p\bigg]\nonumber\\
& \quad + c\sum_{l = 1}^d \mbox{E}\bigg[\Big[\int_{s}^td\theta\int_{\mathbb{R}}d\eta\int_{s - \epsilon}^{s \wedge \theta}dr\int_{\mathbb{R}}dv\, G_{\alpha}^2(t - \theta, x - \eta)\left(D_{r, v}^{(k)}(u_l(\theta, \eta))\right)^2\Big]^p\bigg]\nonumber\\
&:= A_{21} + A_{22}.
\end{align}
We now use H\"{o}lder's inequality with respect to the measure $G_{\alpha}^2(t - \theta, x - \eta)d\theta d\eta$ to find that
\begin{align}\label{eq2017-09-06-14}
A_{21} &\leq c\left|\int_{s - \epsilon}^sd\theta\int_{\mathbb{R}}d\eta \, G_{\alpha}^2(t - \theta, x - \eta)\right|^{p}\sup_{(\theta, \eta)\in [0, T]\times \mathbb{R}} \mbox{E}\left[\left(\int_{0}^{T}dr\int_{\mathbb{R}}dv\left(D_{r, v}^{(k)}(u_l(\theta, \eta))\right)^2\right)^p\right]\nonumber\\
&\leq c \, \epsilon^{(\alpha - 1)q/\alpha},
\end{align}
where the last inequality follows from (\ref{eq2017-09-06-5}) and (\ref{eq2017-09-05-4}). Again, applying H\"{o}lder's inequality with respect to the measure $G_{\alpha}^2(t - \theta, x - \eta)d\theta d\eta$, we see that
\begin{align}\label{eq2017-09-06-15}
A_{22} &\leq c\left|\int_{s}^td\theta\int_{\mathbb{R}}d\eta \, G_{\alpha}^2(t - \theta, x - \eta)\right|^{p - 1}  \nonumber\\
& \quad \times\int_{s}^td\theta\int_{\mathbb{R}}d\eta \, G_{\alpha}^2(t - \theta, x - \eta)\sum_{l = 1}^d \mbox{E}\left[\left(\int_{s - \epsilon}^{s}dr\int_{\mathbb{R}}dv\left(D_{r, v}^{(k)}(u_l(\theta, \eta))\right)^2\right)^p\right]\nonumber\\
&\leq c(t - s)^{\frac{\alpha - 1}{\alpha}(p - 1)}\int_{s}^td\theta\int_{\mathbb{R}}d\eta \, G_{\alpha}^2(t - \theta, x - \eta)K_s(\theta)\nonumber\\
&\leq c\int_{s}^t(t - \theta)^{-\frac{1}{\alpha}}K_s(\theta)d\theta.
\end{align}

We handle the third term in (\ref{eq2017-09-06-7}) in a similar way. First, by the Cauchy-Schwarz inequality  with respect to the measure $G_{\alpha}(t - \theta, x - \eta)d\theta d\eta$, we have
\begin{align}\label{eq2017-09-06-8}
 A_3 &\leq c\, \mbox{E}\left[\left[\int_{s - \epsilon}^sdr\int_{\mathbb{R}}dv\int_r^t\int_{\mathbb{R}} \, G_{\alpha}(t - \theta, x - \eta)\sum_{l = 1}^d\left(D_{r, v}^{(k)}(u_l(\theta, \eta))\right)^2d\theta d\eta\right]^p\right] \nonumber\\
&= c\, \mbox{E}\left[\left[\int_{s - \epsilon}^td\theta\int_{\mathbb{R}}d\eta\int_{s - \epsilon}^{s \wedge \theta}dr\int_{\mathbb{R}}dv \, G_{\alpha}(t - \theta, x - \eta)\sum_{l = 1}^d\left(D_{r, v}^{(k)}(u_l(\theta, \eta))\right)^2\right]^p\right] \nonumber\\
&\leq c\, \mbox{E}\left[\left[\int_{s - \epsilon}^sd\theta\int_{\mathbb{R}}d\eta \, G_{\alpha}(t - \theta, x - \eta)\sum_{l = 1}^d\int_{s - \epsilon}^{s \wedge \theta}dr\int_{\mathbb{R}}dv \left(D_{r, v}^{(k)}(u_l(\theta, \eta))\right)^2\right]^p\right] \nonumber\\
& \quad + c\, \mbox{E}\left[\left[\int_{s}^td\theta\int_{\mathbb{R}}d\eta \, G_{\alpha}(t - \theta, x - \eta)\sum_{l = 1}^d\int_{s - \epsilon}^{s}dr\int_{\mathbb{R}}dv \left(D_{r, v}^{(k)}(u_l(\theta, \eta))\right)^2\right]^p\right] \nonumber\\
&:= A_{31} + A_{32}.
\end{align}
By H\"{o}lder's inequality with respect to the measure $G_{\alpha}(t - \theta, x - \eta)d\theta d\eta$,
\begin{align}\label{eq2017-09-06-10}
A_{31} &\leq c\left|\int_{s - \epsilon}^sd\theta\int_{\mathbb{R}}d\eta \, G_{\alpha}(t - \theta, x - \eta)\right|^{p - 1} \nonumber \\
& \quad  \times \int_{s - \epsilon}^sd\theta\int_{\mathbb{R}}d\eta \, G_{\alpha}(t - \theta, x - \eta)\sum_{l = 1}^d \mbox{E}\left[\left(\int_{s - \epsilon}^{s \wedge \theta}dr\int_{\mathbb{R}}dv\left(D_{r, v}^{(k)}(u_l(\theta, \eta))\right)^2\right)^p\right]\nonumber\\
&\leq c\left|\int_{s - \epsilon}^sd\theta\int_{\mathbb{R}}d\eta \, G_{\alpha}(t - \theta, x - \eta)\right|^{p}\sum_{l = 1}^d\sup_{(\theta, \eta) \in [0, T] \times \mathbb{R}} \mbox{E}\left[\left(\int_{0}^{T}dr\int_{\mathbb{R}}dv\left(D_{r, v}^{(k)}(u_l(\theta, \eta))\right)^2\right)^p\right]\nonumber\\
&\leq c\, \epsilon^p \leq c\, \epsilon^{(\alpha - 1)p/\alpha},
\end{align}
where in the third inequality we use \cite[Lemma 4.1(i)]{ChD15} and (\ref{eq2017-09-05-4}). Similarly,
\begin{align}\label{eq2017-09-06-11}
A_{32} &\leq c\left|\int_{s}^td\theta\int_{\mathbb{R}}d\eta \, G_{\alpha}(t - \theta, x - \eta)\right|^{p - 1}  \nonumber\\
&  \quad \times\int_{s}^td\theta\int_{\mathbb{R}}d\eta \, G_{\alpha}(t - \theta, x - \eta)\sum_{l = 1}^d E\left[\left(\int_{s - \epsilon}^{s}dr\int_{\mathbb{R}}dv\left(D_{r, v}^{(k)}(u_l(\theta, \eta))\right)^2\right)^p\right]\nonumber\\
&\leq c\left|\int_{s}^td\theta\int_{\mathbb{R}}d\eta \, G_{\alpha}(t - \theta, x - \eta)\right|^{p - 1}\int_{s}^td\theta\int_{\mathbb{R}}d\eta \, G_{\alpha}(t - \theta, x - \eta)K_s(\theta)\nonumber\\
&\leq c\int_{s}^tK_s(\theta)d\theta.
\end{align}

Finally, we put (\ref{eq2017-09-06-7}) and (\ref{eq2017-09-06-12})--(\ref{eq2017-09-06-11}) together and obtain that
\begin{align*}
K_s(t) &\leq c\, \epsilon^{(\alpha - 1)p/\alpha} + c\int_s^t(1 + (t - \theta)^{-\frac{1}{\alpha}})K_s(\theta)d\theta \\
&\leq c\, \epsilon^{(\alpha - 1)p/\alpha} + \overline{c}\int_s^t(t - \theta)^{-\frac{1}{\alpha}}K_s(\theta)d\theta.
\end{align*}
Define $\overline{K}_s(\lambda) := K_s(\lambda + s)$. From the above inequality we have
\begin{align*}
\overline{K}_s(t - s) &\leq c\, \epsilon^{(\alpha - 1)p/\alpha} + \overline{c}\int_0^{t - s}(t - s - \theta)^{-\frac{1}{\alpha}}\overline{K}_s(\theta)d\theta.
\end{align*}
By Gronwall's lemma \cite[Lemma 15]{Dal99}, we have
\begin{equation*}
K_s(t) = \overline{K}_s(t - s) \leq c\, \epsilon^{(\alpha - 1)p/\alpha}, \quad  \mbox{for all} \, \, \, s \leq t.  \qedhere
\end{equation*}
\end{proof}

The following lemma is an improvement of Lemma \ref{lemma2017-09-06-1}.

\begin{lemma} \label{lemma2017-09-14-1}
Fix $T > 0, c_0 > 1$ and $0 < \gamma_0 < 1$. For all $p \geq 1$ there exists $C > 0$ such that for all $T \geq t \geq s \geq \epsilon > 0$ with $t - s > c_0\epsilon^{\gamma_0}$ and $x \in \mathbb{R}$,
\begin{align*}
\sum_{k, i = 1}^d E\left[\left(\int_{s - \epsilon}^sdr\int_{\mathbb{R}}dv\left(D_{r, v}^{(k)}(u_i(t, x))\right)^2\right)^p\right] \leq C\epsilon^{(1 - \frac{\gamma_0}{\alpha})p}.
\end{align*}
\end{lemma}
\begin{proof}
We  use the same notations as in the proof of Lemma \ref{lemma2017-09-06-1}.
First, under the condition $t - s > c_0\epsilon^{\gamma_0}$, using \eqref{eq2017-09-05-1}, we have
\begin{align} \label{eq2017-09-14-1}
& \int_{s - \epsilon}^sdr\int_{\mathbb{R}}dv \, G_{\alpha}^2(t - r, x - v) = c((t - s + \epsilon)^{\frac{\alpha - 1}{\alpha}} - (t - s )^{\frac{\alpha - 1}{\alpha}})\nonumber\\
&\quad\leq c((c_0\epsilon^{\gamma_0} + \epsilon)^{\frac{\alpha - 1}{\alpha}} - (c_0\epsilon^{\gamma_0} )^{\frac{\alpha - 1}{\alpha}})=c(c_0\epsilon^{\gamma_0} )^{\frac{\alpha - 1}{\alpha}}((1 + c_0^{-1}\epsilon^{1 - \gamma_0})^{\frac{\alpha - 1}{\alpha}} - 1)\nonumber\\
&\quad\leq c(c_0\epsilon^{\gamma_0})^{\frac{\alpha - 1}{\alpha}}c_0^{-1}\epsilon^{1 - \gamma_0}(\alpha - 1)/\alpha=c\, c_0^{-1/\alpha}\epsilon^{1 - \frac{\gamma_0}{\alpha}}(\alpha - 1)/\alpha,
\end{align}
where the first inequality is because the function $x \mapsto (x + \epsilon)^{\frac{\alpha - 1}{\alpha}} - x^{\frac{\alpha - 1}{\alpha}}$ is decreasing on $[c_0\epsilon^{\gamma_0}, \infty[$, and the second inequality is due to $(1 + x)^{\frac{\alpha - 1}{\alpha}} - 1 \leq \frac{\alpha - 1}{\alpha}x$, for all $x \geq 0$. Therefore,
\begin{align}\label{eq2017-09-14-2}
A_1 \leq c\, \epsilon^{(1 - \frac{\gamma_0}{\alpha})p}.
\end{align}
Using (\ref{eq2017-09-14-1}) instead of \eqref{eq2017-09-06-5}, we see that
\begin{align}\label{eq2017-09-14-3}
A_{21} &\leq c\, \epsilon^{(1 - \frac{\gamma_0}{\alpha})p}.
\end{align}
Due to the choice of $\gamma_0$ and by (\ref{eq2017-09-06-10}), we have
\begin{align}\label{eq2017-09-14-4}
A_{31} \leq c\, \epsilon^p \leq c'\, \epsilon^{(1 - \frac{\gamma_0}{\alpha})p}.
\end{align}
The estimates for other terms remain the same as in the proof of Lemma \ref{lemma2017-09-06-1}. Therefore, we have obtained that
\begin{align*}
K_s(t) &\leq c\, \epsilon^{(1 - \frac{\gamma_0}{\alpha})p} + c\int_s^t(1 + (t - \theta)^{-\frac{1}{\alpha}})K_s(\theta)d\theta \\
&\leq c\, \epsilon^{(1 - \frac{\gamma_0}{\alpha})p} + \overline{c}\int_s^t(t - \theta)^{-\frac{1}{\alpha}}K_s(\theta)d\theta.
\end{align*}
Applying Gronwall's lemma (\cite[Lemma 15]{Dal99}), we have
\begin{equation*}
K_s(t) \leq c\, \epsilon^{(1 - \frac{\gamma_0}{\alpha})p}, \quad \mbox{for all} \,\,\, s \leq t. \qedhere
\end{equation*}
\end{proof}

\begin{remark}\label{remark2017-11-21-2}
The result of Lemma \ref{lemma2017-09-14-1} is also valid for the solutions of stochastic heat equations with Neumann or Dirichlet boundary conditions in which case $\alpha = 2$. This is because the Green kernel of heat equation with Neumann or Dirichlet boundary conditions shares similar properties with the Green kernel of heat equation, which enables us to derive the same estimates as in  \eqref{eq2017-09-14-1}, \eqref{eq2017-09-14-2}, \eqref{eq2017-09-14-3} and \eqref{eq2017-09-14-4} for the solutions of stochastic heat equations with Neumann or Dirichlet boundary conditions.
\end{remark}

\end{appendices}

\noindent ACKNOWLEDGEMENT. \, This paper is based on F. Pu's Ph.D. thesis, written under the supervision of R. C. Dalang. We would like to thank Marta Sanz-Sol\'{e} for useful discussions about the contents of this paper.

\begin{small}

\vspace{1.5cm}

\noindent\textbf{Robert C. Dalang} and \textbf{Fei Pu.}
Institut de Math\'ematiques, Ecole Polytechnique
F\'ed\'erale de Lausanne, Station 8, CH-1015 Lausanne,
Switzerland.\\
Emails: \texttt{robert.dalang@epfl.ch} and \texttt{fei.pu@epfl.ch}\\
\end{small}
\end{document}